\documentclass[a4paper,11pt]{article}
\usepackage{LatexDefinitions}

\newcommand{\HK}{\textnormal{HK}}
\newcommand{\cone}{\mathbf{C}}

\newcommand{\coneProj}{\mathbf{P}}
\newcommand{\Lebesgue}{\mc{L}}
\newcommand{\cont}{C}
\newcommand{\contNC}{C_0}

\newcommand{\WCTM}{W_{\tn{C2M}}}
\newcommand{\WMM}{W_{\tn{MM}}}
\newcommand{\WT}{T}
\newcommand{\cWMM}{c_{\tn{W,MM}}}
\newcommand{\cKL}{c_{\tn{HK,KL}}}
\newcommand{\cKLMM}{c_{\tn{MM,KL}}}
\newcommand{\HKCTM}{\HK_{\tn{C2M}}}
\newcommand{\HKMM}{\HK_{\tn{MM}}}
\newcommand{\cMM}{c_{\tn{MM}}}
\newcommand{\cMMHull}{c_{\tn{MM}}^{\ast \ast}}

\newcommand{\QMM}{Q_{\tn{MM}}}
\newcommand{\QMMTilde}{\tilde{Q}_{\tn{MM}}}
\newcommand{\CSet}{S}
\newcommand{\HKT}{\mathbf{T}}

\DeclareMathOperator{\Cos}{Cos}

\title{Barycenters for the Hellinger--Kantorovich distance over $\R^d$}
\author{Gero Friesecke, Daniel Matthes, Bernhard Schmitzer}
\date{\today}

\begin{document}
\maketitle
\begin{abstract}
	We study the barycenter of the Hellinger--Kantorovich metric over non-negative measures on compact, convex subsets of $\R^d$.
	The article establishes existence, uniqueness (under suitable assumptions) and equivalence between a coupled-two-marginal and a  multi-marginal formulation. We analyze the $\HK$ barycenter between Dirac measures in detail, and find that it differs substantially from the Wasserstein barycenter by exhibiting a local `clustering' behaviour, depending on the length scale of the input measures.
	In applications it makes sense to simultaneously consider all choices of this scale, leading to a 1-parameter family of barycenters. We demonstrate the usefulness of this family by analyzing point clouds sampled from a mixture of Gaussians and inferring the number and location of the underlying Gaussians.
\end{abstract}
\section{Introduction}
\subsection{Overview}
\paragraph{Optimal transport.}
The optimal transport problem dates back to the seminal work of Monge \cite{MongeOT1781} and its modern formulation by Kantorovich \cite{KantorovichOT1958}.
Recent years have seen a tremendous development of the corresponding theory and applications.
We refer to the monographs \cite{Villani-TOT2003,Villani-OptimalTransport-09,Santambrogio-OTAM} for detailed introductions to the transport problem, a historical account and applications in the analysis of PDEs, Riemannian geometry and traffic modelling. An exposition of applications in economics can be found in \cite{GalichonOTEconomics}.

Due to the robustness of Wasserstein distances to `positional noise' and quantization errors optimal transport is also becoming a valuable tool in data analysis. For the viability of these applications efficient numerical algorithms are required.
An overview on the computational aspects of optimal transport and numerical applications is given in \cite{PeyreCuturiCompOT}.

\paragraph{Wasserstein barycenter and multi-marginal problems.}
A common problem in geometric data analysis is the computation of an average between various samples. The standard Euclidean average or mean can be generalized to the center of mass on Riemannian manifolds. The Wasserstein barycenter introduced in \cite{WassersteinBarycenter} shows that this notion is also meaningful in Wasserstein spaces.

The Wasserstein barycenter was defined in two different formulations that were then shown to be equivalent. One is an explicit `center of mass'-based formulation where the weighted (squared) distance of the sought-after mean to all reference measures is minimized. For reasons that will soon become apparent we refer to this as the coupled-two-marginal formulation (see Section \ref{sec:W}).
The other is based on a multi-marginal transport problem introduced by \cite{GangboSwiechMultimarginal98}.
A discussion on numerical applications of the Wasserstein barycenter can be found, for instance, in \cite[Chapter 9.2]{PeyreCuturiCompOT}.

More generally, multi-marginal problems were studied since the 1960s in operations research and probability theory \cite{Kell64, Pi68, RachevRueschendorfVolBoth}, and more recently also found application in economics \cite{CarEke2010} and the approximation of the electronic structure of molecules \cite{CFK-DFT-2013,BDePGG-DFT-2012,CFK18}.

\paragraph{Unbalanced transport.}
The standard Wasserstein distance is restricted to the comparison of measures with equal mass which is evidently restrictive in many applications.
A potential remedy is to use the Hellinger--Kantorovich (or Wasserstein--Fisher--Rao) distance instead \cite{KMV-OTFisherRao-2015,ChizatOTFR2015,LieroMielkeSavare-HellingerKantorovich-2015a}, see also \cite{LieroMielkeSavare-HellingerKantorovich-2015b,ChizatDynamicStatic2018}.
It can be interpreted as a Riemannian infimal convolution between the Wasserstein and Hellinger distances and thus provides a trade-off between transport at small length scales and pointwise interpolation at large length scales.
A corresponding numerical algorithm based on entropic regularization is given in \cite{ChizatEntropicNumeric2018}.

\paragraph{Outline and contribution.}
In this article we address the natural question of the barycenter with respect to the Hellinger--Kantorovich distance. To make the analysis somewhat less technical yet cover typical situations in applications, we focus on non-negative measures $\measp(\Omega)$ on a compact, convex subset $\Omega$ of $\R^d$. Notation and technical preliminaries are established in Section \ref{sec:Notation}.
We then give some reminders about the Wasserstein barycenter (Section \ref{sec:W}) and the Hellinger--Kantorovich distance (Section \ref{sec:HK}).
In particular, it was shown in \cite{LieroMielkeSavare-HellingerKantorovich-2015a} that the Hellinger--Kantorovich distance between two non-negative measures $\mu_1, \mu_2 \in \measp(\Omega)$ can be written as a lifted transport problem over the product space $\cone \assign \Omega \times [0,\infty)$ (where intuitively a point $(x,m) \in \cone$ describes a particle with mass $m$ at location $x$) with respect to a suitable transport cost $c : \cone^2 \to \R$, \eqref{eq:HKC}, and projected marginal constraints. More concretely,
\begin{equation}
	\label{eq:IntroHK}
	\HK(\mu_1,\mu_2)^2 = \inf\left\{
		\int_{\cone^2} c(x,r,y,s)\,\diff \gamma((x,r),(y,s)) \,\middle|\,
		\gamma \in \measp(\cone^2),\,
		\coneProj \pi_{i\sharp} \gamma=  \mu_i
		\right\}
\end{equation}
where $\pi_i : \cone^2 \to \cone$ denotes the projection onto the first and second coordinate and $\coneProj$ is informally given by $(\coneProj \nu)(x) \assign \int_{[0,\infty)} m\,\nu(x,m)\,\diff m$.

In Section \ref{sec:HKCTM} we study a coupled-two-marginal formulation of the Hellinger--Kantorovich barycenter, similar to the Riemannian center of mass,
\begin{equation}
	\label{eq:IntroHKCTM}
	\HKCTM(\mu_1,\ldots,\mu_N)^2 \assign
	\inf \left\{ \sum_{i=1}^N \lambda_i\,\HK(\mu_i,\nu)^2 \middle|
		\nu \in \measp(\Omega)
		\right\}
\end{equation}
where $\mu_1,\ldots,\mu_N \in \measp(\Omega)$ are given reference measures and the weights $\lambda_1,\ldots,\lambda_N \in [0,1]$ sum to one. A minimizing $\nu$ in \eqref{eq:IntroHKCTM} is called an $\HK$-barycenter. Intuitively, as $\lambda_i$ increases, $\nu$ moves towards $\mu_i$.
We show existence of the $\HK$-barycenter. In addition we establish uniqueness under the natural assumption that at least one marginal is Lebesgue-absolutely continuous. Our proof is based solely on the coupled-two-marginal formulation and does not involve the multi-marginal formulation, deviating from \cite{WassersteinBarycenter}.
The argument also applies to the standard Wasserstein barycenter, providing an alternative strategy to that of \cite{WassersteinBarycenter} for which it has been used in \cite{BrendanInfMarginal2013,FixedPointWassersteinBarycenters2016}.

In Section \ref{sec:HKMM} we give a multi-marginal formulation of the HK barycenter.
In analogy to the Wasserstein barycenter this involves the definition of a suitable multi-marginal cost function via computation of a `pointwise barycenter' on the space $\cone$,
\begin{equation}
	\label{eq:IntroCMM}
	\cMM(x_1,m_1,\ldots,x_N,m_M) \assign \inf_{(y,s) \in \cone} \sum_{i=1}^N \lambda_i c(x_i,m_i,y,s).
\end{equation}
Similar to the multi-marginal formulation for the Wasserstein barycenter we then use $\cMM$ to extend the two marginal problem \eqref{eq:IntroHK} to multiple marginals and study the problem
\begin{multline}
	\label{eq:IntroHKMM}
	\HKMM(\mu_1,\ldots,\mu_N)^2 \assign \inf\left\{
		\int_{\cone^N} \cMM(x_1,m_1,\ldots,x_N,m_N)\,\diff \gamma((x_1,m_1),\ldots,(x_N,m_N))
		\right. \\
		\left. \vphantom{\int_{\cone^N}}
		\gamma \in \measp(\cone^N),\,
		\coneProj \pi_{i\sharp} \gamma = \mu_i
		\tn{ for } i \in \{1,\ldots,N\}
		\right\}.
\end{multline}
The pointwise barycenter problem \eqref{eq:IntroCMM} is considerably more intricate than for the Wasserstein distance. First, in general the minimization problem in \eqref{eq:IntroCMM} is non-convex. Second, the resulting function $\cMM$ is non-convex in masses and locations (the latter being a shared feature with the Coulomb cost \cite{CFK-DFT-2013,BDePGG-DFT-2012,CFK18}). Third, minimizers in \eqref{eq:IntroCMM} cannot be given explicitly.
We circumvent these obstacles by a detailed analysis of the convexification of $\cMM$ (in the mass arguments) and invoking dual problems for the HK barycenter and therefore eventually establish the equivalence between both formulations \eqref{eq:IntroHKCTM} and \eqref{eq:IntroHKMM} for the HK barycenter.
The section is concluded by an analysis of the one-dimensional case ($d=1$) at small distances where a more explicit form of $\cMM$ can be given.

Section \ref{sec:HKMMDirac} is dedicated to the HK barycenter between Dirac measures. It turns out that the behaviour is encoded in the convexity properties of $\cMM$ (studied in the previous section) and fundamentally different from the Wasserstein barycenter. While the Wasserstein barycenter between Dirac measures is always a Dirac measure this no longer holds for the HK barycenter. Instead, it may consist of several Dirac measures that can be interpreted as a `clustering' of the Diracs in the reference measures (the barycenter may even be a diffuse measure).
We first give some general results about the characterization of the HK barycenter between Dirac measures and then provide exemplary discussions of the cases $N=2$ and $N=3$ as well as some illustrations.

The `clustering' behaviour of the HK barycenter depends on the length scale of the marginal measures. Instead of merely considering a single scale, in data analysis applications it is natural to consider a whole range of scales.
We call this the \emph{HK barycenter tree}, cf.~Figure \ref{fig:HKBarycenterTree}.
It is formally introduced in Section \ref{sec:HKBarycenterTree} and we illustrate its striking similarity to concepts from topological data analysis. A deeper study of this connection appears to be a promising avenue for future research.

Finally, Section \ref{sec:NonExistence} investigates a potential `soft-marginal formulation' of the HK barycenter. Among the several formulations of the HK distance is a particularly elegant `soft-marginal formulation' \cite{LieroMielkeSavare-HellingerKantorovich-2015a}. We show that under reasonable assumptions no analogous formulation can exist for the HK barycenter for $N\geq3$.

\paragraph{Relation to \cite{HKBarycenters2019}.}
After this work was essentially completed we became aware of the related concurrent article \cite{HKBarycenters2019} which studies the HK barycenter on general metric spaces under some suitable assumptions concerning the existence of the `pointwise barycenter'. The authors study a coupled-two-marginal and a multi-marginal formulation of the HK barycenter, their equivalence, as well as existence and uniqueness (over $\R^d$).
In our article the setting is restricted to compact, convex subsets of $\R^d$, and the exposition is rather different and complementary. First, we provide an alternative proof for the uniqueness of the HK barycenter over $\R^d$.
By a detailed study of the convexity properties of the multi-marginal cost we explicitly show existence of a continuous (and thus measurable) `pointwise barycenter map' without invocation of measurable selection theorems.
In addition we provide a detailed analysis of the HK barycenter between Dirac measures based on the relation to the convexity properties of the multi-marginal cost and provide corresponding numerical illustrations.
Finally, we discuss the existence of a multi-marginal soft-marginal formulation for the HK barycenter.

\subsection{Notation and preliminaries}
\label{sec:Notation}
\begin{itemize}
	\item $\R_+ \assign [0,\infty)$, $\R_{++} \assign (0,\infty)$.
	\item $\Omega$ is a compact, convex subset of $\R^d$ with non-empty interior.
	\item For a Polish space $X$, $\cont(X)$ denotes the set of continuous functions from $X$ to $\R$, equipped with the sup-norm; $\contNC(X)$ denotes the subspace of continuous functions that vanish at infinity.
	\item $\meas(X)$ denotes the set of Radon measures on $X$; $\measp(X)$ the set of non-negative Radon measures; and $\prob(X)$ the set of Radon probability measures.
	\item For compact $X$, we identify the dual space of $\cont(X)$ with $\meas(X)$; for non-compact $X$, we identify the dual space of $\contNC(X)$ with $\meas(X)$.
	\item For a product space $X^n$ we denote by $\pi_i$ the canonical projection $(x_1,\ldots,x_n) \mapsto x_i$. For a measurable map $f$, $f_\sharp$ denotes the push-forward operation acting on corresponding measures.
	\item For $\mu \in \meas(X)^n$ (i.e.~$\mu$ may be a vector-valued measure), $|\mu| \in \measp(X)$ denotes the variation measure. For $\nu \in \measp(X)$, $\mu \ll \nu$ denotes absolute continuity of $\mu$ with respect to $\nu$. For $\sigma \in \measp(X)$, $\mu = \RadNik{\mu}{\sigma} \cdot \sigma + \mu^\perp$ denotes the Lebesgue decomposition of $\mu$ with respect to $\sigma$, into the absolutely continuous part where $\RadNik{\mu}{\sigma}$ denotes the corresponding density, and the singular part $\mu^\perp$.
	\item For a function $f:\R^n \to \RCupInf$, we denote by $f^\ast$ its Fenchel--Legendre conjugate, its subdifferential by $\partial f$. For a set $C \subset \R^n$, we denote by $\iota_C(s) \assign 0$ if $s \in C$ and $+\infty$ otherwise, its indicator function.
	Similar notation will be used for functionals on $\cont(X)$ and $\measp(X)$.
	We assume that the reader has a basic familiarity with convex analysis (in finite dimensions). For an introduction we refer, for instance, to \cite{Rockafellar1972Convex}.
	\item $N \in \N$ is an integer $\geq 2$, specifying the number of reference measures of which we want to compute the barycenter.
	\item $(\lambda_1,\ldots,\lambda_N) \in (0,1)^N$ are the respective weights of the reference measures, satisfying $\sum_{i=1}^N \lambda_i = 1$. Reference measures with $\lambda_i=0$ can be removed from the problem, thus we assume $\lambda_i>0$ (and thus $\lambda_i<1$).
\end{itemize}

The results in this article rely heavily on convex duality between positively 1-homogeneous integral functionals on measures and indicator functions on continuous functions which is provided in the following Lemma, which is essentially due to Rockafellar \cite[Theorem 6]{Rockafellar-IntegralConvexFunctionals71} with assumptions slightly simplified due to \cite[Lemma A2]{BoVa1988}. It was used in this form already in \cite[Lemma 2.9]{ChizatDynamicStatic2018}.
\begin{lemma}
\label{lem:IntConjugation}
Let $X$ be a compact metric space and $f : X \times \R^n \to \RCupInf$ a lower-semi\-con\-tinuous function such that for all $x \in X$, $f_x(\cdot)\assign f(x,\cdot)$ is convex, positively 1-homogeneous and proper. Then $f_x = \iota_{Q(x)}^\ast$ for some family of closed convex sets $Q(x) \subset \R^n$.
Then $I_f: \meas(X)^n \to \RCupInf$ and $I_{f^\ast} : \cont(X)^n \to \RCupInf$ defined as
\begin{align*}
	I_f(\mu) \assign \int_X f(x,\RadNik{\mu}{\sigma})\,\diff \sigma
	\qquad \tn{and} \qquad
	I_{f^\ast}(\phi) \assign \begin{cases}
		0 & \tn{if } \phi(x) \in Q(x) \,\forall\, x \in X, \\
		+ \infty & \tn{else,}
		\end{cases}
\end{align*}
form a pair of convex, proper, lower-semicontinuous conjugate functions with respect to the sup-norm topology on $\cont(X)^n$ and the weak$\ast$ topology on $\meas(X)^n$.
In the definition of $I_f$, $\sigma$ is some arbitrary measure in $\measp(X)$ with $\mu \ll \sigma$. By positive 1-homogeneity of $f(x,\cdot)$ the value of the functional does not depend on the choice of $\sigma$.
\end{lemma}

\section{Reminder: Wasserstein barycenter}
\label{sec:W}
We introduce the standard Wasserstein-2 distance on $\Omega$ in the Kantorovich-formulation and its corresponding dual.
\begin{definition}[Wasserstein distance]
\label{def:W}
For $\mu_1,\mu_2 \in \prob(\Omega)$ set
\begin{align}
	\label{eq:WPrimal}
	W(\mu_1,\mu_2)^2 & \assign \inf \left\{ \int_{\Omega^2} |x_1-x_2|^2\,\diff \gamma(x_1,x_2) \middle|
		\gamma \in \measp(\Omega^2),\,\pi_{i\sharp} \gamma=\mu_i\right\}. \\
	\intertext{A dual formulation is given by}
	\label{eq:WDual}
	W(\mu_1,\mu_2)^2 & = \sup \left\{ \sum_{i=1}^N \int_\Omega \psi_i\,\diff \mu_i \middle|
			\psi_1,\psi_2 \in \cont(\Omega),\,\sum_{i=1}^2 \psi_i(x_i)\leq |x_1-x_2|^2
			\tn{ for } x_1,x_2 \in \Omega
			\right\}.
\end{align}
The set $\Pi(\mu_1,\mu_2)\assign\{\gamma \in \measp(\Omega^2)\,|\,\pi_{i\sharp} \gamma=\mu_i\}$ is called the set of couplings or transport plans between $\mu_1$ and $\mu_2$.
\end{definition}
For a comprehensive introduction to Wasserstein spaces we refer to the monographs \cite{Villani-OptimalTransport-09,Santambrogio-OTAM}.
\begin{remark}
\label{rem:WDuality}
We briefly illustrate how duality between \eqref{eq:WPrimal} and \eqref{eq:WDual} can be established via the Fenchel--Rockafellar theorem as this will be instructive for more involved proofs later on.
Let
\begin{align*}
	G & : \cont(\Omega)^2 \to \R, & (\psi_1,\psi_2) & \mapsto -\sum_{i=1}^2 \int_\Omega \psi_i\,\diff \mu_i, \\
	F & : \cont(\Omega^2) \to \RCupInf, & \phi & \mapsto \begin{cases}
		0 & \tn{if } \phi(x_1,x_2) \leq |x_1-x_2|^2 \tn{ for all } x_1,x_2 \in \Omega, \\
		+\infty & \tn{else,}
		\end{cases} \\
	A & : \cont(\Omega)^2 \to \cont(\Omega^2), & (\psi_1,\psi_2) & \mapsto \phi
	\tn{ with } \phi(x_1,x_2) = \sum_{i=1}^2 \psi_i(x_i).
\end{align*}
Note that $F$ and $G$ are convex, $G$ is continuous, at $\phi=(x,y) \mapsto -2=A(x \mapsto -1,y \mapsto -1)$ the function $F$ is finite and continuous and $A$ is a bounded linear map.
With the assistance of Lemma \ref{lem:IntConjugation}, explicit calculations yield
\begin{align*}
	G^\ast & : \meas(\Omega)^2 \to \RCupInf, & (\rho_1,\rho_2) & \mapsto \sum_{i=1}^2 \iota_{\{-\mu_i\}}(\rho_i), \\
	F^\ast & : \meas(\Omega^2) \to \RCupInf, & \gamma & \mapsto \begin{cases}
		\int_{\Omega^2} |x_1-x_2|^2\,\diff \gamma(x_1,x_2) & \tn{if } \gamma \geq 0, \\
		+\infty & \tn{else,}
		\end{cases} \\
	A^\ast & : \meas(\Omega^2) \to \meas(\Omega)^2, & \gamma & \mapsto (\pi_{1\sharp}\gamma,\pi_{2\sharp} \gamma)\,.
\end{align*}
Then, by Fenchel--Rockafellar duality, \eqref{eq:WDual} can be rewritten as
\begin{align*}
	\tn{\eqref{eq:WDual}} & = -\inf\left\{ G(\psi_1,\psi_2) + F(A(\psi_1,\psi_2))
		\middle| (\psi_1,\psi_2) \in \cont(\Omega)^2 \right\} \\
	& = \inf\left\{ G^\ast(-A^\ast \gamma)+F^\ast(\gamma) \middle| \gamma \in \meas(\Omega^2) \right\} \\
	& = \inf \left\{  \sum_{i=1}^2 \iota_{\{\mu_i\}}(\pi_{i\sharp} \gamma)
		+ \int_{\Omega^2} |x_1-x_2|^2\,\diff \gamma(x_1,x_2)
		\middle| \gamma \in \measp(\Omega^2)
		\right\}
	= \tn{\eqref{eq:WPrimal}}.
\end{align*}
\end{remark}

Similar to Euclidean space or Riemannian manifolds one can now wonder what the weighted center of mass of a tuple of probability measures $\mu_1,\ldots,\mu_N$ in $\prob(\Omega)$ with weights $\lambda_1,\ldots,\lambda_N$ ($\lambda_i>0$, $\sum_{i=1}^N \lambda_i=1$, see Section \ref{sec:Notation}) with respect to the squared Wasserstein distance is.
\begin{definition}[Coupled-two-marginal formulation for Wasserstein barycenter]
\label{def:WCTM}
For $\mu_1,\ldots,\mu_N\allowbreak\in \prob(\Omega)$ set
\begin{align}
	\label{eq:WCTM}
	\WCTM(\mu_1,\ldots,\mu_N)^2 & \assign \inf\left\{
		\sum_{i=1}^N \lambda_i \cdot W(\mu_i,\nu)^2 \middle|
		\nu \in \measp(\Omega) \right\}.
\end{align}
\end{definition}
This is a nested optimization problem where one needs to minimize over $\nu \in \measp(\Omega)$ and over each $\gamma \in \measp(\Omega^2)$ within the $W(\mu_i,\nu)^2$ terms. Hence, we refer to this as the coupled-two-marginal formulation, as opposed to the multi-marginal formulation introduced below.
Since $W(\mu_i,\nu)^2=+\infty$ when $\|\mu_i\| \neq \|\nu\|$ (as the feasible set in \eqref{eq:WPrimal} is empty), we need not add the constraint $\nu \in \prob(\Omega)$, as it is enforced automatically.
\begin{proposition}
	Minimizers $\nu$ of \eqref{eq:WCTM} exist. A minimizer is called Wasserstein barycenter of $(\mu_1,\ldots,\mu_N)$ with weights $(\lambda_1,\ldots,\lambda_N)$.
\end{proposition}
A proof can be found in \cite{WassersteinBarycenter} or follows from standard arguments about weak$\ast$ compactness of bounded measures and weak$\ast$ continuity of the Wasserstein distance on compact metric spaces.

Complementarily, the Wasserstein barycenter problem can also be formulated as a multi-marginal transport problem on $\Omega^N$ with a suitable cost function.
\begin{definition}[Multi-marginal formulation for Wasserstein barycenter]
\label{def:WMM}
\begin{align}
	\label{eq:WCMM}
	\cWMM(x_1,\ldots,x_N) & \assign \inf_{y \in \Omega} \sum_{i=1}^N \lambda_i \, |x_i-y|^2
	= \sum_{i=1}^N \lambda_i\, |x_i-\WT(x_1,\ldots,x_N)|^2 = \sum_{i,j=1}^N \frac{\lambda_i\,\lambda_j}{2} |x_i-x_j|^2 \\
	\label{eq:WTMap}
	& \tn{where} \qquad T(x_1,\ldots,x_N) \assign \sum_{i=1}^N \lambda_i\, x_i \
	\intertext{takes the points $(x_1,\ldots,x_N)$ to the unique minimizer $y=T(x_1,\ldots,x_N)$ in the first line.}
	\label{eq:WMM}
	\WMM(\mu_1,\ldots,\mu_N)^2 & \assign \inf\left\{
		\int_{\Omega^N} \cWMM\,\diff \gamma \,\middle|\,
		\gamma \in \measp(\Omega^N), \, \pi_{i\sharp} \gamma= \mu_i \right\}
\end{align}
\end{definition}

\begin{proposition}[Agueh--Carlier \cite{WassersteinBarycenter}]
	\label{prop:WBarycenterEquivalence}
	$\WCTM(\mu_1,\ldots,\mu_N)^2 = \WMM(\mu_1,\ldots,\mu_N)^2$.
	$\nu$ is a minimizer of \eqref{eq:WCTM} if and only if there exists a minimizer $\gamma$ of \eqref{eq:WMM} such that $\WT_\sharp \gamma=\nu$ (with $T$ given by \eqref{eq:WTMap}).
	Consequently, for minimizers $\gamma$ of \eqref{eq:WMM} we will also call $\WT_\sharp \gamma$ a barycenter.
\end{proposition}

This result was shown in \cite{WassersteinBarycenter}. Since we will later re-use some of its central arguments we give a short sketch as illustration.
\begin{proof}[Sketch of proof]
	Let $\gamma$ be a minimizer of \eqref{eq:WMM}, set $\nu \assign \WT_\sharp \gamma$ and $\gamma_i \assign (\pi_i,\WT)_\sharp \gamma$ for $i=1,\ldots,N$, where $(\pi_i,\WT) : \Omega^N \to \Omega^2$, $(x_1,\ldots,x_N) \mapsto (x_i,\WT(x_1,\ldots,x_N))$.
	One finds that $\pi_{1\sharp} \gamma_i= (\pi_1 \circ( \pi_i,\WT))_\sharp \gamma=\pi_{i\sharp} \gamma=\mu_i$ and similarly $\pi_{2\sharp}\gamma_i=\nu$, so that $\gamma_i \in \Pi(\mu_i,\nu)$. Therefore, one finds that
	\begin{align}
		\WMM(\mu_1,\ldots,\mu_N)^2 & = \int_{\Omega^N} \cWMM\,\diff \gamma
		= \int_{\Omega^N} \left[ \sum_{i=1}^N \lambda_i |x_i-\WT(x_1,\ldots,x_N)|^2 \right] \,\diff \gamma(x_1,\ldots,x_N) 
		\nonumber \\
		& = \sum_{i=1}^N \lambda_i \int_{\Omega^N} |\pi_i(x_1,\ldots,x_N)-\WT(x_1,\ldots,x_N)|^2\,\diff \gamma(x_1,\ldots,x_N)
		\nonumber \\
		& = \sum_{i=1}^N \lambda_i \int_{\Omega^2} |x-y|^2\,\diff [(\pi_i,T)_\sharp \gamma](x,y)
		\nonumber \\ 
		& = \sum_{i=1}^N \lambda_i \int_{\Omega^2} |x-y|^2\,\diff \gamma_i(x,y) \geq \WCTM(\mu_1,\ldots,\mu_N)^2\,.
		\label{eq:WMMWCTMEquivProofA}
	\end{align}
	
	Conversely, let now $\hat{\nu}$ be a minimizer of \eqref{eq:WCTM} and let $\hat{\gamma}_i \in \Pi(\mu_i,\hat{\nu})$ be a minimizer of $W(\mu_i,\hat{\nu})^2$ in \eqref{eq:WPrimal} for $i=1,\ldots,N$. Further, let $(\hat{\gamma}_i^y)_{y \in \Omega}$ be the disintegration of $\hat{\gamma}_i$ w.r.t.~its second marginal.
	Introduce now the measure $\hat{\gamma} \in \measp(\Omega^N)$ via
	\begin{align}
		\label{eq:WMMW2TMEquivGluing}
		\int_{\Omega^N} \phi\,\diff \hat{\gamma} & \assign \int_{\Omega^{N+1}} \phi(x_1,\ldots,x_N)\,
			\diff \hat{\gamma}_1^y(x_1)\ldots \diff \hat{\gamma}_N^y(x_N)\,\diff \hat{\nu}(y)
	\end{align}
	for test functions $\phi \in \cont(\Omega^N)$.
	One then finds for $\phi \in \cont(\Omega)$ that
	\begin{align*}
		 \int_\Omega \phi \circ \pi_i \,\diff \hat{\gamma} &
		 = \int_{\Omega^2} \phi(x_i)\,\diff\hat{\gamma}_i^y(x_i)\,\diff \hat{\nu}(y)
		 = \int_{\Omega^2} \phi \circ \pi_1 \,\diff \hat{\gamma}_i = \int_\Omega \phi\,\diff \mu_i
	\end{align*}
	and therefore that $\pi_{i\sharp} \hat{\gamma} = \mu_i$. Consequently,
	\begin{align}
		\WMM(\mu_1,\ldots,\mu_N)^2 & \leq
		\int_{\Omega^N} \cWMM\,\diff \hat{\gamma}
		=\int_{\Omega^{N+1}} \left( \inf_{z \in \Omega} \sum_{i=1}^N \lambda_i |x_i-z|^2 \right)
		\diff \hat{\gamma}_1^y(x_1) \ldots \diff \hat{\gamma}_N^y(x_N)\,\diff \hat{\nu}(y) \nonumber \\
		& \leq \int_{\Omega^{N+1}} \left(\sum_{i=1}^N \lambda_i |x_i-y|^2\right)
		\diff \hat{\gamma}_1^y(x_1) \ldots \diff \hat{\gamma}_N^y(x_N)\,\diff \hat{\nu}(y) \nonumber \\		
		& = \sum_{i=1}^N \lambda_i \int_{\Omega^2} |x_i-y|^2
		\diff \hat{\gamma}_i^y(x_i) \,\diff \hat{\nu}(y)
		= \sum_{i=1}^N \lambda_i \int_{\Omega^2} |x_i-y|^2 \diff \hat{\gamma}_i(x_i,y) \nonumber \\
		& = \sum_{i=1}^N \lambda_i\,W(\mu_i,\hat{\nu})^2 = \WCTM(\mu_1,\ldots,\mu_N)^2.
		\label{eq:WMMWC2MEquivProofB}
	\end{align}
	Combining \eqref{eq:WMMWCTMEquivProofA} and \eqref{eq:WMMWC2MEquivProofB} one finds that $\WMM(\mu_1,\ldots,\mu_N)^2 = \WCTM(\mu_1,\ldots,\mu_N)^2$ and that $\nu$ constructed from $\gamma$ is optimal for $\WCTM(\mu_1,\ldots,\mu_N)^2$ and $\hat{\gamma}$ constructed from $\hat{\nu}$ is optimal for $\WMM(\mu_1,\ldots,\mu_N)^2$.
	
	In addition, by equality of $\WMM(\mu_1,\ldots,\mu_N)^2$ and $\WCTM(\mu_1,\ldots,\mu_N)^2$ the second inequality in \eqref{eq:WMMWC2MEquivProofB} must be an equality and thus one must have that $y$ is a minimizer of $z \mapsto \sum_{i=1}^N \lambda_i\,|x_i-z|^2$ $\diff \hat{\gamma}_1^y(x_1) \ldots \diff \hat{\gamma}_N^y(x_N)\,\diff \hat{\nu}(y)$-almost everywhere, i.e.~$y=T(x_1,\ldots,x_N)$ almost surely. Therefore, one finds
	\begin{align*}
		\int_{\Omega^N} \phi \circ \WT\,\diff \hat{\gamma} = \int_{\Omega} \phi\,\diff \hat{\nu}
	\end{align*}
	for $\phi \in \cont(\Omega)$ and thus $\WT_\sharp \hat{\gamma}=\hat{\nu}$.
\end{proof}

It will also be instructive to study a dual formulation of the coupled-two-marginal formulation for the Wasserstein barycenter, \eqref{eq:WCTM}.
\begin{proposition}[Dual formulation of Wasserstein barycenter]
\label{prop:WCTMDual}
\begin{multline}
	\WCTM(\mu_1,\ldots,\mu_N)^2 = \sup \left\{
		\sum_{i=1}^N \int_\Omega \psi_i\,\diff \mu_i \middle|
		\psi_1,\ldots,\psi_N,\phi_1,\ldots,\phi_N \in \cont(\Omega),
		\right.
	 \\
	\left. \vphantom{\sum_{i=1}^N}
		\psi_i(x)+\phi_i(y) \leq \lambda_i\,|x-y|^2 \,\forall\, i \in \{1,\ldots, N\},\, x,y \in \Omega,\,
		\sum_{i=1}^N \phi_i \geq 0
		\right\}
	\label{eq:WCTMDual}
\end{multline}
\end{proposition}
A very similar dual formulation was established in \cite{WassersteinBarycenter}. Based on Remark \ref{rem:WDuality} we sketch a proof.
\begin{proof}
Let
\begin{align*}
	G & : \cont(\Omega)^{2N} \to \RCupInf, & & (\psi_1,\ldots,\psi_N,\phi_1,\ldots,\phi_N) \mapsto \begin{cases}
		-\sum_{i=1}^N \int_\Omega \psi_i\,\diff \mu_i & \tn{if } \sum_{i=1}^N \phi_i \geq 0, \\
		+ \infty & \tn{else,}
		\end{cases} \\
	F & : \cont(\Omega^2)^N \to \RCupInf, & & (\xi_1,\ldots,\xi_N) \mapsto \begin{cases}
		0 & \tn{if } \xi_i(x,y) \leq \lambda_i\,|x-y|^2 \\
		& \tn{for all } i \in \{1,\ldots,N\},\, x,y \in \Omega, \\
		+\infty & \tn{else,}
		\end{cases} \\
	A & : \cont(\Omega)^{2N} \to \cont(\Omega^2)^N, & & (\psi_1,\ldots,\psi_N,\phi_1,\ldots,\phi_N) \mapsto (\xi_1,\ldots,\xi_N) \\
	& & & \qquad \qquad \tn{with } \xi_i(x,y) = \psi_i(x) + \phi_i(y).
\end{align*}
Then one can write
\begin{multline*}
	\tn{\eqref{eq:WCTMDual}}
	 = - \inf \{ G(\psi_1,\ldots,\psi_N,\phi_1,\ldots,\phi_N) + F(A(\psi_1,\ldots,\psi_N,\phi_1,\ldots,\phi_N)) \,| \\
	\psi_1,\ldots,\psi_N,\phi_1,\ldots,\phi_N \in \cont(\Omega) \}\,.
\end{multline*}
For conjugates of $F$ and $G$ and the adjoint of $A$ one finds:
\begin{align*}
	G^\ast & : \meas(\Omega)^{2N} \to \RCupInf, & & (\rho_1,\ldots,\rho_N,\sigma_1,\ldots,\sigma_N) \mapsto
		\begin{cases}
		0 & \tn{if } \exists\, \nu \in \measp(\Omega) \tn{ s.t. }
			\rho_i=-\mu_i\\
			& \wedge\, \sigma_i=-\nu \tn{ for } i \in\{1,\ldots,N\}, \\
		+ \infty & \tn{else,}
		\end{cases}
		\\
	F^\ast & : \meas(\Omega^2)^N \to \R, & & (\gamma_1,\ldots,\gamma_N) \!\mapsto\! \begin{cases}
		\sum_{i=1}^N \lambda_i\,\int_{\Omega^2} |x-y|^2\,\diff \gamma_i(x,y) & \tn{if } \gamma_i \geq 0 \\
			& \tn{for } i \in \{1,\ldots,N\}, \\
		+\infty & \tn{else,}
		\end{cases} \\
	A^\ast & : \meas(\Omega^2)^N \to \meas(\Omega)^{2N}, & & (\gamma_1,\ldots,\gamma_N) \mapsto	
		(\pi_{1\sharp}\gamma_1,\ldots,\pi_{1\sharp} \gamma_N, \pi_{2\sharp} \gamma_1,\ldots,\pi_{2\sharp} \gamma_N)
\end{align*}
and again by Fenchel--Rockafellar duality that
\begin{align*}
	\tn{\eqref{eq:WCTMDual}} = \inf \{ G^\ast(-A(\gamma_1,\ldots,\gamma_N))
		+ F^\ast(\gamma_1,\ldots,\gamma_N) | \gamma_1,\ldots,\gamma_N \in \meas(\Omega^2) \}
		= \tn{\eqref{eq:WCTM}}.
	& \qedhere
\end{align*}
\end{proof}
\section{Reminder: Hellinger--Kantorovich distance}
\label{sec:HK}
\subsection{Dynamic formulation}
\label{sec:HKDynamic}
Possibly the most intuitive way to define the Hellinger--Kantorovich distance is via modifying the Benamou--Brenier formula for the Wasserstein metric by introducing a source term into the continuity equation and a corresponding penalty into the energy.
\begin{definition}[Dynamic Benamou--Brenier-type formulation \cite{KMV-OTFisherRao-2015,ChizatOTFR2015,LieroMielkeSavare-HellingerKantorovich-2015a}]
For $\mu_1, \mu_2 \in \measp(\Omega)$ the Hellinger--Kantorovich distance between them is given by
\begin{align}
	\label{eq:HKBB}
	\HK(\mu_0,\mu_1)^2 & \assign \inf \left\{
		\int_{[0,1] \times \Omega}
		\left[
		\left(\RadNik{\omega}{\rho}\right)^2 + \tfrac14 \left(\RadNik{\zeta}{\rho}\right)^2
		\right]\diff \rho
		\middle| (\rho,\omega,\zeta) \in \CE(\mu_0,\mu_1)
		\right\}
\end{align}
where $\CE(\mu_0,\mu_1) \subset \meas([0,1] \times \Omega)^{1 \times d \times 1}$ are distributional solutions of the continuity equation with source, satisfying $\omega, \zeta \ll \rho$, and interpolating between $\mu_0$ and $\mu_1$, i.e.~they solve
\begin{align}
	\partial_t \rho + \ddiv \omega = \zeta, \qquad
	\rho(0)=\mu_0, \qquad \rho(1)= \mu_1
\end{align}
in a distributional sense.
\end{definition}
For more details we refer to \cite{KMV-OTFisherRao-2015,ChizatOTFR2015,LieroMielkeSavare-HellingerKantorovich-2015a} where it was shown that $\HK$ is a geodesic distance on $\measp(\Omega)$.
The distance between two Dirac measures is of particular interest.
\begin{proposition}[$\HK$-distance between Dirac measures \cite{ChizatOTFR2015,LieroMielkeSavare-HellingerKantorovich-2015a}]
\label{prop:HKDirac}
\begin{align}
	\HK(\delta_{x_1} \cdot m_1,\delta_{x_2} \cdot m_2)^2 & =
	m_1 + m_2 - 2\sqrt{m_1\,m_2} \Cos(|x_1-x_2|)
\end{align}
where $\Cos(s) \assign \cos(\min\{s, \tfrac{\pi}{2}\})$.
\end{proposition}
This motivates search for equivalent Kantorovich-type formulations of $\HK$. Several such formulations are given in \cite{LieroMielkeSavare-HellingerKantorovich-2015a,ChizatDynamicStatic2018}. We review them in following subsections.

\subsection{Kantorovich-type formulations}
\label{sec:HKStatic}
Motivated by Proposition \ref{prop:HKDirac} we introduce the following function.
\begin{definition}[HK cost function]
\begin{align}
	\label{eq:HKC}
	c(x_1,m_1,x_2,m_2) \assign \begin{cases}
		m_1 + m_2 - 2\sqrt{m_1\,m_2} \Cos(|x_1-x_2|) & \tn{if } m_1,m_2 \geq 0, \\
		+ \infty & \tn{else.}
	\end{cases}
\end{align}
\end{definition}
The extension to negative masses by $+\infty$ allows us to look at the Fenchel--Legendre conjugate of $c$ with respect to the mass arguments $m_1,m_2$ (for fixed positions $x_1,x_2$), which will naturally appear in a dual problem.
We now very briefly introduce a formulation of $\HK$ that was discovered in \cite{LieroMielkeSavare-HellingerKantorovich-2015a} where more details can be found (with slightly different conventions).
\begin{definition}[Cone over $\Omega$]
In the following, let
\begin{align}
	\cone \assign \Omega \times \R_+
\end{align}
where we recall our notation $\R_+ \assign [0,\infty)$.
If one identifies all points $\Omega \times \{0\}$ in $\cone$ then $\cone^2 \ni ((x_1,m_1),(x_2,m_2)) \mapsto \sqrt{c(x_1,m_1,x_2,m_2)}$ is a metric on $\cone$. With this identification $\cone$ becomes a \emph{cone} over $\Omega$. For our purposes the identification is not required and thus we dispense with it.

Further, we introduce the sets
\begin{align*}
	\meast(\cone) & \assign \left\{\gamma \in \meas(\cone) \middle| \int_{\cone} m\,\diff |\gamma|(x,m) < +\infty \right\}, \\
	\meast(\cone^N) & \assign \left\{\gamma \in \meas(\cone^N) \middle| \int_{\cone^N} \left(\sum_{i=1}^N m_i\right)\,\diff |\gamma|((x_1,m_1),\ldots,(x_N,m_N)) < +\infty \right\},
\end{align*}
of measures with bounded moment along the $\R_+$-axis of $\cone$ and the equivalent on $\cone^N$.
By $\measpt(\cone)$ and $\measpt(\cone^N)$ we denote the sets of non-negative measures in $\meast(\cone)$ and $\meast(\cone^N)$.

Finally, let
\begin{align}
	\label{eq:ConeMeasureProj}
	p &: \cone \to \Omega, \quad (x,m) \mapsto x, &
	\coneProj &: \meast(\cone) \to \meas(\Omega), \quad
	\mu \mapsto p_\sharp (m \cdot \mu).
\end{align}
\end{definition}
We can interpret a Dirac measure $\delta_{(x,m)} \in \measp(\cone)$ for $(x,m) \in \cone$ as representing a Dirac mass at $x \in \Omega$ with mass $m \in \R_+$. The projection of $\delta_{(x,m)}$ to a measure on $\Omega$ is obtained by the operator $\coneProj$. One has $\coneProj \delta_{(x,m)}=m \cdot \delta_x$.
More generally, the representation of $m \cdot \delta_x$ by $\delta_{(x,m)}$ is not unique. Let $\rho \in \measp(\R_+)$ with $\int_{\R_+} \tilde{m}\,\diff \rho(\tilde{m})=m$. Then $\coneProj (\delta_x \otimes \rho)=m \cdot \delta_x$ since
\begin{align*}
	\int_\Omega \phi\,\diff\coneProj(\delta_x \otimes \rho)=
	\int_{\Omega \times \R_+} \phi(\tilde{x}) \cdot \tilde{m}\,\diff \delta_x(\tilde{x})\,\diff \rho(\tilde{m})
	= m \cdot \phi(x).
\end{align*}
It is easy to see that any measure $\mu \in \measp(\Omega)$ can be represented (in a non-unique way) by a measure on $\meast(\cone)$ and it was realized in \cite{LieroMielkeSavare-HellingerKantorovich-2015a} that $\HK(\mu_1,\mu_2)^2$ can be formulated as transport problem on $\cone$ between pairs of measures $\sigma_1,\sigma_2 \in \meast(\cone)$ with $\coneProj \sigma_i=\mu_i$ with respect to the cost function $c$.
\begin{proposition}[Cone lifting formulation \cite{LieroMielkeSavare-HellingerKantorovich-2015a}]
\label{prop:HKLifted}
For $\mu_1,\mu_2 \in \measp(\Omega)$ one has
\begin{align}
	\HK(\mu_1,\mu_2)^2 & = \inf\left\{
		\int_{\cone^2} c(x,r,y,s)\,\diff \gamma((x,r),(y,s)) \,\middle|\,
		\gamma \in \measpt(\cone^2),\,
		\coneProj \pi_{i\sharp} \gamma=  \mu_i
		\right\}.
	\label{eq:HKLifted}
\end{align}
Minimizers exist.
\end{proposition}
Since this is a formulation of $\HK$ in terms of a transport plan $\gamma$ (albeit on $\cone^2$, and with somewhat generalized marginal constraints) we refer to this as a Kantorovich-type formulation in analogy to Definition \ref{def:W}.

An alternative Kantorovich-type formulation was given in \cite{ChizatDynamicStatic2018} in terms of two transport plans $\gamma_1, \gamma_2$ (and an auxiliary measure $\gamma$).
\begin{proposition}[Semi-coupling formulation \cite{ChizatDynamicStatic2018}]
\begin{align}
	\HK(\mu_1,\mu_2)^2 & = \inf\left\{
		\int_{\Omega^2}
		c\big(x_1,\RadNik{\gamma_1}{\gamma},x_2,\RadNik{\gamma_2}{\gamma}\big)
		\diff \gamma(x_1,x_2)
		\middle|
		\gamma,\gamma_1,\gamma_2 \in \measp(\Omega^2),\,
		\pi_{i\sharp} \gamma_i = \mu_i, \, \gamma_i \ll \gamma
		\right\}
		\label{eq:HKSemiCoupling}
\end{align}
Minimizers exist.
\end{proposition}
Note that the value of the objective in \eqref{eq:HKSemiCoupling} does not depend on the choice of $\gamma$, as long as $\gamma_i \ll \gamma$, due to the joint positive 1-homogeneity of $c$ in the mass arguments.
While the objective is now non-linear in the transport plans it has the technical advantage of involving only measures on compact spaces which somewhat simplifies our duality arguments.
In general, throughout the article we find it convenient to have multiple formulations of $\HK$ available as it allows to pick the one most suitable for any given proof.

It is intriguing that $\HK$ can also be rewritten as an optimization problem over a single transport plan on $\measp(\Omega^2)$ with a particular linear transport cost (linear as opposed to the more general non-linear cost used in \eqref{eq:HKSemiCoupling}) where the strict marginal constraints $\pi_{i\sharp} \gamma=\mu_i$ are replaced by `soft-marginal' constraints, that penalize the deviation between $\pi_{i\sharp} \gamma$ and $\mu_i$ with the Kullback--Leibler divergence.

\begin{proposition}[Soft-marginal formulation \cite{LieroMielkeSavare-HellingerKantorovich-2015a}]
\label{prop:HKSoftMarginal}
\begin{align}
	\HK(\mu_1,\mu_2)^2 & = \inf\left\{ \int_{\Omega^2} \cKL\,\diff \gamma
		+ \sum_{i=1}^N \KL(\pi_{i\sharp} \gamma|\mu_i) \,\middle|\,
		\gamma \in \measp(\Omega^2)
		\right\}
	\label{eq:HKSoftMarginal}
\end{align}
where
\begin{align}
	\KL(\mu|\nu) & = \begin{cases}
		\int \varphi(\RadNik{\mu}{\nu})\,\diff \nu & \tn{if } \mu,\nu \geq 0, \mu \ll \nu, \\
		+ \infty & \tn{else,}
		\end{cases} \\
	\varphi(s) & = s\,\log(s)-s+1, \\
	\cKL(x_1,x_2) & = \begin{cases}
		-2\log(\cos(|x_1-x_2|)) & \tn{if } |x_1-x_2| < \tfrac{\pi}{2}, \\
		+ \infty & \tn{else.}
	\end{cases}
\end{align}
Minimizers of \eqref{eq:HKSoftMarginal} exist.
\end{proposition}

\begin{corollary}
\label{cor:HKLowerBound}
	$\HK(\mu_1,\mu_2)^2 \geq (\sqrt{\|\mu_1\|}-\sqrt{\|\mu_2\|})^2$.
\end{corollary}
\begin{proof}
	This follows from \eqref{eq:HKSoftMarginal} by using $\cKL\geq 0$ and $\KL(\pi_{i\sharp} \gamma|\mu_i) \geq \varphi(\|\gamma\|/\|\mu_i\|) \cdot \|\mu_i\|$.
\end{proof}

Finally, we complete our collection of formulations with a corresponding dual.
\begin{proposition}[Dual formulation]
\label{prop:HKDual}
\begin{align}
	\HK(\mu_1,\mu_2)^2 & = \sup \left\{
		\sum_{i=1}^N \int_\Omega \psi_i\,\diff \mu_i \,\middle|\,
		\psi_1,\psi_2 \in C(\Omega), \,
		(\psi_1(x_1),\psi_2(x_2)) \in Q(x_1,x_2)\,\forall\,x_1,x_2 \in \Omega
		\right\}
	\label{eq:HKDual}
\end{align}
where the closed convex set $Q(x_1,x_2)$ is characterized by $c^\ast(x_1,\cdot,x_2,\cdot)=\iota_{Q(x_1,x_2)}$ and
\begin{align}
	\label{eq:HKQ}
	Q(x_1,x_2) & = \left\{ (a,b) \in (-\infty,1]^2 \,|\,
		(1-a)\,(1-b) \geq \Cos(|x_1-x_2|)^2 \right\}.
\end{align}
Here $c^\ast(x_1,\cdot,x_2,\cdot)$ denotes the Fenchel--Legendre conjugate of $c$ w.r.t.~the second and fourth arguments for fixed $x_1, x_2 \in \Omega$.
\end{proposition}

\subsection{Some equivalence proofs}
Equivalence of all formulations in Sections \ref{sec:HKDynamic} and \ref{sec:HKStatic} has already been established in \cite{LieroMielkeSavare-HellingerKantorovich-2015a} and \cite{ChizatDynamicStatic2018}. However, analogous to Section \ref{sec:W} we provide here some selected arguments in preparation for later proofs.
To the best of our knowledge the construction given in Remark \ref{rem:HKLiftedSCEquiv} is new and we will reuse it for Theorem \ref{thm:HKMMSemiCoupling}.
A sketch on how to prove the equivalence of \eqref{eq:HKSoftMarginal} can be found in Section \ref{sec:NonExistence}.

\begin{remark}[Equivalence between cone lifting and semi-coupling formulation]
\label{rem:HKLiftedSCEquiv}
\hfill\\
\textbf{Part 1: \eqref{eq:HKSemiCoupling} $\geq$ \eqref{eq:HKLifted}.}
Let $\mu_1,\mu_2 \in \measp(\Omega)$ be given.
Let $\gamma, \gamma_1, \gamma_2$ be minimizers in \eqref{eq:HKSemiCoupling}. Then $\gamma_1,\gamma_2 \ll \gamma$ and we abbreviate $u_i \assign \RadNik{\gamma_i}{\gamma}$. One has $u_1,u_2 \in L^1(\Omega^2,\gamma)$ and in particular they are measurable. Denote by $((\pi_1,u_1),(\pi_2,u_2))$ the map $\Omega^2 \to \cone^2$ that assigns $(x_1,x_2) \mapsto ((x_1,u_1(x_1,x_2)),(x_2,u_2(x_1,x_2)))$.
Set now $\hat{\gamma} \assign ((\pi_1,u_1),(\pi_2,u_2))_\sharp \gamma$. For $\phi \in \cont(\Omega)$ one has
\begin{align*}
	\int_\Omega \phi\,\diff \coneProj \pi_{i\sharp} \hat{\gamma}
	& = \int_{\cone^2} \phi(x_i)\,m_i\,\diff \hat{\gamma}((x_1,m_1),(x_2,m_2))
	= \int_{\Omega^2} \phi(x_i)\,u_i(x_1,x_2)\,\diff \gamma(x_1,x_2) \\
	& = \int_{\Omega^2} \phi(x_i)\,\diff \gamma_i(x_1,x_2)
	= \int_{\Omega} \phi\,\diff \pi_{i\sharp} \gamma_i
	= \int_{\Omega} \phi\,\diff \mu_i	
\end{align*}
and therefore $\coneProj \pi_{i\sharp} \hat{\gamma} = \mu_i$ for $i=1,2$.
Setting $\phi(x)=1$ above we find that $\hat{\gamma} \in \meast(\cone^2)$ and since $\hat{\gamma}$ is the push-forward of a non-negative measure it is non-negative. Therefore, $\hat{\gamma}$ is feasible in \eqref{eq:HKLifted} and we find
\begin{align*}
	\tn{\eqref{eq:HKLifted}} & \leq \int_{\cone^2} c(x_1,m_1,x_2,m_2)\,\diff \hat{\gamma}((x_1,m_1),(x_2,m_2)) \\
	& = \int_{\Omega^2} c(x_1,u_1(x_1,x_2),x_2,u_2(x_1,x_2))\,\diff \gamma(x_1,x_2) = \tn{\eqref{eq:HKSemiCoupling}}.
\end{align*}

\textbf{Part 2: \eqref{eq:HKSemiCoupling} $\leq$ \eqref{eq:HKLifted}.}
Let now $\hat{\gamma}$ be a minimizer of \eqref{eq:HKLifted} and set $\gamma \assign (p,p)_\sharp \hat{\gamma}$, $\gamma_i \assign (p,p)_\sharp m_i \cdot \hat{\gamma}$ for $i=1,2$.
Since $m_i \in L^1(\cone^2,\hat{\gamma})$ by construction we have $\gamma_i \ll \gamma$.
For $\phi \in \cont(\Omega)$ we find
\begin{align*}
	\int_{\Omega^2} \phi(x_i)\,\diff \gamma_i(x_1,x_2) & = \int_{\cone^2} \phi(x_i)\,m_i\,\diff \hat{\gamma}((x_1,m_1),(x_2,m_2))
	= \int_{\Omega} \phi\,\diff (\coneProj \pi_{i\sharp} \hat{\gamma}) = \int_\Omega \phi\, \diff \mu_i
\end{align*}
and thus $(\gamma,\gamma_1,\gamma_2)$ are feasible for \eqref{eq:HKSemiCoupling}.
Let $(\hat{\gamma}^{(x_1,x_2)})_{(x_1,x_2) \in \Omega^2}$ be the disintegration of $\hat{\gamma}$	with respect to its $\Omega^2$-marginal $\gamma$, i.e.~for $\phi \in L^1(\cone^2,\hat{\gamma})$ one has
\begin{align*}
	\int_{\cone^2} \phi\,\diff \hat{\gamma}
	= \int_{\Omega^2} \left[ \int_{\R_+^2} \phi(x_1,m_1,x_2,m_2)\,\diff \hat{\gamma}^{(x_1,x_2)}(m_1,m_2) \right]
		\diff \gamma(x_1,x_2).	
\end{align*}
For $\phi \in \cont(\Omega^2)$ one finds
\begin{align*}
	\int_{\Omega^2} \phi\,\diff \gamma_i
	& = \int_{\cone^2} \phi(x_1,x_2)\,m_i\,\diff \hat{\gamma}(x_1,m_1,x_2,m_2) \\
	& = \int_{\Omega^2} \phi(x_1,x_2) \left[
		\int_{\R_+^2} m_i\,\diff \hat{\gamma}^{(x_1,x_2)}(m_1,m_2) \right] \diff \gamma(x_1,x_2).
\end{align*}
and we conclude that $\gamma$-almost everywhere
\begin{align}
	u_i(x_1,x_2) & \assign \RadNik{\gamma_i}{\gamma}(x_1,x_2) = \int_{\R_+^2} m_i\,\diff \hat{\gamma}^{(x_1,x_2)}(m_1,m_2).
	\label{eq:ProofLiftedSCEquivDensity}
\end{align}
Finally, we find
\begin{align*}
	\tn{\eqref{eq:HKSemiCoupling}} & \leq \int_{\Omega^2} c(x_1,u_1(x_1,x_2),x_2,u_2(x_1,x_2))\,\diff \gamma(x_1,x_2) \\
	& \leq \int_{\Omega^2} \left[ \int_{\R_+^2} c(x_1,m_1,x_2,m_2)\,\diff \hat{\gamma}^{(x_1,x_2)}(m_1,m_2) \right]
		\diff \gamma(x_1,x_2) \\
	& = \int_{\cone^2} c(x_1,m_1,x_2,m_2)\,\diff \hat{\gamma}((x_1,m_1),(x_2,m_2)) = \tn{\eqref{eq:HKLifted}}
\end{align*}
where the second inequality is due to Jensen, \eqref{eq:ProofLiftedSCEquivDensity} and the joint convexity of $c$ in the mass arguments.
\end{remark}
\begin{remark}[Duality]
\label{rem:HKDual}
We now show duality between \eqref{eq:HKSemiCoupling} and \eqref{eq:HKDual}.
Standard arguments can be applied since \eqref{eq:HKSemiCoupling} only involves measures on compact spaces (as opposed to \eqref{eq:HKLifted}).
In analogy to Remark \ref{rem:WDuality} we can reformulate \eqref{eq:HKDual} as
\begin{align*}
	\tn{\eqref{eq:HKDual}} = -\inf \{ G(\psi_1,\psi_2) + F(A(\psi_1,\psi_2)) | (\psi_1,\psi_2) \in \cont(\Omega)^2 \}
\end{align*}
by choosing
\begin{align*}
G & : \cont(\Omega)^2 \to \R, & (\psi_1,\psi_2) & \mapsto -\sum_{i=1}^2 \int_\Omega \psi_i\,\diff \mu_i, \\
F & : \cont(\Omega^2)^2 \to \RCupInf, & (\phi_1,\phi_2) & \mapsto \begin{cases}
	0 & \tn{if } (\phi_1(x_1,x_2),\phi_2(x_1,x_2)) \in Q(x_1,x_2) \\
		& \tn{for all } x_1,x_2 \in \Omega, \\
	+\infty & \tn{else,}
	\end{cases} \\
A & : \cont(\Omega)^2 \to \cont(\Omega^2)^2, & (\psi_1,\psi_2) & \mapsto (\phi_1,\phi_2)
\tn{ with } \phi_i(x_1,x_2) = \psi_i(x_i)\,.
\end{align*}
Note that $F$ and $G$ are convex, $G$ is continuous, at $A(x \mapsto -1,y \mapsto -1)$ the function $F$ is finite and continuous and $A$ is a bounded linear map.
For the conjugates of $F$ and $G$ and the adjoint of $A$ one finds (once more invoking Lemma \ref{lem:IntConjugation}):
\begin{align*}
	G^\ast & : \meas(\Omega)^2 \to \RCupInf, & (\rho_1,\rho_2) & \mapsto \sum_{i=1}^2 \iota_{\{-\mu_i\}}(\rho_i), \\
	F^\ast & : \meas(\Omega^2)^2 \to \RCupInf, & (\gamma_1,\gamma_2) & \mapsto
		\int_{\Omega^2}
		c\big(x_1,\RadNik{\gamma_1}{\gamma},x_2,\RadNik{\gamma_2}{\gamma}\big)
		\diff \gamma(x_1,x_2), \\
	A^\ast & : \meas(\Omega^2)^2 \to \meas(\Omega)^2, & (\gamma_1,\gamma_2) & \mapsto (\pi_{1\sharp}\gamma_1,\pi_{2\sharp} \gamma_2)\,,
\end{align*}
where $\gamma \in \measp(\Omega^2)$ in the expression of $F^\ast$ is some measure with $\gamma_i \ll \gamma$. Due to the joint positive 1-homogeneity of $c$ in the mass arguments the definition of $F^\ast$ does not depend on the choice of $\gamma$, as long as $\gamma_i \ll \gamma$.
Then one finds
\begin{align*}
	\tn{\eqref{eq:HKLifted}} = \inf\left\{ G^\ast(-A^\ast (\gamma_1,\gamma_2))+F^\ast(\gamma_1,\gamma_2) \middle| \gamma_1,\gamma_2 \in \measp(\Omega^2) \right\}
\end{align*}
and thus equivalence of \eqref{eq:HKLifted} and \eqref{eq:HKDual} by Fenchel--Rockafellar duality.
\end{remark}

\section{HK barycenter: coupled-two-marginal formulation}
\label{sec:HKCTM}
\subsection{Definition, existence and uniqueness of HK barycenter}
In analogy to Definition \ref{def:WCTM} in this section we introduce a coupled-two-marginal formulation of the $\HK$ barycenter by minimizing a (weighted) sum of squared distances to the reference measures.
This formulation is suitable for showing existence and uniqueness of the barycenter (under suitable assumptions), as well as for numerical approximation, see \cite{ChizatEntropicNumeric2018}.
\begin{definition}[Coupled-two-marginal formulation of $\HK$ barycenter]
For $\mu_1,\ldots,\allowbreak\mu_N \!\in \!\measp(\Omega)$ set
\begin{align}
	\HKCTM(\mu_1,\ldots,\mu_N)^2 & \assign
	\inf \left\{ \sum_{i=1}^N \lambda_i\,\HK(\mu_i,\nu)^2 \middle|
		\nu \in \measp(\Omega)
		\right\}.
	\label{eq:HKCTM}
\end{align}
\end{definition}
Note that we can use all three definitions, (\ref{eq:HKLifted}, \ref{eq:HKSemiCoupling}, \ref{eq:HKSoftMarginal}) or the corresponding dual formulation \eqref{eq:HKDual} for $\HK(\cdot,\cdot)^2$ in \eqref{eq:HKCTM} and this flexibility is very convenient in proofs.

\begin{proposition}[Existence]
	\label{prop:HKCTMExistence}
	For $\mu_1,\ldots,\mu_N \in \measp(\Omega)$ minimizers $\nu$ in \eqref{eq:HKCTM} exist. We call these minimizers $\HK$-barycenters of the tuple $(\mu_1,\ldots,\mu_N)$ with weights $(\lambda_1,\ldots,\lambda_N)$.
\end{proposition}
\begin{proof}
	Let $(\nu_k)_k$ be a minimizing sequence in \eqref{eq:HKCTM}.
	From the lower bound $\HK(\mu,\nu)^2 \geq (\sqrt{\|\mu\|}-\sqrt{\|\nu\|})^2$ (Corollary \ref{cor:HKLowerBound}) one concludes that the mass of $(\nu_k)_k$ must be uniformly bounded and thus that the sequence must have a weak$\ast$ cluster point $\nu$.
	
	As $\HK$ metrizes weak$\ast$ convergence on $\measp(\Omega)$ \cite[Theorem 7.15]{LieroMielkeSavare-HellingerKantorovich-2015a} it is weak$\ast$ lower-semi\-con\-tinuous and thus so is the objective in \eqref{eq:HKCTM}. Therefore, $\nu$ must be a minimizer.
\end{proof}

\begin{proposition}[Uniqueness]
	When at least one $\mu_i$ is Lebesgue-absolutely continuous, the $\HK$-barycenter is unique.
\end{proposition}
\begin{proof}
	W.l.o.g.~assume $\mu_1 \ll \Lebesgue$.
	Plugging the soft-marginal formulation of $\HK$, \eqref{eq:HKSoftMarginal}, into \eqref{eq:HKCTM} we find that
	\begin{align*}
		\HKCTM(\mu_1,\ldots,\mu_N)^2 & =
		\inf \left\{ \hat{E}(\nu,\gamma_1,\ldots,\gamma_N) \middle|
			 \nu \in \measp(\Omega), \gamma_1,\ldots,\gamma_N \in \measp(\Omega^2)
			\right\}
		\intertext{with}
		\hat{E}(\nu,\gamma_1,\ldots,\gamma_N) & \assign \sum_{i=1}^N \lambda_i \left[ \int_{\Omega^2} \cKL \,\diff \gamma_i + \KL(\pi_{1\sharp} \gamma_i|\mu_i)
			+ \KL(\pi_{2\sharp} \gamma_i|\nu)\right]
	\end{align*}
	Existence of minimizers for $\hat{E}$ follows from combining Propositions \ref{prop:HKSoftMarginal} and \ref{prop:HKCTMExistence}.
	Let $(\hat{\nu}^1,\allowbreak\hat{\gamma}^1_1,\allowbreak\ldots,\allowbreak\hat{\gamma}^1_N)$ and $(\hat{\nu}^2,\allowbreak \hat{\gamma}^2_1,\allowbreak \ldots,\allowbreak\hat{\gamma}^2_N)$ be two minimizers of $\hat{E}$ (in particular $\hat{\nu}^1$ and $\hat{\nu}^2$ are two $\HK$-barycenters).
	Set now
	\begin{align*}
		(\hat{\nu}^3,\hat{\gamma}^3_1,\ldots,\hat{\gamma}^3_N) \assign \tfrac12
			\sum_{j=1}^2 (\hat{\nu}^j,\hat{\gamma}^j_1,\ldots,\hat{\gamma}^j_N).
	\end{align*}
	Since $\KL$ is jointly convex in its two arguments, $\hat{E}$ is jointly convex in all its arguments.
	Therefore, $(\hat{\nu}^3,\hat{\gamma}^3_1,\ldots,\hat{\gamma}^3_N)$ must also minimize $\hat{E}$.
	
	Note that $\hat{\gamma}^j_1$ must be a minimizer for $\HK(\mu_1,\hat{\nu}^j)^2$ in the formulation \eqref{eq:HKSoftMarginal} for $j \in \{1,2,3\}$.
	Therefore, since $\mu_1 \ll \Lebesgue$ by assumption, from \cite[Theorem 6.6]{LieroMielkeSavare-HellingerKantorovich-2015a} it follows that $\hat{\gamma}^j_1$ is concentrated on the graph of a map $t^j : \Omega \to \Omega$ and can thus be written as $\hat{\gamma}^j_1=(\id,t^j)_\sharp \tilde{\gamma}^j$ for some $\tilde{\gamma}^j \in \measp(\Omega)$ and $(\id,t^j)$ takes $x \mapsto (x,t^j(x))$ for $j \in \{1,2,3\}$.
	
	Since we have $\hat{\gamma}^3=\tfrac12(\hat{\gamma}^1 +\hat{\gamma}^2)$ we find that there must be a single map $t : \Omega \to \Omega$ such that $\hat{\gamma}^j=(\id,t)_\sharp \tilde{\gamma}^j$ for $j \in \{1,2,3\}$ with $\tilde{\gamma}^3=\tfrac12 (\tilde{\gamma}^1+\tilde{\gamma}^2)$.
	Let now
	\begin{align*}
		\tilde{E} & : \measp(\Omega)^2 \times \measp(\Omega^2)^{N-1} \to \RCupInf, &
		(\nu,\tilde{\gamma},\gamma_2,\ldots,\gamma_N) \mapsto \hat{E}(\nu,(\id,t)_\sharp \tilde{\gamma},\gamma_2,\ldots,\gamma_N).
	\end{align*}
	We see directly that
	\begin{align*}	
	\tilde{E}(\hat{\nu}^j,\tilde{\gamma}^j,\hat{\gamma}_2^j,\ldots,\hat{\gamma}_N^j)=
		\hat{E}(\hat{\nu}^j,\hat{\gamma}_1^j,\hat{\gamma}_2^j,\ldots,\hat{\gamma}_N^j)
		\qquad \tn{for} \qquad j \in \{1,2,3\}.
	\end{align*}
	More explicitly one finds that
	\begin{multline*}
		\tilde{E}(\nu,\tilde{\gamma},\gamma_2,\ldots,\gamma_N) =
			\lambda_1 \left[ \int_\Omega \cKL(x,t(x))\,\diff \tilde{\gamma}(x) + \KL(\tilde{\gamma}|\mu_1) + \KL(t_\sharp \tilde{\gamma}|\nu) \right] \\
			+ \sum_{i=2}^N \lambda_i \left[ \int_{\Omega^2} \cKL \,\diff \gamma_i + \KL(\pi_{1\sharp} \gamma_i|\mu_i)
			+ \KL(\pi_{2\sharp} \gamma_i|\nu)\right].
	\end{multline*}
	From convexity of $\tilde{E}$, strict convexity of $\tilde{\gamma} \mapsto \KL(\tilde{\gamma}|\mu_1)$ and from equality of $\tilde{E}(\hat{\nu}^j,\tilde{\gamma}^j,\hat{\gamma}_2^j,\ldots,\hat{\gamma}_N^j)$ for $j\in \{1,2,3\}$ we find that all $\tilde{\gamma}^j$ must be equal. Since $\nu \mapsto \KL(t_\sharp \tilde{\gamma}^j|\nu)$ is strictly convex, again equality of $\tilde{E}$ for all three candidates then implies that all $\hat{\nu}^j$ must agree and that thus the barycenter is unique.
\end{proof}
\begin{remark}
	The same proof strategy also applies to the standard Wasserstein barycenter, and provides an alternative to the proof of uniqueness given in \cite{WassersteinBarycenter}.
	In that case the equivalent of $\hat{E}$ is given by
	\begin{align*}
		\hat{E}(\nu,\gamma_1,\ldots,\gamma_N) \assign
			\sum_{i=1}^N \lambda_i \left[ \int_{\Omega}^2 |x-y|^2\,\diff \gamma_i(x,y) + \iota_{\{\mu_i\}}(\pi_{1\sharp} \gamma_i)
				+ \iota_{\{\nu\}}(\pi_{2\sharp}\gamma_i) \right].
	\end{align*}
	By existence of a transport map one then obtains, as above, that any two optimal $\gamma_1$ are concentrated on the graph of the same map, $\hat{\gamma}_1=(\id,t)_\sharp \tilde{\gamma}$. The term $\iota_{\{\mu_1\}}(\tilde{\gamma})$ that then appears in the equivalent of $\tilde{E}$ then ensures that the optimal $\tilde{\gamma}$ is unique and the term $\iota_{\{\nu\}}(t_\sharp \tilde{\gamma})$ then ensures uniqueness of $\nu$.
	For the uniqueness of the Wasserstein barycenter this argument has essentially been used in \cite[Section 3.2]{BrendanInfMarginal2013} and \cite[Section 2]{FixedPointWassersteinBarycenters2016}.
\end{remark}

\subsection{Dual formulation}
Once more, a dual formulation can be given. It will be particularly useful in establishing the equivalence between the coupled-two-marginal and multi-marginal formulations of the $\HK$ barycenter.
\begin{theorem}
	\begin{align}
	\HKCTM(\mu_1,\ldots,\mu_N)^2 & = \sup \left\{
		\sum_{i=1}^N \int_\Omega \psi_i\,\diff \mu_i \middle| \psi_i,\phi_i \in C(\Omega),\,
		\left(\tfrac{\psi_i(x)}{\lambda_i},\tfrac{\phi_i(y)}{\lambda_i}\right) \in Q(x,y) \right. \nonumber \\
		& \left. \vphantom{\sum_{i=1}^N \int_\Omega} \hphantom{=\sup}
		\forall\,i=1,\ldots,N,\, x,y \in \Omega,\,
		\sum_{i=1}^N \phi_i \geq 0 \right\}
	\label{eq:HKCTMLiftedDual}
	\end{align}
\end{theorem}
\begin{proof}
	The proof quickly follows from combining the ideas of Proposition \ref{prop:WCTMDual} and Remark \ref{rem:HKDual}. We merely give the functions $F$ and $G$ and the operator $A$. The rest follows as before. $G: \cont(\Omega)^{2N} \to \RCupInf$ is chosen as in Proposition \ref{prop:WCTMDual}. For $F$ and $A$ we pick:
\begin{align*}
	F & : \cont(\Omega^2)^{2N} \to \RCupInf, & & (\xi_1,\ldots,\xi_N,\zeta_1,\ldots,\zeta_N) \!\mapsto\! \begin{cases}
		0 & \tn{\!\!\!\!if } \left(\tfrac{\xi_i(x_1,x_2)}{\lambda_i},\tfrac{\zeta_i(x_1,x_2)}{\lambda_i}\right) \in Q(x_1,x_2) \\
		& \tn{\!\!\!\!for all } i \in \{1,\ldots,N\},\, x_1,x_2 \in \Omega, \\
		+\infty & \tn{\!\!\!\!else,}
		\end{cases} \\
	A & : \cont(\Omega)^{2N} \to \cont(\Omega^2)^{2N}, & & (\psi_1,\ldots,\psi_N,\phi_1,\ldots,\phi_N) \mapsto (\xi_1,\ldots,\xi_N,\zeta_1,\ldots,\zeta_N) \\
	& & & \qquad \qquad \tn{with } \xi_i(x_1,x_2) = \psi_i(x_1), \, \zeta_i(x_1,x_2) = \phi_i(x_2).
\end{align*}
This yields
\begin{multline*}
	\tn{\eqref{eq:HKCTMLiftedDual}} = - \inf \{ G(\psi_1,\ldots,\psi_N,\phi_1,\ldots,\phi_N) + F(A(\psi_1,\ldots,\psi_N,\phi_1,\ldots,\phi_N)) \,| \\
	\psi_1,\ldots,\psi_N,\phi_1,\ldots,\phi_N \in \cont(\Omega) \}\,.
\end{multline*}
Conjugation of $F$ and $G$ and determining the adjoint of $A$ follows analogously to Proposition \ref{prop:WCTMDual} and Remark \ref{rem:HKDual} and yields, via Fenchel--Rockafellar duality and \eqref{eq:HKSemiCoupling}, equivalence with \eqref{eq:HKCTM}.
\end{proof}
\section{HK barycenter: multi-marginal formulation}
\label{sec:HKMM}
In this section we introduce a multi-marginal formulation of the $\HK$ barycenter problem.
We study the corresponding multi-marginal cost function and show equivalence to the coupled-two-marginal formulation.
\subsection{Cone lifting and semi-coupling multi-marginal formulation}
We now define the analogon for the $\HK$ barycenter to the multi-marginal cost function for the Wasserstein barycenter, \eqref{eq:WCMM}.
\begin{definition}[Multi-marginal cost function]
\label{def:CMM}
For $(x_1,\ldots,x_N) \in \Omega^N$, $(m_1,\ldots,m_N) \in \R^N$ set
\begin{align}
	\cMM(x_1,m_1,\ldots,x_N,m_N) & \assign \inf_{y \in \Omega, s \geq 0} \sum_{i=1}^N \lambda_i\,c(x_i,m_i,y,s) .
	\label{eq:CMM}
\end{align}
As soon as some $m_i<0$ one finds $\cMM(x_1,m_1,\ldots,x_N,m_N)=\infty$.
\end{definition}
Unlike in the Wasserstein case \eqref{eq:WCMM}, a minimizer $y$ can in general not be given explicitly (see Section \ref{sec:CMMExplicit} for an exceptional special case).
Further, due to the minimization over $y \in \Omega$, $\cMM$ is in general not a convex function in the $m_1,\ldots,m_N$ for fixed $x_1,\ldots,x_N$. It turns out that this non-convexity and the corresponding convex hull of the function play a crucial role in understanding the $\HK$ barycenter and this becomes particular explicit in the study of barycenters between Dirac measures, see Section \ref{sec:HKMMDirac}. Therefore, we now extend the above definition.
\begin{definition}
\label{def:CMMConvex}
We introduce the convex conjugate of $\cMM$ with respect to the mass arguments:
\begin{align}
	\cMM^\ast(x_1,\psi_1,\ldots,x_N,\psi_N) & \assign \sup \left\{ \sum_{i=1}^N \psi_i \cdot m_i - \cMM(x_1,m_1,\ldots,x_N,m_N)
		\middle| m_1,\ldots,m_N \in \R
		\right\}
\end{align}
Due to positive 1-homogeneity of $\cMM$ in the mass arguments (which is inherited from the positive 1-homogeneity of $c$ in the mass arguments) one has that $\cMM^\ast(x_1,\cdot,\ldots,x_N,\cdot)$ can be written as the indicator function of a closed convex set which we denote by $\QMM(x_1,\ldots,x_N)$:
\begin{align}
	\cMM^\ast(x_1,\psi_1,\ldots,x_N,\psi_N) & = \iota_{\QMM(x_1,\ldots,x_N)}(\psi_1,\ldots,\psi_N)
\end{align}
A more explicit form of $\QMM$ will be given in Proposition \ref{prop:QMM}. The convex hull of $\cMM$ with respect to the mass arguments is given by
\begin{align}
	\cMMHull(x_1,m_1,\ldots,x_N,m_N) & \assign \sup \left\{ \sum_{i=1}^N \psi_i \cdot m_i \middle|
		(\psi_1,\ldots,\psi_N) \in \QMM(x_1,\ldots,x_N) \right\}.
\end{align}
And finally we introduce the set where $\cMM$ and its convex hull coincide (and a variant with at least one $m_i>0$),
\begin{align}
	\CSet_0 & \assign \left\{((x_1,m_1),\ldots,(x_N,m_N)) \in \cone^N \middle|
		\cMM(x_1,m_1,\ldots,x_N,m_N) = \cMMHull(x_1,m_1,\ldots,x_N,m_N) \right\}, \nonumber \\
	\CSet & \assign \left\{((x_1,m_1),\ldots,(x_N,m_N)) \in \CSet_0 \middle|
		(m_1,\ldots,m_N) \neq 0 \right\}.
		\label{eq:CSet}
\end{align}
\end{definition}

\begin{remark}
\label{rem:ShortNotation}
For brevity we will often write $\vec{x}=(x_1,\ldots,x_N)$, $\vec{m}=(m_1,\ldots,m_N)$, $\vec{\psi}=(\psi_1,\ldots,\psi_N)$ and use the notations
\begin{align*}
	\cMM(\vec{x},\vec{m}) & \assign \cMM(x_1,m_1,\ldots,x_N,m_N), \\
	\cMM^\ast(\vec{x},\vec{\psi}) & \assign \cMM^\ast(x_1,\psi_1,\ldots,x_N,\psi_N), \\
	\cMMHull(\vec{x},\vec{m}) & \assign \cMMHull(x_1,m_1,\ldots,x_N,m_N), \\
	\QMM(\vec{x}) & \assign \QMM(x_1,\ldots,x_N),
\end{align*}
and similar notations for related objects.
\end{remark}

Based on these definitions we can now introduce a multi-marginal $\HK$ barycenter problem by combining the ideas from Definition \ref{def:WMM} and Proposition \ref{prop:HKLifted}: a multi-marginal transport problem on the cone $\cone$ with suitable `projection marginal constraints' and the cost function $\cMMHull$.
\begin{definition}[Multi-marginal barycenter formulation]
\begin{multline}
	\HKMM(\mu_1,\ldots,\mu_N)^2 \assign \inf\left\{
		\int_{\cone^N} \cMMHull(x_1,m_1,\ldots,x_N,m_N)\,\diff \gamma((x_1,m_1),\ldots,(x_N,m_N))
		\right. \\
		\left. \vphantom{\int_{\cone^N}}
		\gamma \in \measpt(\cone^N),\,
		\coneProj \pi_{i\sharp} \gamma = \mu_i
		\tn{ for } i \in \{1,\ldots,N\}
		\right\}
	\label{eq:HKMMLifted}
\end{multline}
\end{definition}
In the above definition we have used the convex relaxation $\cMMHull$ instead of $\cMM$ itself as it will be technically more convenient in the following. However, we will show in Theorem \ref{thm:HKCTMMMEquiv} that both choices yield the same functional and the above definition is equivalent to formula \eqref{eq:IntroHKMM}.

In analogy to \eqref{eq:HKSemiCoupling} we also introduce a semi-coupling multi-marginal formulation of the $\HK$ barycenter.
As before, it is defined in terms of measures over compact spaces.
\begin{theorem}[Semi-coupling multi-marginal barycenter formulation]
\label{thm:HKMMSemiCoupling}
\begin{multline}
	\HKMM(\mu_1,\ldots,\mu_N)^2 = \inf\left\{
		\int_{\Omega^N}
		\cMMHull\big(x_1,\RadNik{\gamma_1}{\gamma},\ldots,x_N,\RadNik{\gamma_N}{\gamma}\big)
		\diff \gamma(x_1,\ldots,x_N)
		\right| \\
		\left. \vphantom{\int_{\Omega^N}}
		\gamma,\gamma_1,\ldots,\gamma_N \in \measp(\Omega^N),\,
		\pi_{i\sharp} \gamma_i = \mu_i, \, \gamma_i \ll \gamma
		\tn{ for } i \in \{1,\ldots,N\}
		\right\}
		\label{eq:HKMMSemiCoupling}
\end{multline}
\end{theorem}
\begin{proof}
	The proof works in complete analogy to Remark \ref{rem:HKLiftedSCEquiv}.
\end{proof}

\subsection[Some properties of c\_MM and Q\_MM]{Some properties of $\cMM$ and $\QMM$}
Before we are able to make more detailed statements about the multi-marginal formulation of $\HK$ and its relation to the coupled-two-marginal formulation we need to have a closer look at $\cMM$ and $\QMM$.
We start with some fundamental observations about $\cMM$.
\begin{lemma}\hfill
\label{lem:CMMMinimizerExist}
\begin{enumerate}[(i)]
	\item For given $(x_1,m_1),\ldots,(x_N,m_N) \in \cone$ minimizers $(y,s) \in \cone$ in the definition of $\cMM$, \eqref{eq:CMM}, exist.
	\label{item:CMMMinimizerExist}
	\item $\cMM$ is continuous on $\cone^N$. For $\vec{m} \notin \R_+^N$ one has $\cMM=\infty$, thus, $\cMM$ is lower-semicontinuous on $(\Omega \times \R)^N$.
	\label{item:CMMCont}
	\item The minimization in \eqref{eq:CMM} can be restricted to $(y,s) \in H \times \R_+$ where $H$ denotes the convex hull of the points $x_1,\ldots,x_N$ in $\Omega$.
	\label{item:CMMConvexHull}
\end{enumerate}
\end{lemma}
\begin{proof}
\textbf{\eqref{item:CMMMinimizerExist}:} Let $(\iterl{y},\iterl{s})_\ell$ be a minimizing sequence in $\cone$.
Since $c(x_i,m_i,\iterl{y},\iterl{s}) \geq (\sqrt{m_i}-\sqrt{\iterl{s}})^2$, $(\iterl{s})_\ell$ must be bounded and therefore, by compactness of $\Omega$, $(\iterl{y},\iterl{s})_\ell$ must have a cluster point.
By continuity of $c$ any cluster point must be a minimizer.

\textbf{\eqref{item:CMMCont}:} $\cMM$ is upper-semicontinuous as it is the infimum over a family of continuous functions.
Let $(\iterl{\vec{x}},\iterl{\vec{m}})_\ell$ be a sequence in $\cone^N$ that converges to $(\vec{x},\vec{m}) \in \cone^N$ and let $(\iterl{y},\iterl{s})_\ell$ be a corresponding sequence of minimizers. Since $(\iterl{\vec{m}})_\ell$ is bounded $(\iterl{s})_\ell$ must be bounded (arguing as above) and thus by compactness of $\Omega$, $(\iterl{y},\iterl{s})_\ell$ must have some cluster point $(y,s)$. Up to selection of a suitable subsequence one then has by continuity of $c$,
\begin{align*}
	\cMM(\iterl{\vec{x}},\iterl{\vec{m}}) & = \sum_{i=1}^N \lambda_i \, c(\iterl{x}_{i},\iterl{m}_{i},\iterl{y},\iterl{s})
	\xrightarrow{\ell \to \infty} \sum_{i=1}^N \lambda_i \, c(x_{i},m_{i},y,s) \geq \cMM(\vec{x},\vec{m})
\end{align*}
and thus $\cMM$ is continuous on $\cone^N$.

\textbf{\eqref{item:CMMConvexHull}:}
Assume $y$ is a minimizer with $y \notin H$ and let $y'$ be the projection of $y$ onto $H$ (which exists, since $H$ is a closed, convex, non-empty set). Then $|x_i-y'|<|x_i-y|$ for all $i \in \{1,\ldots,N\}$ and thus, by monotonicity of $\Cos(\cdot)$ the objective for $y'$ is potentially better.
\end{proof}

As discussed below the introduction of $\cMM$, Definition \ref{def:CMM}, a minimizer $y$ can in general not be given explicitly and it causes issues with the convexity of $\cMM$ in the mass arguments. Therefore, often it is useful to not minimize over $y$ but only over $s$, for instance, when the minimizer $y$ is assumed to be known.
\begin{lemma}
\label{lem:CMMAux}
For $\vec{x} \in \Omega^N$, $\vec{m} \in \R^N$ and $y \in \Omega$ let
\begin{align}
	\label{eq:CMMAux}
	\cMM(\vec{x},\vec{m},y) \assign \inf_{s \in \R_+} \sum_{i=1}^N \lambda_i\,c(x_i,m_i,y,s)
	= \sum_{i=1}^N \lambda_i\,m_i - \left( \sum_{i=1}^N \lambda_i\,\sqrt{m_i}\,\Cos(|x_i-y|)\right)^2
\end{align}
with the second equality holding for $\vec{m} \in \R_+^N$.
For all $\vec{x} \in \Omega^N$, $y \in \Omega$, the function $\R^N \ni \vec{m} \mapsto \cMM(\vec{x},\vec{m},y)$ is convex and lower-semi\-continuous. There is a family $(\vec{x},z) \mapsto \QMM(\vec{x},z) \allowbreak \subset \allowbreak \R^N$ of closed, convex sets such that for $\vec{\psi} \in \R^N$
\begin{align}
	\label{eq:CMMAuxConjugate}
	\cMM(\vec{x},\vec{m},y) & = \iota_{\QMM(\vec{x},y)}^\ast(\vec{m}), &
	\cMM^\ast(\vec{x},\vec{\psi},y) & = \iota_{\QMM(\vec{x},y)}(\vec{\psi}),
\end{align}
where $\cMM^\ast(\vec{x},\cdot,y)$ is the conjugate of $\cMM(\vec{x},\cdot,y)$ with respect to the mass arguments.
More explicitly, one finds
\begin{align}
	\label{eq:QMMAux}
	\QMM(\vec{x},y) & = \left\{ \vec{\psi} \in \R^N \middle|
		\psi_i \leq \lambda_i \tn{ for } i\in\{1,\ldots,N\},\,
		\sum_{i=1}^N \frac{\lambda_i\,\Cos(|x_i-y|)^2}{1-\psi_i/\lambda_i} \leq 1 \right\}
\end{align}
where for $\psi_i=\lambda_i$ we adopt the convention
\begin{align}
	\label{eq:QMMAuxConvention}
	\frac{\lambda_i\,\Cos(|x_i-y|)^2}{1-\psi_i/\lambda_i} = \begin{cases}
		0 & \tn{if } \Cos(|x_i-y|)=0, \\
		+ \infty & \tn{else.}
		\end{cases}
\end{align}
Further, one has
\begin{align}
	\label{eq:MMAuxMMRelation}
	\cMM(\vec{x},\vec{m}) & = \inf_{y \in \Omega} \cMM(\vec{x},\vec{m},y), &
	\QMM(\vec{x}) & = \bigcap_{y \in \Omega} \QMM(\vec{x},y).
\end{align}
\end{lemma}
\begin{proof}
The explicit form of $\cMM(\vec{x},\vec{m},y)$ in \eqref{eq:CMMAux} follows from direct maximization over $s \in \R_+$.
For $\vec{m} \notin \R_+^N$, $\cMM(\vec{x},\vec{m},y)=\infty$. For $\vec{m} \in \R_+^N$, from the explicit form we deduce immediately that $\vec{m} \mapsto \cMM(\vec{x},\vec{m},y)$ is continuous, convex and positively 1-homogeneous.
Thus, on $\R^N$ the function is lower-semicontinuous, convex and positively 1-homogeneous.
So it is the conjugate of the indicator function of a closed, convex set and it coincides with its biconjugate.
So it satisfies relations \eqref{eq:CMMAuxConjugate} for some closed, convex $\QMM(\vec{x},y)$.

The explicit form of $\QMM(\vec{x},y)$ can be obtained by direct computations: Let $\vec{\psi} \in \R^N$. Then
\begin{align}
	\iota_{\QMM(\vec{x},y)}(\vec{\psi}) & = \sup_{\vec{m} \in \R_+^N} \sum_{i=1}^N \psi_i \cdot m_i - \inf_{s \in \R_+}
		\sum_{i=1}^N \lambda_i\,c(x_i,m_i,y,s)
		\nonumber \\
	& = \sup_{m_1,\ldots,m_N,s \in \R_+}
		\sum_{i=1}^N \left[ \psi_i \cdot m_i - \lambda_i\,m_i - \lambda_i\,s + 2\lambda_i\,\sqrt{m_i \cdot s}\,\Cos(|x_i-y|) \right]
		\label{eq:ProofQMMAuxConjugation}
\end{align}
If $\psi_i>\lambda_i$ for some $i \in \{1,\ldots,N\}$, we find by sending $m_i \to \infty$ that $\iota_{\QMM(\vec{x},y)}(\vec{\psi})=\infty$.
If $\psi_i=\lambda_i$ and $\Cos(|x_i-y|)>0$, by fixing some $s>0$ and sending $m_i \to \infty$ one also has $\iota_{\QMM(\vec{x},y)}(\vec{\psi})=\infty$.
For $\psi_i=\lambda_i$ and $\Cos(|x_i-y|)=0$ the objective does not depend on $m_i$.
Assume now that $\psi_i \leq \lambda_i$ and that $[\psi_i=\lambda_i]$ $\Rightarrow$ $[\Cos(|x_i-y|)=0]$ (as otherwise $\iota_{\QMM(\vec{x},y)}(\vec{\psi})=\infty$). Set $I(y)=\{i \in \{1,\ldots,N\} | \Cos(|x_i-y|)>0\}$. Then, in \eqref{eq:ProofQMMAuxConjugation}, for $i \in I(y)$ one can explicitly maximize over $m_i$ (for $s$ fixed) and obtains
\begin{align*}
	\iota_{\QMM(\vec{x},y)}(\vec{\psi}) & = \sup_{s \in \R_+} \left[\sum_{i \in I(y)} \frac{\lambda_i\,s\,\Cos(|x_i-y|)^2}{1-\psi_i/\lambda_i} \right] - s.
\end{align*}
Taking now the supremum over $s \in \R_+$ and adopting the convention \eqref{eq:QMMAuxConvention} one arrives at \eqref{eq:QMMAux}.

The first equation of \eqref{eq:MMAuxMMRelation} follows directly from the definitions of $\cMM(\vec{x},\cdot)$ and $\cMM(\vec{x},\cdot,y)$. The second equality follows from
\begin{align*}
	\iota_{\QMM(\vec{x})}(\vec{\psi}) & = \sup_{\vec{m} \in \R^N} \sum_{i=1}^N \psi_i \cdot m_i - \cMM(\vec{x},\vec{m}) \\
		& = \sup_{\vec{m} \in \R^N, y \in \Omega} \sum_{i=1}^N \psi_i \cdot m_i - \cMM(\vec{x},\vec{m},y)
		= \sup_{y \in \Omega} \iota_{\QMM(\vec{x},y)}(\vec{\psi}).
	\qedhere
\end{align*}
\end{proof}

Based on \eqref{eq:MMAuxMMRelation} we can now give a somewhat explicit expression for the set $\QMM(\vec{x})$.
\begin{proposition}
\label{prop:QMM}
For $\vec{x} \in \Omega^N$ one has
\begin{align}
	\QMM(\vec{x})
	& = \left\{ \vec{\psi} \in \R^N \,\middle|\,
	\psi_i \leq \lambda_i-\lambda_i^2 \tn{ for } i \in \{1,\ldots,N\},\,
	\sum_{i=1}^N \frac{\lambda_i\,\Cos(|x_i-y|)^2}{1-\psi_i/\lambda_i} \leq 1\,
	\forall\,y \in \Omega
	\right\}.
	\label{eq:QMM}
\end{align}
It suffices to enforce the constraint in \eqref{eq:QMM} for all $y$ in the convex hull of the points $x_1,\ldots,x_N$.
\end{proposition}
\begin{proof}
Combining \eqref{eq:QMMAux} and \eqref{eq:MMAuxMMRelation} one finds
\begin{align*}
	\QMM(\vec{x})
	& = \left\{ \vec{\psi} \in \R^N \,\middle|\,
	\psi_i \leq \lambda_i \tn{ for } i \in \{1,\ldots,N\},\,
	\sum_{i=1}^N \frac{\lambda_i\,\Cos(|x_i-y|)^2}{1-\psi_i/\lambda_i} \leq 1\,
	\forall\,y \in \Omega
	\right\}
\end{align*}
where we still need to keep in mind convention \eqref{eq:QMMAuxConvention}. By considering the constraint for $y=x_i$, we find that $\psi_i\leq\lambda_i-\lambda_i^2$ and can thus restrict $\psi_i$ to $(-\infty,\lambda_i-\lambda_i^2]$ and dispense with convention \eqref{eq:QMMAuxConvention}.

For the convex hull we argue as in Lemma \ref{lem:CMMMinimizerExist} \eqref{item:CMMConvexHull}. If $y$ does not lie in the convex hull, let $y'$ be the corresponding projection. Then $|x_i-y|>|x_i-y|$, and thus by the monotonicity of $\Cos(\cdot)$ the constraint for $y'$ is stricter than the one for $y$ and the latter is redundant.
\end{proof}

The duality between $\cMMHull$ and $\QMM$ will, in the following, become a very useful tool and the next Lemma will play a central role.
\begin{lemma}
\label{lem:PsiUniqueness}
For $\vec{x} \in \Omega^N$ the function $\cMMHull(\vec{x},\cdot)$ is differentiable on $\R_{++}^N$, i.e.~for every $\vec{m} \in \R_{++}^N$ there exists a unique $\vec{\psi} \in \R^N$ such that $\vec{m}$ and $\vec{\psi}$ satisfy the three following equivalent conditions:
\begin{align}
	\label{eq:PsiMSubDiffRelation}
	[\vec{\psi} \in \partial \cMMHull(\vec{x},\vec{m})]
	\Leftrightarrow [\vec{m} \in \partial \iota_{\QMM(\vec{x})}(\vec{\psi})]
	\Leftrightarrow \left[\sum_{i=1}^N \psi_i \cdot m_i = \cMMHull(\vec{x},\vec{m})
	\wedge \vec{\psi} \in \QMM(\vec{x}) \right]
\end{align}
where $\partial \cMMHull$ denotes the subdifferential of $\cMMHull$ with respect to the mass arguments (for fixed $\vec{x}$).
\end{lemma}
\begin{proof}
Since $\cMMHull$ is by construction proper, convex and lower-semicontinuous in the mass arguments, equivalence of the three conditions \eqref{eq:PsiMSubDiffRelation} is a well-known result from convex analysis.

As $c$ is finite for all non-negative masses and bounded from below, $\cMM$ and $\cMMHull$ are finite for all non-negative masses.
By construction, $\cMMHull$ is convex in its mass arguments and thus continuous in its mass arguments for strictly positive masses $\vec{m} \in \R_{++}^N$.
Therefore, $\partial \cMMHull(\vec{x},\vec{m}) \neq \emptyset$ and some $\vec{\psi}=(\psi_1,\ldots,\psi_N) \in \partial \cMMHull(\vec{x},\vec{m})$ exists, which therefore satisfies \eqref{eq:PsiMSubDiffRelation}.

For uniqueness, assume $\vec{\phi}=(\phi_1,\ldots,\phi_N) \in \partial \cMMHull(\vec{x},\vec{m})$ and thus also satisfies
\begin{align*}
	\vec{\phi} & \in \QMM(\vec{x}), &
	\sum_{i=1}^N \phi_i \cdot m_i = \cMMHull(\vec{x},\vec{m}).
\end{align*}
Let now $\vec{\xi}=\tfrac12 (\vec{\psi} + \vec{\phi})$. One has immediately that $\sum_{i=1}^N \xi_i \cdot m_i = \cMMHull(\vec{x},\vec{m})$ and $\vec{\xi} \in \QMM(\vec{x})$ and thus that
\begin{align}
	\label{eq:ProofCMMMinimizerUniqueSubdiff}
	\vec{m} \in \partial \iota_{\QMM(\vec{x})}(\vec{\xi}).
\end{align}
Assume now $\vec{\phi} \neq \vec{\psi}$. Since $\vec{\xi} \in \QMM(\vec{x})$, which implies $\xi_i <\lambda_i$, and in particular by strict convexity of $(-\infty,1) \ni z \mapsto \tfrac{1}{1-z}$ one has that
\begin{align*}
	\frac{1}{1-\xi_i/\lambda_i} < \frac12 \frac{1}{1-\psi_i/\lambda_i} + \frac12 \frac{1}{1-\phi_i/\lambda_i}
	\qquad \tn{for }  i \in \{1,\ldots,N\}
\end{align*}
and all $\vec{\tilde{\xi}} \in \R^N$ that satisfy
\begin{align*}
	\tilde{\xi}_i < \lambda_i, \qquad
	\frac{1}{1-\tilde{\xi}_i/\lambda_i} \leq \frac12 \frac{1}{1-\psi_i/\lambda_i} + \frac12 \frac{1}{1-\phi_i/\lambda_i}
	\qquad \tn{for }  i \in \{1,\ldots,N\}
\end{align*}
are also contained in $\QMM(\vec{x})$. This implies that $\vec{\xi}$ lies in the interior of $\QMM(\vec{x})$ and therefore $\partial \iota_{\QMM(\vec{x})}(\vec{\xi})=\{0\}$, which contradicts \eqref{eq:ProofCMMMinimizerUniqueSubdiff}. Therefore, $\vec{\phi}=\vec{\psi}$ and thus $\vec{\psi}$ is unique.
\end{proof}

For the equivalence between the coupled-two-marginal and the multi-marginal formulation of the Wasserstein barycenter the map $\WT$, see \eqref{eq:WTMap}, played a central role. For the $\HK$ barycenter we were able to show that minimizers $(y,s)$ exist, Lemma \ref{lem:CMMMinimizerExist}, but we still need a measurable map $(\vec{x},\vec{m})$ to minimizers $(y,s)$ to mimic the proof strategy of Proposition \ref{prop:WBarycenterEquivalence}. The following Proposition, the main result of this section, establishes that such a map also exists for the $\HK$ multi-marginal cost function, at least on a subset of $\cone^N$, and is even continuous.
\begin{proposition}
\label{prop:CMMMinimizerUnique}
For $(\vec{x},\vec{m}) \in \CSet$ (i.e.~if $\cMM(\vec{x},\vec{m})=\cMMHull(\vec{x},\vec{m})$ and $\vec{m} \neq 0$) the minimizer $(y,s)$ in the definition of $\cMM$, \eqref{eq:CMM} is unique. $y$ lies in the convex hull of the $x_i$ for which $m_i>0$.
The map $\HKT: \CSet \ni (\vec{x},\vec{m}) \mapsto (y,s) \in \cone$, which takes $(\vec{x},\vec{m})$ to the unique minimizer $(y,s)$, is continuous. One has $|y-x_i|<\pi/2$ for all $i \in \{1,\ldots,N\}$ where $m_i>0$.
\end{proposition}
\begin{proof}
\textbf{Step 1: Preparation.} Throughout this proof let $(x_1,m_1),\ldots,(x_N,m_N) \in \cone$ be fixed such that $(\vec{x},\vec{m})\in \CSet$.
Further, assume for now that $\vec{m} \in \R_{++}^N$. Partially zero masses are discussed further below.
By Lemma \ref{lem:PsiUniqueness} there exists a unique $\vec{\psi} \in \R^N$ that satisfies the three equivalent conditions \eqref{eq:PsiMSubDiffRelation}.

\textbf{Step 2: Optimality condition for $y$.}
By Lemma \ref{lem:CMMMinimizerExist} there exists some $y \in \Omega$ such that $\cMM(\vec{x},\vec{m})=\cMM(\vec{x},\vec{m},y)$, cf.~\eqref{eq:CMMAux}.
Since $\vec{\psi} \in \QMM(\vec{x})$ one must find $\vec{\psi} \in \QMM(\vec{x},y)$, cf.~\eqref{eq:MMAuxMMRelation}.
And since $\cMM(\vec{x},\vec{m},y)=\cMM(\vec{x},\vec{m})=\cMMHull(\vec{x},\vec{m})=\sum_{i=1}^N \psi_i \cdot m_i$ one must have that $\vec{m} \in \partial \iota_{\QMM(\vec{x},y)}(\vec{\psi})$, cf.~\eqref{eq:CMMAuxConjugate} and \eqref{eq:PsiMSubDiffRelation}.
In the primal picture this means $\psi \in \partial \cMM(\vec{x},\vec{m},y)$, which is differentiable in $\vec{m}$, see \eqref{eq:CMMAux}. This condition can be written down explicitly to obtain
\begin{align*}
	\psi_i = \lambda_i - \eta \frac{\Cos(|x_i-y|)}{\sqrt{m_i}}
	\quad \tn{for } i \in \{1,\ldots,N\}
	\quad \tn{where} \quad
	\eta \assign \sum_{i=1}^N \lambda_i\,\sqrt{m_i}\,\Cos(|x_i-y|)
\end{align*}
and $\eta$ does not depend on the choice of the minimizer $y$ for $\cMM(\vec{x},\vec{m})$.
Since $\psi_i<\lambda_i$ by $\vec{\psi} \in \QMM(\vec{x})$ this implies $\Cos(|x_i-y|)>0$ (and therefore $|x_i-y|<\pi/2$) and $\eta>0$ and we can thus resolve for $|x_i-y|$ to get
\begin{align*}
	|x_i-y| & = r_i \assign \arccos\left(\frac{\sqrt{m_i}\,(\lambda_i-\psi_i)}{\eta}\right)
	\qquad \tn{for } i\in\{1,\ldots,N\}.
\end{align*}
This implies that $y \in \ol{B}(x_i,r_i)$ for $i \in \{1,\ldots,N\}$ where $\ol{B}(x_i,r_i)$ denotes the closed ball around $x_i$ with radius $r_i$. If the intersection $B \assign \bigcap_{i=1}^N \ol{B}(x_i,r_i)$ contains more than a single point, it must contain some point $y'$ which satisfies $|x_i-y'| < r_i$ for $i \in \{1,\ldots,N\}$ and thus $\Cos(|x_i-y'|)>\Cos(|x_i-y|)$. Plugging this into \eqref{eq:CMM} we find that $y$ could not be a minimizer, which is a contradiction.
Thus, $B = \{y\}$ and $y$ must be unique.

\textbf{Step 3: Convex hull.}
By Lemma \ref{lem:CMMMinimizerExist} \eqref{item:CMMConvexHull} the search for minima $y$ can be restricted to the convex hull of the $x_1,\ldots,x_N$. Therefore, the unique minimizer must lie in the hull.

\textbf{Step 4: Unique $s$.}
Finally, by assumption $m_i>0$ and for the unique minimizer $y$, we have noted above that $\Cos(|x_i-y|)>0$ and thus $s \mapsto c(x_i,m_i,y,s)$ is strictly convex, which implies that the minimizing $s$ in \eqref{eq:CMMAux} is unique.

\textbf{Step 5: Extension to partially zero masses.}
Now we treat the case where some (but not all) $m_i$ are zero.
In this case the subdifferential $\partial \cMM(\vec{x},\vec{m})$ may be empty and thus we argue by reduction to the strictly positive case.
By reordering the arguments we may assume that $m_i>0$ for $i \in \{1,\ldots,K\}$ and $m_i=0$ for $i \in \{K+1,\ldots,N\}$ for some $K \in \{1,\ldots,N-1\}$. Set
\begin{align*}
	\Lambda & \assign \sum_{i=1}^K \lambda_i, &
	\vec{x}^K & \assign (x_1,\ldots,x_K), &
	\vec{m}^K & \assign \Lambda^2 \cdot (m_1,\ldots,m_K)
\end{align*}
and finally
\begin{align*}
	\cMM^K(\vec{x}^K,\vec{m}^K) \assign \inf_{(y,s) \in \cone} \sum_{i=1}^K \frac{\lambda_i}{\Lambda} c(x_i^K,m_i^K,y,s).
\end{align*}
One can quickly verify that the objectives for $\cMM(\vec{x},\vec{m})$ and $\cMM^K(\vec{x}^K,\vec{m}^K)$ differ by $(\sum_{i=1}^K \lambda_i\allowbreak m_i) \cdot (1-\Lambda)$ and thus both problems have the same minimizer set $(y,s) \in \cone$. In particular, the minimizer for $\cMM(\vec{x},\vec{m})$ is unique if the one for $\cMM^K(\vec{x}^K,\vec{m}^K)$ is unique.
Since $m^K_i>0$ for $i\in\{1,\ldots,K\}$, the above arguments apply for the uniqueness of the minimizer when $\cMM^{K\ast\ast}(\vec{x}^K,\vec{m}^K)=\cMM^K(\vec{x}^K,\vec{m}^K)$.

Assume for contradiction that $\cMM^{K\ast\ast}(\vec{x}^K,\vec{m}^K)<\cMM^K(\vec{x}^K,\vec{m}^K)$ which (by 1-homogeneity) implies that there exist $\vec{m}^{K,j} \in \R^K_+$, $j \in \{1,\ldots,J\}$ for some $J \in \N$ such that
\begin{align*}
	\sum_{j=1}^J \vec{m}^{K,j} & = \vec{m}^K, &
	\cMM^K(\vec{x}^K,\vec{m}^K) & > \sum_{j=1}^J \cMM^K(\vec{x}^K,\vec{m}^{K,j}).
\end{align*}
For $j \in \{1,\ldots,J\}$ set now
\begin{align*}
	\vec{m}^j \assign \Lambda^{-2} \cdot (m^{K,j}_1,\ldots,m^{K,j}_K,0,\ldots,0)
\end{align*}
where we fill entries $K+1$ to $N$ with zeros. Then $\sum_{j=1}^J \vec{m}^j = \vec{m}$ and arguing as above we find
\begin{align*}
	\cMM(\vec{x},\vec{m}^j) = \cMM^K(\vec{x}^K,\vec{m}^{K,j}) + (1-\lambda) \cdot \sum_{i=1}^K \lambda_i\, m_i^j.
\end{align*}
Finally, this yields
\begin{multline*}
	\cMM(\vec{x},\vec{m}) = \cMM^K(\vec{x}^K,\vec{m}^K)
	+ (1-\lambda) \cdot \sum_{i=1}^K \lambda_i\, m_i \\
	> \sum_{j=1}^J \left( \cMM^K(\vec{x}^K,\vec{m}^{K,j})
	+ (1-\lambda) \cdot \sum_{i=1}^K \lambda_i\, m_{i}^j \right)
	= \sum_{j=1}^J \cMM(\vec{x},\vec{m}^{j})
\end{multline*}
which contradicts $(\vec{x},\vec{m}) \in \CSet$.

\textbf{Step 6: Continuity.}
Now, let $(\iterl{\vec{x}},\iterl{\vec{m}})_\ell$ be a sequence in $\CSet^N$ that converges to $(\vec{x},\vec{m}) \in \CSet$ and let $(\iterl{y},\iterl{s})\ell$ be the sequence of corresponding unique minimizers. Since $(\iterl{\vec{m}})_\ell$ is bounded, $\iterl{s}$ must be bounded (cf.~Lemma \ref{lem:CMMMinimizerExist}) and thus by compactness of $\Omega$, $(\iterl{y},\iterl{s})_\ell$ must have some cluster point $(y,s) \in \cone$.
By continuity of $\cMM$ (Lemma \ref{lem:CMMMinimizerExist}) and $c$ and up to selection of a suitable subsequence one thus finds:
\begin{align*}
	\cMM(\vec{x},\vec{m}) & = \lim_{\ell \to \infty} \cMM(\iterl{\vec{x}},\iterl{\vec{m}})
	= \lim_{\ell \to \infty} \sum_{i=1}^N c(\iterl{x}_i,\iterl{m}_i,\iterl{y},\iterl{s})
	= \sum_{i=1}^N c(x_i,m_i,y,s).
\end{align*}
Therefore, any cluster point of $(\iterl{y},\iterl{s})_\ell$ must be equal to the unique minimizer for $\cMM(\vec{x},\vec{m})$ and thus the whole sequence must converge to $(y,s)$.
\end{proof}
			
The last Lemma of this section gives an explicit relation between $\cMMHull$ and $\cMM$, which is instructive in understanding the $\HK$ barycenter between Dirac measures, see Proposition \ref{prop:HKMMDirac}. Moreover, it provides the technical aspect that $\cMMHull$ is jointly lower semi-continuous in $(\vec{x},\vec{m})$ which is required for the derivation of the dual of the multi-marginal problem, Theorem \ref{thm:HKMMDual}.
\begin{lemma}
\label{lem:CMMHullCaratheodory}
For $\vec{x} \in \Omega^N$, $\vec{m} \in \R_+^N$ one has the representation
\begin{align}
	\label{eq:CMMHullCaratheodory}
	\cMMHull(\vec{x},\vec{m}) & = \min \left\{
		\sum_{j=1}^N \cMM(\vec{x},\vec{m}^j) \middle|
			\vec{m}^1,\ldots,\vec{m}^N \in \R_+^N,\,
			\sum_{j=1}^N \vec{m}^j = \vec{m}
		\right\}
\end{align}
	Minimizing $\vec{m}^j$ satisfy $(\vec{x},\vec{m}^j) \in \CSet_0$.
	$\cMMHull$ is lower-semicontinuous on $(\Omega \times \R)^N$.
\end{lemma}
\begin{proof}
	For $\vec{m}=0$ formula \eqref{eq:CMMHullCaratheodory} is trivial, for $\vec{m} \neq 0$, and with $\inf$ in place of $\min$, it follows from \cite[Corollary 17.1.6]{Rockafellar1972Convex}.
	Existence of minimizers follows from continuity of $\cMM$ (Lemma \ref{lem:CMMMinimizerExist}) and compactness of the minimization set.
	Minimizers must satisfy $(\vec{x},\vec{m}^j) \in \CSet$, otherwise an even better decomposition could be found.
	
	$\cMM$ is finite and non-negative on $\cone^N$ and $+\infty$ on $(\Omega \times \R)^N \setminus \cone^N$. Thus, $\cMMHull$ shares these properties. It remains to prove that $\cMMHull$ is lower-semicontinuous on $\cone^N$.	
	Let $\vec{x} \in \Omega^N$, $\vec{m} \in \R_+^N$ and let $(\iterl{\vec{x}})_\ell$ and $(\iterl{\vec{m}})_\ell$ be sequences in $\Omega^N$ and $\R_+^N$ converging to $\vec{x}$ and $\vec{m}$. For $\iterl{\vec{x}}$ and $\iterl{\vec{m}}$ let $(\iter{\vec{m}}{1,\ell},\ldots,\iter{\vec{m}}{N,\ell})$ be minimizers for $\cMMHull(\iterl{\vec{x}},\iterl{\vec{m}})$ in \eqref{eq:CMMHullCaratheodory}.
	Since $(\iterl{\vec{m}})_\ell$ is bounded, the sequence of minimizers is bounded and a converging subsequence that attains the limit inferior of the sequence $(\cMMHull(\iterl{\vec{x}},\iterl{\vec{m}}))_\ell$ can be extracted (for simplicity the sequence notation now refers to this subsequence). Let $(\vec{m}^1,\ldots,\vec{m}^N)$ be the limit which satisfies $\sum_{j=1}^N \vec{m}^j = \vec{m}$ and therefore by continuity of $\cMM$ on $\cone^N$,
	\begin{equation*}
		\liminf_{\ell \to \infty} \cMMHull(\iterl{\vec{x}},\iterl{\vec{m}})
		= \lim_{\ell \to \infty} \sum_{j=1}^N \cMM(\iterl{\vec{x}},\iter{\vec{m}}{j,\ell})
		= \sum_{j=1}^N \cMM(\vec{x},\vec{m}^j)
		\geq \cMMHull(\vec{x},\vec{m}).
		\qedhere
	\end{equation*}
\end{proof}

\subsection{Duality and equivalence with coupled-two-marginal formulation}
In this section we establish a dual problem for the multi-marginal barycenter problem. As before, we show duality with the semi-coupling formulation as it merely involves measures on compact spaces.
\begin{theorem}[Dual multi-marginal barycenter problem and existence of minimizers]
\label{thm:HKMMDual}
\begin{multline}
	\HKMM(\mu_1,\ldots,\mu_N)^2 =
	\sup\left\{
	\sum_{i=1}^N \int_\Omega \psi_i \,\diff \mu_i \middle|
	\psi_1,\ldots,\psi_N \in \cont(\Omega), \right. \\
	\left. \vphantom{\sum_{i=1}^N \int_\Omega}
	(\psi_1(x_1),\ldots,\psi_N(x_N)) \in \QMM(x_1,\ldots,x_N)
	\tn{ for all } x_1,\ldots,x_N \in \Omega
	\right\}
	\label{eq:HKMMDual}
\end{multline}
Minimizers of \eqref{eq:HKMMLifted} and \eqref{eq:HKMMSemiCoupling} exist.
\end{theorem}
\begin{proof}
The proof works in close analogy to Remark \ref{rem:HKDual}.
We choose:
\begin{align*}
	G & : \cont(\Omega)^N \to \R, & & (\psi_1,\ldots,\psi_N) \mapsto -\sum_{i=1}^N \int_\Omega \psi_i\,\diff \mu_i, \\
	F & : \cont(\Omega^N)^N \to \RCupInf, & & (\xi_1,\ldots,\xi_N) \!\mapsto \!\begin{cases}
		0 & \tn{\!\!\!\!if } (\xi_1,\ldots,\xi_N)(x_1,\ldots,x_N) \in \QMM(x_1,\ldots,x_N) \\
		& \tn{\!\!\!\!for all } x_1,\ldots,x_N \in \Omega, \\
		+\infty & \tn{\!\!\!\!else,}
		\end{cases} \\
	A & : \cont(\Omega)^{N} \to \cont(\Omega^N)^N, & & (\psi_1,\ldots,\psi_N) \mapsto (\xi_1,\ldots,\xi_N) \\
	& & & \qquad \qquad \tn{with } \xi_i(x_1,\ldots,x_N) = \psi_i(x_i).
\end{align*}
With Lemmas \ref{lem:IntConjugation} and \ref{lem:CMMHullCaratheodory} conjugates of $F$ and $G$ and the adjoint of $A$ are obtained as before.
Once more with Fenchel--Rockafellar duality one finds
\begin{align*}
	\tn{\eqref{eq:HKMMDual}} & = -\inf \left\{G(\psi_1,\ldots,\psi_N) + F(A(\psi_1,\ldots,\psi_N))
		\middle| \psi_1,\ldots,\psi_N \in \cont(\Omega) \right\} \\
		& = \inf \left\{G^\ast(-A^\ast\gamma) + F^\ast(\gamma)
		\middle| \gamma \in \meast(\cone^N) \right\}
		= \tn{\eqref{eq:HKMMSemiCoupling}}
\end{align*}
and one directly implies existence of minimizers for \eqref{eq:HKMMSemiCoupling}.
Via the explicit construction underlying the proof of Theorem \ref{thm:HKMMSemiCoupling} (see also Remark \ref{rem:HKLiftedSCEquiv}) this implies existence of minimizers in \eqref{eq:HKMMLifted}.
\end{proof}

Now we have gathered all ingredients to show equivalence between the multi-marginal and the semi-coupling formulation.
Since we have established the existence of a measurable map $\HKT$ of tuples $\vec{x} \in \Omega^N$, $\vec{m} \in \R_+^N$ to corresponding minimizers $(y,s) \in \cone$ in \eqref{eq:CMM} only for $(\vec{x},\vec{m}) \in \CSet$, the reasoning from Proposition \ref{prop:WBarycenterEquivalence} is not entirely sufficient here. To overcome this difficulty we rely on the dual problems as well.
\begin{theorem}
\label{thm:HKCTMMMEquiv}
Let $\mu_1,\ldots,\mu_N \in \measp(\Omega)$.
\begin{enumerate}[(i)]
	\item One has $\HKCTM(\mu_1,\ldots,\mu_N)^2 = \HKMM(\mu_1,\ldots,\mu_N)^2$.
		\label{item:HKCTMMMEquivValues}
	\item This equality also holds if one uses $\cMM$ instead of $\cMMHull$ in $\HKMM(\mu_1,\ldots,\mu_N)^2$, \eqref{eq:HKMMLifted}.
		\label{item:HKCTMMMEquivNonRelax}
	\item There exist minimizers of \eqref{eq:HKMMLifted} that are concentrated on $\CSet$, see \eqref{eq:CSet}.
		\label{item:HKCTMMMEquivConcentration}
	\item If $\gamma$ minimizes \eqref{eq:HKMMLifted} and is concentrated on $\CSet$ then $\coneProj \HKT_\sharp \gamma$ is a \emph{barycenter}, i.e.~it minimizes \eqref{eq:HKCTM}.
		\label{item:HKCTMMMEquivProject}
	\item Conversely, for every minimizer $\nu$ of \eqref{eq:HKCTM} there is some $\gamma$ that minimizes \eqref{eq:HKMMLifted} with $\coneProj \HKT_\sharp \gamma=\nu$.
		\label{item:HKCTMMMEquivLift}
\end{enumerate}
\end{theorem}
\begin{proof}
\textbf{Step 1: $\HKCTM(\mu_1,\ldots,\mu_N)^2 \geq \HKMM(\mu_1,\ldots,\mu_N)^2$.}
Let $\nu$ be a minimizer for $\HKCTM(\allowbreak\mu_1,\allowbreak \ldots,\allowbreak \mu_N)^2$, \eqref{eq:HKCTM} and let $\gamma_i \in \measpt(\cone^2)$ be minimizers for $\HK(\mu_i,\nu)^2$ in \eqref{eq:HKLifted} for $i \in \{1,\ldots,N\}$.
By the constraints of \eqref{eq:HKLifted} one has $\coneProj \pi_{2\sharp} \gamma_i = \nu$ for $i \in \{1,\ldots,N\}$ but the marginals $\pi_{2\sharp} \gamma_i$ are not necessarily all identical.
By virtue of \cite[Lemma 7.10]{LieroMielkeSavare-HellingerKantorovich-2015a} there are $\hat{\gamma}_i$ that are also minimal in \eqref{eq:HKLifted} for $\HK(\mu_i,\nu)^2$ and that satisfy
\begin{align*}
	\pi_{2\sharp} \hat{\gamma}_i = \nu \otimes \delta_1 \assignRe \rho \in \measpt(\cone).
\end{align*}
for $i \in \{1,\ldots,N\}$. Here $\rho=\nu \otimes \delta_1$ denotes the measure characterized by
\begin{align*}
	\int_\cone \phi(x,m)\,\diff \rho(x,m) = \int_\Omega \phi(x,1)\,\diff \nu(x)
	\qquad \tn{for } \phi \in \contNC(\cone).
\end{align*}
We can now repeat the construction \eqref{eq:WMMW2TMEquivGluing} from the proof of Proposition \ref{prop:WBarycenterEquivalence}: Let $(\hat{\gamma}_i^{(y,s)})_{(y,s) \in \cone}$ be the disintegration of $\hat{\gamma}_i$ with respect to its second marginal and introduce $\hat{\gamma} \in \measpt(\cone^N)$ by
\begin{align}
		\int_{\cone^N} \phi\,\diff \hat{\gamma} & \assign \int_{\cone^{N+1}} \phi((x_1,m_1),\ldots,(x_N,m_N))\,
			\diff \hat{\gamma}_1^{(y,s)}(x_1,m_1)\ldots \diff \hat{\gamma}_N^{(y,s)}(x_N,m_N)\,\diff \rho(y,s)
		\label{eq:ProofHKCTMMMEquivGammaConstruction}
\end{align}
for $\phi \in \contNC(\cone^N)$.
Analogous to Proposition \ref{prop:WBarycenterEquivalence} we find $\pi_{i\sharp} \hat{\gamma}=\hat{\gamma}_i$ for $i \in \{1,\ldots,N\}$ and by using
\begin{align*}
	\cMMHull(\vec{x},\vec{m}) \leq \cMM(\vec{x},\vec{m}) \leq \sum_{i=1}^N \lambda_i\,c(x_i,m_i,y,s)
	\qquad \tn{for all } (y,s) \in \cone
\end{align*}
one finds, in analogy to \eqref{eq:WMMWC2MEquivProofB}, that
\begin{align}
	\HKMM(\mu_1,\ldots,\mu_N)^2 & \leq \int_{\cone^N} \cMMHull(x_1,m_1,\ldots,x_N,m_N)
		\,\diff \hat{\gamma}((x_1,m_1),\ldots,(x_N,m_N))
		\nonumber \\
	& \leq \int_{\cone^N} \cMM(x_1,m_1,\ldots,x_N,m_N)
		\,\diff \hat{\gamma}((x_1,m_1),\ldots,(x_N,m_N))
		\nonumber \\
	& \leq \sum_{i=1}^N \lambda_i \int_{\cone^2} c(x,r,y,s)\,\diff \hat{\gamma}_i((x,r),(y,s))
		\nonumber \\
	& = \sum_{i=1}^N \lambda_i\,\HK(\mu_i,\nu)^2 = \HKCTM(\mu_1,\ldots,\mu_N)^2.
	\label{eq:ProofHKCTMMMEquivSandwich}
\end{align}

\textbf{Step 2: $\HKCTM(\mu_1,\ldots,\mu_N)^2 \leq \HKMM(\mu_1,\ldots,\mu_N)^2$.}
For the converse inequality we cannot replicate the argument of Proposition \ref{prop:WBarycenterEquivalence} since we have, in this article, not established the existence of a measurable map that takes $\Omega \times \R_+^N \ni (\vec{x},\vec{m})$ to minimizers $(y,s) \in \Omega\times \R_+$ in the definition of $\cMM(\vec{x},\vec{m})$ beyond the set $\CSet$.
Instead we argue via the dual problems.

By a simple density argument in \eqref{eq:HKCTMLiftedDual} we can add the constraint $\psi_i(x)<\lambda_i$ for all $i \in \{1,\ldots,N\}$, $x \in \Omega$ without changing the value of the supremum (the constraint $(\tfrac{\psi_i(x)}{\lambda_i},\tfrac{\phi_i(y)}{\lambda_i}) \in Q(x,y)$ already entails the constraint $\psi_i(x)\leq \lambda_i$ and any feasible $\psi_i$ does not become unfeasible by slightly decreasing its value. (We are not interested in the existence of maximizers in \eqref{eq:HKCTMLiftedDual}.)
Let now $(\psi_1,\ldots,\psi_N,\phi_1,\ldots,\phi_N) \in \cont(\Omega)^{2N}$ be feasible candidates for \eqref{eq:HKCTMLiftedDual} with $\psi_i < \lambda_i$. Then one finds with \eqref{eq:HKQ} for all $x_1,\ldots,x_N,y \in \Omega$ that
\begin{align}
	\sum_{i=1}^N \frac{\lambda_i\,\Cos(|x_i-y|)^2}{1-\psi_i(x_i)/\lambda_i} \leq 1- \sum_{i=1}^N \phi_i(y) \leq 1
\end{align}
where the last inequality is due to the constraint on $\phi_1,\ldots,\phi_N$ in \eqref{eq:HKCTMLiftedDual} and thus that
\begin{align*}
	(\psi_1(x_1),\ldots,\psi_N(x_N)) \in \QMM(x_1,\ldots,x_N).
\end{align*}
So $(\psi_1,\ldots,\psi_N)$ is feasible for \eqref{eq:HKMMDual} with the same score as for \eqref{eq:HKCTMLiftedDual}. Since we can do this for a dense set of candidates in \eqref{eq:HKCTMLiftedDual} one finds $\tn{\eqref{eq:HKCTMLiftedDual}} \leq \tn{\eqref{eq:HKMMDual}}$.

\textbf{Step 3: Conclusion.}
As in Proposition \ref{prop:WBarycenterEquivalence} these two inequalities establish \eqref{item:HKCTMMMEquivValues}. Combing this with the chain of inequalities \eqref{eq:ProofHKCTMMMEquivSandwich}, which we now know to be equalities, and $\cMMHull \leq \cMM$ one finds \eqref{item:HKCTMMMEquivNonRelax}.
This also implies that the $\hat{\gamma}$ constructed in \eqref{eq:ProofHKCTMMMEquivGammaConstruction} is a minimizer concentrated on $\CSet_0$.
Since $\cMMHull=0$ on $\CSet_0 \setminus \CSet$ and $\coneProj \pi_{i\sharp} (\hat{\gamma} \restr S)=\coneProj \pi_{i\sharp} \hat{\gamma}$ one has that $\hat{\gamma} \restr S$ is also a minimizer, which is concentrated on $S$, thus proving \eqref{item:HKCTMMMEquivConcentration}.
If $\gamma$ is a minimizer of \eqref{eq:HKMMLifted} that is concentrated on $\CSet$ we can work analogous to \eqref{eq:WMMWCTMEquivProofA} with the map $\HKT$ and the operator $\coneProj$ to show that $\coneProj \HKT_{\sharp} \gamma$ is a barycenter, yielding \eqref{item:HKCTMMMEquivProject}.
Finally, repeating the arguments from Proposition \ref{prop:WBarycenterEquivalence} one finds that $\HKT_\sharp \hat{\gamma}=\rho$ and eventually $\coneProj \HKT_\sharp \hat{\gamma}=\nu$ thus showing \eqref{item:HKCTMMMEquivLift}.
\end{proof}

\subsection[An explicit form for c\_MM and Q\_MM for d=1 and small distances]{An explicit form for $\cMM$ and $\QMM$ for $d=1$ and small distances}
\label{sec:CMMExplicit}
In one dimension and for sufficiently small distances the particular form of $c(x_1,m_1,x_2,m_2)$ allows more explicit representations of $\cMM$ and $\QMM$ which are instructive in their own right and useful for some additional results (e.g.~Section \ref{sec:NonExistence}).
\begin{proposition}
	\label{prop:CMMSimple}
	For $d=1$ and $\Omega=[0,\pi/2]$ one has
	\begin{align}
		\label{eq:CMMSimple}
		\cMM(\vec{x},\vec{m}) = \sum_{i=1}^N \lambda_i\,m_i - \sum_{i,j=1}^N \lambda_i\,\lambda_j \sqrt{m_i\,m_j} \cos(|x_i-x_j|)
	\end{align}
	and corresponding minimizers in \eqref{eq:CMM} are given by
	\begin{align}
		\label{eq:CMMSimpleMinimizers}
		y & \assign \arg\left(z\right), &
		s & \assign \left|z\right|^2 &
		\tn{for} \qquad z & \assign \sum_{i=1}^N \lambda_i \sqrt{m_i}\,e^{i\,x_i}
	\end{align}
	where $\arg(z)$ is the unique $\delta \in \Omega$ such that $z=|z| e^{i\delta}$ for $z\neq 0$ and we may choose an arbitrary value in $\Omega$ for $z=0$.
\end{proposition}
\begin{proof}
	Since $|x-y| \leq \pi/2$ on $\Omega$, we have for $(x,m),(y,s) \in \cone$, $c(x,m,y,s)=|\sqrt{m}\,e^{ix}-\sqrt{s}\,e^{iy}|^2$ and therefore
	\begin{equation}
		\label{eq:CMMSimplePre}
		\cMM(\vec{x},\vec{m})
		= \min_{(y,s) \in \cone} \sum_{i=1}^N \lambda_i\,\left|\sqrt{m_i}\,e^{ix_i}-\sqrt{s}\,e^{iy}\right|^2
		= \min_{q \in \C} \sum_{i=1}^N \lambda_i\,\left|p_i-q\right|^2
	\end{equation}
	where we introduced $p_i \assign \sqrt{m_i}\,e^{ix_i}$ and the parametrization $q \assign \sqrt{s}\,e^{iy}$. Note that the minimizing $q$, given by $\sum_{i=1}^N \lambda_i\,p_i$, will lie in the convex hull of the $p_i$ and thus the minimizing $y$ will lie in the convex hull of the $x_i$ and thus in $\Omega$, and thus this reparametrization is indeed valid. This establishes \eqref{eq:CMMSimpleMinimizers}. A simple explicit computation yields that
	\begin{equation*}
		\cMM(\vec{x},\vec{m}) = \sum_{i,j=1}^N \frac{\lambda_i\,\lambda_j}{2}\,|p_i-p_j|^2
		= \sum_{i=1}^N \lambda_i\,m_i - \sum_{i,j=1}^N \lambda_i\,\lambda_j \sqrt{m_i\,m_j} \cos(|x_i-x_j|).
		\qedhere
	\end{equation*}
\end{proof}

In this specific setup it is also possible to explicitly determine the maximal value of the function $\sum_{i=1}^N \frac{\lambda_i \Cos(|x_i-y|)^2}{1-\psi_i/\lambda_i}$ in \eqref{eq:QMM} and thus to simplify the description of $\QMM(\vec{x})$. For this, it is more convenient to switch to scaled coordinates $\chi_i \assign \frac{\lambda_i}{1-\psi_i/\lambda_i}$ where the reparametrization $\psi_i \mapsto \chi_i$ is an increasing bijection on $(-\infty,\lambda_i-\lambda_i^2] \leftrightarrow (0,1]$.
\begin{proposition}
	\label{prop:QMMQuadratic} a) For $\vec{x} \in \Omega^N$, $\vec{\psi} \in \R^N$,
	\begin{align}
		\label{eq:QMMTildeEquiv}
		[\vec{\psi} \in \QMM(\vec{x})]
		\qquad \Leftrightarrow \qquad
		[\exists\, \vec{\chi} \in \QMMTilde(\vec{x}) \tn{ with }
		\psi_i=\lambda_i-\lambda_i^2/\chi_i
		\tn{ for } i \in \{1,\ldots,N\}]
	\end{align}
	where
	$$
        \QMMTilde(\vec{x}) = \left\{\vec{\chi}\in(0,1]^N \, \middle| \,
	    \sum_{i=1}^N\chi_i \Cos(|x_i-y|)^2\le 1 \, \forall\, y\in\Omega\right\}.
    $$
	b) For $d=1$ and $\Omega=[0,\pi/2]$,
	\begin{align}
		\label{eq:QMMTilde}
		\QMMTilde(\vec{x}) = \left\{ \vec{\chi} \in (0,1]^N \middle|
			\sum_{i=1}^N \chi_i - \sum_{1\leq j<k\leq N} \chi_j\chi_k \sin(x_j-x_k)^2 \le 1
			\right\}.
	\end{align}
\end{proposition}
Thus in scaled coordinates the set $\QMM$ is described just by trivial linear upper bounds on each component and a single quadratic inequality. The boundary of the set satisfying the quadratic inequality is a two-sheeted hyperboloid and the linear bounds just serve to select the right branch.

\begin{proof}
a) is immediate from Proposition \ref{prop:QMM}. As for b), in the description of $\QMMTilde$ in a) the functions $\Cos(|x_i-y|)$ reduce to $\cos(|x_i-y|)$.  By the trigonometric identity $\cos(s)^2 = \frac12(1+\cos(2s))$ it follows that 
\begin{equation}
	\label{eq:fSumComplex}
   \sum_{i=1}^N \chi_i \cos(|y-x_i|)^2 = \frac12 \sum_{i=1}^N \chi_i + \frac12 \underbrace{\sum_{i=1}^N \chi_i \cos(2(y-x_i))}_{=:g(y)}.
\end{equation}
We now polar-decompose $\sum_{j=1}^N \chi_j e^{-2ix_j} = \chi_* e^{-2i\delta}$ with $\chi_*\ge 0$, $\delta\in\R$ (and in particular it is possible to choose $\delta$ in the convex hull of the $(x_i)_i$). It follows by multiplying with $e^{2iy}$ and taking real parts that $g(y)=\chi_*\cos(2(y-\delta))$ (physically, a superposition of one-dimensional same-frequency harmonics is again a harmonic of that frequency) and therefore $\max_{y\in\Omega} g(y)=\max_{y \in \R} g(y)=\chi_*$, and by taking absolute values that
$$
   \chi_* = \sqrt{\sum_{j=1}^N \chi_j^2 + 2 \sum_{1\le j<k\le N}\chi_j\chi_k\cos(2(x_j-x_k))}.
$$
Hence $\chi\in(0,1]^N$ belongs to $\tilde{Q}_{MM}$ if and only if
$$
      \frac12 \sum_{i=1}^N \chi_i + \frac{\chi_*}{2} \le 1
$$
with $\chi_*$ as above. The latter inequality is equivalent to 
$\chi_*\le 2 - \sum_{i=1}^N \chi_i$ and, because of the nonnegativity of $\chi_*$, to its analogon obtained by squaring both sides,  $\chi_*^2 \le (2-\sum_{i=1}^N\chi_i)^2$. After some rearrangement and using $\cos(2(x_j-x_k))=2\cos(x_j-x_k)^2-1$, it finally becomes
\begin{align*}
   \sum_{j=1}^N \chi_j - \sum_{1\le j<k\le N} \chi_j\chi_k \sin(x_j-x_k)^2 \le 1. & \qedhere
\end{align*}
\end{proof}

\begin{example}
\label{ex:N2:1D}
For $N=2$ and $|x_1-x_2| \leq \pi/2$ the computation of $\cMM$ can always be reduced to Proposition \ref{prop:CMMSimple} since the minimization over $y$ in \eqref{eq:CMM} can be restricted to the convex hull of $x_1,x_2$, i.e.~the one-dimensional line segment between $x_1$ and $x_2$ (see Lemma \ref{lem:CMMMinimizerExist} \eqref{item:CMMConvexHull}) which can be embedded isometrically into the interval $[0,\pi/2]$.
	
For $t \in (0,1)$ let $\lambda_1=1-t$, $\lambda_2=t$. Then one finds with \eqref{eq:CMMSimple}
\begin{align}
	\label{eq:ExampleN2CMM}
	\cMM(x_1,m_1,x_2,m_2) & = (1-t)\,t\,c(x_1,m_1,x_2,m_2)
\end{align}
and since this function is jointly convex in $(m_1,m_2)$ (for fixed $(x_1,x_2)$) this also equals $\cMMHull$.
In the dual picture we find with Proposition \ref{prop:QMMQuadratic},
\begin{align}
	\label{eq:ExampleN2QMM}
	\QMMTilde(\vec{x})=\left\{(\chi_1,\chi_2) \in (0,1]^2 \middle| \chi_1 + \chi_2 -\chi_1\,\chi_2 \sin(x_1-x_2)^2 \leq 1\right\}
\end{align}
which, after some rearranging, can be shown to be equivalent to $\QMM(\vec{x})=(1-t)\,t\,Q$ for $Q$ as in \eqref{eq:HKQ}, which is consistent with \eqref{eq:ExampleN2CMM}.
We will complete this discussion in Example \ref{ex:CMM:N2}.
\end{example}

\section{HK barycenter between Dirac measures}
\label{sec:HKMMDirac}
In this section we study the $\HK$ barycenter between Dirac measures $\mu_i\assign m_i \cdot \delta_{x_i}$, $(x_i,m_i) \in \cone$, $i \in \{1,\ldots,N\}$.
The case $N=2$ is already well understood, since here the barycenter corresponds to a point of the geodesic connecting the two marginals (cf.~Examples \ref{ex:N2:1D} and \ref{ex:CMM:N2}). For $|x_1-x_2|<\pi/2$ the barycenter is given by a single Dirac measure on the line segment between $x_1$ and $x_2$. For $|x_1-x_2|>\pi/2$ it is given by two Dirac measures at $x_1$ and $x_2$. At the cut locus $|x_1-x_2|=\pi/2$ any superposition of the two options is viable \cite{ChizatOTFR2015}.
In this section we observe that this behaviour is fully encoded in the behaviour of the function $\cMM$ and the relation to its convex hull $\cMMHull$.
For $N\geq 3$ we show that the `splitting' behaviour of the $\HK$ barycenter between Dirac measures no longer merely depends on the mutual distance of the points but also on the point masses.

We begin with some general results in Section \ref{sec:HKMMDiracGeneral} and then apply these to the cases $N=2$ and $N=3$ for illustration in Section \ref{sec:HKMMDiracExamples}.

\subsection{General results}
\label{sec:HKMMDiracGeneral}
The two main results of this section are Propositions \ref{prop:HKMMDirac} and \ref{prop:HKMMDiracDual}.
Roughly speaking, the former provides sufficient conditions for $\HK$ barycenters between Dirac measures and the latter, by leveraging duality between $\cMMHull$ and $\QMM$, necessary conditions.

The following lemma helps to reduce the complexity of the multi-marginal problem in the case that some of the marginals are Dirac measures. This will make it easier in the following to solve the `remaining problem'.
\begin{lemma}
  \label{lem:HKBarycenterSupportRestriction}
  Let $N$ measures $\mu_1,\ldots,\mu_N \in \measp(\Omega)$ be given, among which the first $I\le N$ are Diracs,
  i.e., $\mu_i=\bar m_i\cdot\delta_{\bar x_i}$, with $\bar x_i\in\Omega$ and $\bar m_i\in\R_+$, for $i=1,\ldots,I$.
  Then the minimum in the definition \eqref{eq:HKMMLifted} of $\HKMM$ is attained by some $\gamma\in\measpt(\cone^N)$
  that are concentrated on $\Theta\times\cone^{N-I}$,
  with $\Theta:=\prod_{i=1}^I\big(\{\bar x_i\}\times\R_+\big)\subset \cone^I$.

  Likewise, a minimum in formulation \eqref{eq:HKMMSemiCoupling} is attained by some $\gamma,\gamma_1,\ldots,\gamma_N \in \measp(\Omega)^N$ that are concentrated on $\left(\prod_{i=1}^I \{\bar{x}_i\}\right) \times \Omega^{N-I}$.
  Further, these arguments apply to the formulations for $\HK$, \eqref{eq:HKLifted} and \eqref{eq:HKSemiCoupling}.
\end{lemma}
\newcommand{\reducmap}{U}
\begin{proof}
  Let $\gamma\in\measpt(\cone^N)$ be a minimizer in \eqref{eq:HKMMLifted} where we substituted $\cMMHull$ by $\cMM$, as allowed by Theorem \ref{thm:HKCTMMMEquiv} \eqref{item:HKCTMMMEquivNonRelax}.
  Define $\reducmap:\cone^N\to\cone^N$ by
  \begin{align*}
    \reducmap(x_1,m_1,\ldots,x_I,m_I,x_{I+1},m_{I+1},\ldots,x_N,m_N)
    = (\bar x_1,m_1,\ldots,\bar x_I,m_I, x_{I+1},m_{I+1},\ldots,x_N,m_N).
  \end{align*}
  Clearly, $\reducmap_\sharp\gamma$ is concentrated on $\Theta\times\cone^{N-I}$.
  Since $\gamma$ is feasible in the minimization problem \eqref{eq:HKMMLifted},
  we have that $\coneProj \pi_{i\sharp} \gamma = \bar m_i\delta_{\bar x_i}$ for all $i=1,\ldots,I$,
  and so, $\gamma$-almost every $(x_1,m_1,\ldots,x_N,m_N)\in\cone^N$
  satisfies
  \begin{align}
    \label{eq:special}
    x_i=\bar x_i\quad\text{or}\quad m_i=0\qquad\text{for each $i=1,\ldots,I$}.
  \end{align}
  Therefore, $\coneProj \pi_{i\sharp} \gamma=\coneProj \pi_{i\sharp} \reducmap_\sharp \gamma$ for $i =1,\ldots,N$ and so $\reducmap_\sharp\gamma$ is an admissible competitor in \eqref{eq:HKMMLifted}.
  So the claim follows if we can show that $\int_{\Omega^N}\cMM\,\diff \reducmap_\sharp\gamma \le \int_{\Omega^N}\cMM \,\diff\gamma$.
  For the latter, it suffices to show that  $\cMM\circ \reducmap\le\cMM$ $\gamma$-almost everywhere.
  We are now going to prove that any $z\in\cone^N$ with the property \eqref{eq:special} satisfies $\cMM(\reducmap(z))=\cMM(z)$.
  For arbitrary $y\in\Omega$ and $s\in\R_+$, we have that $c(x_i,m_i,y,s)=c(\bar x_i,m_i,y,s)$ for $i=1,\ldots,I$,
  since unless $x_i=\bar x_i$, in which case the arguments on both sides are identical,
  one has that $m_i=0$, and both sides are equal to $s$.
  It now follows directly from the definition \eqref{eq:CMM} of $\cMM$ that $\cMM(\reducmap(z))=\cMM(z)$.
  
  By using the construction used in the equivalence proof of Theorem \ref{thm:HKMMSemiCoupling} (see also Remark \ref{rem:HKLiftedSCEquiv}), one can construct the required minimizers for the statement about \eqref{eq:HKMMSemiCoupling}.
\end{proof}

\begin{proposition}
\label{prop:HKMMDirac}
Let $\vec{x} \in \Omega^N$, $\vec{m} \in \R_+^N$ and set $\mu_i=m_i \cdot \delta_{x_i}$ for $i \in \{1,\ldots,N\}$.
\begin{enumerate}[(i)]
	\item Then $\HKMM(\mu_1,\ldots,\mu_N) = \cMMHull(\vec{x},\vec{m})$.
	\label{item:HKMMDiracCMMHull}
	\item Assume $\vec{m} \neq 0$. Then, if and only if $(\vec{x},\vec{m}) \in \CSet$ a $\HK$ barycenter between the measures $(\mu_1,\ldots,\mu_N)$ (for the weights $(\lambda_1,\ldots,\lambda_N)$) is a single Dirac $\nu=s \cdot \delta_y$. In this case $(y,s)=\HKT(\vec{x},\vec{m})$.
	\label{item:HKMMDiracSingle}
	\item More generally, given a decomposition
	\begin{align*}
		\vec{m} = \sum_{j=1}^J \vec{m}^j,
		\quad \vec{m}^j \in \R_+^N
		\tn{ for } j=1,\ldots,J,
		\quad \tn{such that } \cMMHull(\vec{x},\vec{m})=\sum_{j=1}^J \cMM(\vec{x},\vec{m}^j)
	\end{align*}
	where $J \in \{1,\ldots,N\}$ is a suitable integer, then $(\vec{x},\vec{m}^j) \in \CSet_0$ for $j \in \{1,\ldots,J\}$ and $\sum_{j=1}^J s_j \cdot \delta_{y_j}$ is a barycenter where $(y_j,s_j)=\HKT(\vec{x},\vec{m}^j)$ if $\vec{m}^j \neq 0$, and $s_j=0$ otherwise.
	Existence of such a decomposition for any $\vec{m}$ is provided by Lemma \ref{lem:CMMHullCaratheodory}.
	\label{item:HKMMDiracSplit}
\end{enumerate}
\end{proposition}
\begin{proof}
	\textbf{\eqref{item:HKMMDiracCMMHull}:}
	We apply Lemma \ref{lem:HKBarycenterSupportRestriction} with $I=N$.
	This means we may restrict the support of the $\gamma_i$ and $\gamma$ in \eqref{eq:HKMMSemiCoupling} to $\{\vec{x}\}$ and therefore by the marginal constraints must have $\gamma_i = m_i \cdot \delta_{\vec{x}}$. W.l.o.g.~we may choose $\gamma=\delta_{\vec{x}}$ and so we find by \eqref{eq:HKMMSemiCoupling}
	\begin{align*}
		\HKMM(\mu_1,\ldots,\mu_N)^2 = 
		\int_{\Omega^N}
		\cMMHull\big(z_1,\RadNik{\gamma_1}{\gamma},\ldots,z_N,\RadNik{\gamma_N}{\gamma}\big)
		\diff \gamma(z_1,\ldots,z_N)
		= \cMMHull(\vec{x},\vec{m}).
	\end{align*}

	\textbf{\eqref{item:HKMMDiracSingle}:} Assume now $(\vec{x},\vec{m}) \in \CSet$. Plugging $\gamma \assign \delta_{((x_1,m_1),\ldots,(x_N,m_N))}$ into \eqref{eq:HKMMLifted} we find that it is optimal by \eqref{item:HKMMDiracCMMHull}. Since by construction $\gamma$ is concentrated on $\CSet$, by Theorem \ref{thm:HKCTMMMEquiv} \eqref{item:HKCTMMMEquivProject} $\nu \assign \coneProj \HKT_{\sharp} \gamma$ is a $\HK$ barycenter. One finds $\HKT_{\sharp} \gamma=\delta_{(y,s)}$ with $(y,s) = \HKT(\vec{x},\vec{m})$ and subsequently $\nu=s \cdot \delta_y$. Thus, there exists a single-Dirac barycenter.
	
	Conversely, let now $(\vec{x},\vec{m}) \notin \CSet$, $\vec{m} \neq 0$. Then, for any $\nu \assign s \cdot \delta_y$ we have $\HK(\mu_i,\nu)^2 = c(x_i,m_i,y,s)$ and plugging $\nu$ into \eqref{eq:HKCTM} one obtains
	\begin{align*}
		\sum_{i=1}^N \lambda_i c(x_i,m_i,y,s) \geq \cMM(\vec{x},\vec{m}) > \cMMHull(\vec{x},\vec{m}) = \HKMM(\mu_1,\ldots,\mu_N)
	\end{align*}
	and so $\nu$ cannot be optimal in \eqref{eq:HKCTM} and thus cannot be a $\HK$ barycenter.
	
	\textbf{\eqref{item:HKMMDiracSplit}:} Let $(\vec{m}^1,\ldots,\vec{m}^J)$ be such a decomposition of $\vec{m}$.
	Since by construction
	\begin{align*}	
	\cMMHull(\vec{x},\vec{m}) \leq \sum_{j=1}^J \cMMHull(\vec{x},\vec{m}^j)
	\quad \tn{and} \quad
	\cMMHull(\vec{x},\vec{m}^j) \leq \cMM(\vec{x},\vec{m}^j)
	\end{align*}
	one finds that $\cMMHull(\vec{x},\vec{m}^j) = \cMM(\vec{x},\vec{m}^j)$ and so $(\vec{x},\vec{m}^j) \in \CSet_0$.
	Therefore, by \eqref{item:HKMMDiracSingle} $\nu_j \assign s_j \cdot \delta_{y_j}$ with $(y_j,s_j) \assign \HKT(\vec{x},\vec{m}^j)$  is a HK barycenter for $(\vec{x},\vec{m}^j)$ if $\vec{m}^j \neq 0$. If $\vec{m}^j = 0$, the barycenter $\nu_j$ is trivially $0$.
	Since the function
	\begin{align*}
		(\mu_1,\ldots,\mu_N,\nu) \mapsto \sum_{i=1}^N \lambda_i\,\HK(\mu_i,\nu)^2
	\end{align*}
	is subadditive in its arguments, we thus find that $\nu \assign \sum_{j=1}^J \nu_j$ is a feasible candidate for the $\HK$ barycenter between $(\mu_1,\ldots,\mu_N)$ in \eqref{eq:HKCTM} with score at most $\sum_{j=1}^J \cMM(\vec{x},\vec{m}^j) = \cMMHull(\vec{x},\vec{m})$ and thus $\nu$ must be optimal by \eqref{item:HKMMDiracCMMHull}.
\end{proof}

Proposition \ref{prop:HKMMDiracDual} leverages duality between $\cMMHull$ and $\QMM$ and requires a characterization of the normal cone of $\QMM(\vec{x})$ (or equivalently, the subdifferential of $\iota_{\QMM(\vec{x})}$). As $\QMM(\vec{x})$ is described as an uncountable intersection of sets, see \eqref{eq:MMAuxMMRelation} and \eqref{eq:QMM}, this characterization is not entirely trivial and the following Lemma takes care of this technical aspect.
\begin{lemma}
	\label{lem:DropInactiveConstraints}
	Let $\vec{x} \in \Omega^N$, $\vec{m} \in \R_{++}^N$ and let $\vec{\psi}$ be the gradient of $\cMMHull(\vec{x},\cdot)$ at $\vec{m}$, cf.~Lemma \ref{lem:PsiUniqueness}.
	Let further
	\begin{align}
		\label{eq:ActiveConstraintSet}
		Y \assign \left\{y \in \Omega \middle|
			\sum_{i=1}^N \frac{\lambda_i\,\Cos(|x_i-y|)^2}{1-\psi_i/\lambda_i} = 1
			\right\},
		\qquad
		\QMM^Y \assign \bigcap_{y \in Y} \QMM(\vec{x},y).
	\end{align}
	
	\begin{enumerate}[(i)]
	\item Then $\cMMHull(\vec{x},\vec{m})=\iota_{\QMM^Y}^\ast(\vec{m})$ and $\vec{\psi}$ is also maximal for this conjugation, i.e.~$\vec{m}\!\in\!\partial \iota_{\QMM^Y}(\vec{\psi})$.
	\item If $Y$ is finite, then
	\begin{align}
		\label{eq:NormalConeSum}
		\partial \iota_{\QMM^Y}(\vec{\psi}) = \sum_{y \in Y} \partial \iota_{\QMM(\vec{x},y)}(\vec{\psi})
	\end{align}
	where the sum denotes the Minkowski sum for sets.
	\end{enumerate}
\end{lemma}
\begin{proof}
	\textbf{Step 1: Preparation.}
	Throughout this proof let
	\begin{align*}
		f & : \R^N \times \Omega \to \RCupInf, &
		(\vec{\xi},y) \mapsto \begin{cases}
			+ \infty & \tn{if } \xi_i > \lambda_i \tn{ for some } i=1,\ldots,N, \\
			\sum_{i=1}^N \frac{\lambda_i\,\Cos(|x_i-y|)^2}{1-\psi_i/\lambda_i} & \tn{else}
			\end{cases}
	\end{align*}
	where we use convention \eqref{eq:QMMAuxConvention} for the situation when $\xi_i=\lambda_i$.
	$\vec{\xi} \in \QMM(\vec{x},y)$ is equivalent to $f(\vec{\xi},y)\leq 1$, otherwise we say that $\vec{\xi}$ violates the constraint at $y$.
	Further, let $h(\vec{\xi}) \assign \sum_{i=1}^N m_i\,\xi_i$.
	
	\textbf{Step 2: Contradiction.}
	Assume that some $\vec\psi'$ with $h(\vec\psi')>h(\vec\psi)$ satisfies $f(\vec\psi',y)\leq 1$ for all $y\in Y$.
	We may assume that $\psi'_i<\lambda_i$, since otherwise we can slightly decrease $\psi'_i$ such that one still has $h(\vec\psi')>h(\vec\psi)$.
	We are going to derive a contradiction to the maximality of $\vec\psi$
	by constructing a $\vec\psi''$ with $h(\vec\psi'')>h(\vec\psi)$ that satisfies $f(\vec\psi'',y)\le1$ even at all $y\in\Omega$.
	
	First, fix some $\delta>0$ sufficiently small such that $\vec\psi^\delta:=(\psi'_1-\delta,\ldots,\psi'_N-\delta)$
	still has $h(\vec\psi^\delta)>h(\vec\psi)$.
	By strict monotonicity of each $\chi_i: \psi_i \mapsto \tfrac{\lambda_i}{1-\psi_i/\lambda_i}$ on $(-\infty,\lambda_i)\to\R_{++}$,
	one has $\kappa:=\max_i\frac{\chi_i(\psi'_i-\delta)}{\chi_i(\psi_i')}<1$,
	and since the functions $\Cos^2$ are non-negative,
	it follows that $f(\vec\psi^\delta,y)\le\kappa f(\vec\psi',y)$ for all $y\in\Omega$,
	and in particular, $f(\vec\psi^\delta,y)\le\kappa$ for all $y\in Y$.
	Since the $\Cos^2$ are globally Lipschitz-continuous with Lipschitz constant one,
	it follows that
	\begin{align}
		\label{eq:ProofNormalConefBound}
		f(\vec\psi^\delta,y)\le1
		\quad\text{for all}\quad
		y\in Y^\delta \assign \left\{y''\in \Omega \middle|\inf_{y'\in Y}|y''-y'|<1-\kappa\right\}.
	\end{align}
	To finish the construction, observe that $\Omega\ni y\mapsto f(\vec\psi,y)$ is a continuous map,
	and recall that $f(\vec\psi,y)<1$ for all $y\in\Omega\setminus Y$ by definition of $Y$,
	to conclude that the supremum of $f(\vec\psi,\cdot)$ on the compact set $\Omega\setminus Y^\delta$ is less than one.
	Now define $\vec\psi'':=(1-\alpha)\vec\psi+\alpha\vec\psi^\delta$
	with $\alpha\in(0,1)$ so small that $f(\vec\psi'',y)\le1$ for all $y\in\Omega\setminus Y^\delta$;
	here we use that $f(\cdot,y)$ is a convex function.
	In view of \eqref{eq:ProofNormalConefBound}, it follows that $F(\vec\psi',y)\le1$ at all $y\in\Omega$,
	and trivially, $h(\vec\psi'')=(1-\alpha)h(\vec\psi)+\alpha h(\vec\psi^\delta)>h(\vec\psi)$.
	This is the sought for contradiction and $\vec{\psi}$ must be optimal with respect to $\QMM^Y$.

	\textbf{Step 3: Subdifferential.}
	The second part of the Lemma is a classical result from convex analysis, see e.g.~\cite[Theorem 23.8]{Rockafellar1972Convex} which applies since $(-\infty,0]^N \subset \QMM(\vec{x},y)$ for all $y \in \Omega$ and thus the functions $\iota_{\QMM(\vec{x},y)}$ have a common point in the (relative) interior of their domain.
\end{proof}

\begin{proposition}
	\label{prop:HKMMDiracDual}
	Let $\vec{x} \in \Omega^N$, $\vec{m} \in \R_{++}^N$, let $\vec{\psi}$ be the gradient of $\cMMHull(\vec{x},\cdot)$ at $\vec{m}$, cf.~Lemma \ref{lem:PsiUniqueness}, and assume that the set $Y$ from Lemma \ref{lem:DropInactiveConstraints} \eqref{eq:ActiveConstraintSet} is finite.
	\begin{enumerate}[(i)]
		\item Then there exists a decomposition
		\begin{align}
			\label{eq:MNormalConeDecomposition}
			\vec{m} = \sum_{y \in Y} \vec{m}^y,
			\qquad
			\vec{m}^y \in \partial \iota_{\QMM(\vec{x},y)}(\vec{\psi})
			\tn{ for } y \in Y
		\end{align}
		such that
		\begin{align}
		\label{eq:MNormalConeDecompositionApplication}
		\cMMHull(\vec{x},\vec{m})= \sum_{y \in Y} \cMM(\vec{x},\vec{m}^y,y)
		\quad \tn{and} \quad
		\cMM(\vec{x},\vec{m}^y,y) = \cMM(\vec{x},\vec{m}^y) = \cMMHull(\vec{x},\vec{m}^y).
		\end{align}
		\label{item:HKMMDiracDualDecomp}
		\item For this decomposition one finds $\sum_{y \in Y} s^y\,\delta_{y}$ is a $\HK$ barycenter between $\mu_1,\ldots,\mu_N$ with $(y,s^y) = \HKT(\vec{x},\vec{m}^y)$ for $y \in Y$.
		\label{item:HKMMDiracDualBarycenter}
		\item There are no $\HK$ barycenters between $\mu_1,\ldots,\mu_N$ that are supported outside of $Y$.
		This also holds when $Y$ is not finite.
		\label{item:HKMMDiracDualSupport}
		\item Any $\HK$ barycenter between $\mu_1,\ldots,\mu_N$ can be written in the way described in \eqref{item:HKMMDiracDualDecomp} and \eqref{item:HKMMDiracDualBarycenter}.
		\label{item:HKMMDiracDualNecessary}		
		\item If the subdifferentials $\partial \iota_{\QMM(\vec{x},y)}(\vec{\psi})$ for $y \in Y$ are linearly independent the barycenter is unique.
		\label{item:HKMMDiracDualUnique}
		\item Every $\mu_i=m_i \cdot \delta_{x_i}$ (where $m_i>0$ by assumption) is (partially) transported to some (possibly multiple) non-zero $s^y \cdot \delta_y$ in the $\HK$ barycenter, i.e.~no mass particle `vanishes completely'.
		\label{item:HKMMDiracDualNonZero}
	\end{enumerate}
\end{proposition}
\begin{proof}
	\textbf{\eqref{item:HKMMDiracDualDecomp}:}
	By Lemma \ref{lem:DropInactiveConstraints} we have $\vec{m} \in \partial \iota_{\QMM^Y}(\vec{\psi})$ and
	\begin{align*}
		\partial \iota_{\QMM^Y}(\vec{\psi}) = \sum_{y \in Y} \partial \iota_{\QMM(\vec{x},y)}(\vec{\psi})
	\end{align*}
	thus decomposition \eqref{eq:MNormalConeDecomposition} exists. The relation $\vec{m}^y \in \partial \iota_{\QMM(\vec{x},y)}(\vec{\psi})$ implies (cf.~\eqref{eq:PsiMSubDiffRelation})
	$$\cMM(\vec{x},\vec{m}^y,y)=\sum_{i=1}^N \psi_i\,m_i^y$$
	and so
	\begin{multline*}
		\cMMHull(\vec{x},\vec{m})=\sum_{i=1}^N \psi_i\,m_i
		= \sum_{y \in Y} \sum_{i=1}^N \psi_i\,m_i^y
		= \sum_{y \in Y} \cMM(\vec{x},\vec{m}^y,y) \\
		\geq \sum_{y \in Y} \cMM(\vec{x},\vec{m}^y)
		\geq \sum_{y \in Y} \cMMHull(\vec{x},\vec{m}^y)
		\geq \cMMHull(\vec{x},\vec{m}).
	\end{multline*}		
	Thus, all inequalities must be equalities, and equality must also hold for each term in the sums separately. This proves \eqref{eq:MNormalConeDecompositionApplication}.
	
	\textbf{\eqref{item:HKMMDiracDualBarycenter}:}	The left part of \eqref{eq:MNormalConeDecompositionApplication} establishes that $(\vec{x},\vec{m}^y) \in \CSet$ and that $\HKT(\vec{x},\vec{m}^y)=(y,s^y)$ for some $s^y \in \R_+$.
	The fact that $\sum_{y \in Y} s^y \delta_y$ is a barycenter then follows from Proposition \ref{prop:HKMMDirac} \eqref{item:HKMMDiracSplit}.
	
	\textbf{\eqref{item:HKMMDiracDualSupport}:} We will prove this by a complimentary slackness argument. Let $\mu_i \assign m_i \cdot \delta_{x_i}$ and let $\nu \in \measp(\Omega)$ some non-negative measure. From Lemma \ref{lem:HKBarycenterSupportRestriction} we find that in \eqref{eq:HKSemiCoupling} for $\HK(\mu_i,\nu)^2$ we can restrict the support of $\gamma_1, \gamma_2$ and $\gamma$ to $\{x_i\} \times \Omega$. This implies that $\gamma_1=\delta_{x_i} \otimes \rho_i$ for some $\rho_i \in \measp(\Omega)$ with $\|\rho_i\|=m_i$ and $\gamma_2 = \delta_{x_i} \otimes \nu$ and thus one finds
	\begin{align*}
		\HK(\mu_i,\nu)^2 = \min
			\left\{ \int_{\Omega} c\big(x_i,\RadNik{\rho_i}{\gamma}(y),y,\RadNik{\nu}{\gamma}(y)\big)
			\diff \gamma(y)
			\middle|
			\rho_i,\gamma \in \measp(\Omega),\,\|\rho_i\|=m_i,\, \rho_i,\nu \ll \gamma
			\right\}.
	\end{align*}
	Plugging this into \eqref{eq:HKCTM} we get
	\begin{multline*}
		\HKCTM(\mu_1,\ldots,\mu_N)^2 = \min\left\{
			\int_\Omega \sum_{i=1}^N \lambda_i\,
				c\big(x_i,\RadNik{\rho_i}{\gamma}(y),y,\RadNik{\nu}{\gamma}(y)\big)\,\diff \gamma(y)
				\right| \\
				\left. \vphantom{\sum_{i=1}^N}
				\rho_1,\ldots,\rho_N,\nu,\gamma \in \measp(\Omega),\,
				\|\rho_i\|=m_i \tn{ for } i\in\{1,\ldots,N\},\,
				\rho_1,\ldots,\rho_N,\nu \ll \gamma
				\right\}.
	\end{multline*}
	Let now $\rho_1,\ldots,\rho_N,\nu,\gamma$ be a minimizer. With Theorem \ref{thm:HKCTMMMEquiv} and Proposition \ref{prop:HKMMDirac} we find
	\begin{align*}
		\cMMHull(\vec{x},\vec{m}) & = \HKMM(\mu_1,\ldots,\mu_N)^2 = \HKCTM(\mu_1,\ldots,\mu_N)^2 \\
		& = \int_\Omega \sum_{i=1}^N \lambda_i\,
				c\big(x_i,\RadNik{\rho_i}{\gamma}(y),y,\RadNik{\nu}{\gamma}(y)\big)\,\diff \gamma(y)
		\geq \int_\Omega \cMM\big(\vec{x},\RadNik{\vec{\rho}}{\gamma}(y),y\big)\,\diff \gamma(y)
		\intertext{(where $\RadNik{\vec{\rho}}{\gamma}(y)=(\RadNik{\rho_1}{\gamma}(y),\ldots,\RadNik{\rho_N}{\gamma}(y))$ is the vector of densities and we recall definition \eqref{eq:CMMAux})}
		& \geq \int_\Omega \left[ \sum_{i=1}^N \psi_i \cdot \RadNik{\rho_i}{\gamma}(y) \right]\,\diff \gamma(y)
		= \sum_{i=1}^N \psi_i \cdot m_i = \cMMHull(\vec{x},\vec{m})
	\end{align*}
	where we have used $\vec{\psi} \in \QMM(\vec{x}) \subset \QMM(\vec{x},y)$ for all $y \in \Omega$, see \eqref{eq:MMAuxMMRelation}.
	Once more, this implies that all inequalities are actually equalities and that the integrands must agree $\gamma$-almost everywhere. In particular
	\begin{align*}
		\cMM\big(\vec{x},\RadNik{\vec{\rho}}{\gamma}(y),y\big)=\sum_{i=1}^N \psi_i \cdot \RadNik{\rho_i}{\gamma}(y)
		\qquad \tn{$\gamma(y)$-almost everywhere}
	\end{align*}
	which means $\RadNik{\vec{\rho}}{\gamma}(y) \in \partial \iota_{\QMM(\vec{x},y)}(\vec{\psi})$ $\gamma(y)$-almost everywhere.
	Since $\partial \iota_{\QMM(\vec{x},y)}(\vec{\psi}) = \{0\}$ if $\vec{\psi}$ lies in the interior of $\QMM(\vec{x},y)$, which in turn is equivalent to $y \notin Y$, cf.~\eqref{eq:ActiveConstraintSet}, we find that $\RadNik{\vec{\rho}}{\gamma}$ is zero $\gamma$-almost everywhere outside of $Y$.
	Finally, by direct computation we find that the optimal $\nu$ must satisfy [$\RadNik{\vec{\rho}}{\gamma}=0$] $\Rightarrow$ [$\RadNik{\nu}{\gamma}=0$] $\gamma$-almost everywhere (since $s=0$ is the unique minimizer in \eqref{eq:CMMAux} for $m_1=\ldots=m_N=0$).
	Therefore, the barycenter $\nu$ must be concentrated on $Y$.
	
	\textbf{\eqref{item:HKMMDiracDualNecessary}:}
	We continue the argument from the previous part. The minimal $\rho_i$ must be concentrated on $Y$ and thus they can be written as $\rho_i = \sum_{y \in Y} m_i^y\,\delta_y$ and the constraint $\|\rho_i\|=m_i$ becomes $\sum_{y \in Y} m_i^y=m_i$. Likewise, the optimal $\nu$ can be written as $\sum_{y \in Y} s^y\,\delta_{y}$.
	The corresponding objective for $\HKCTM(\mu_1,\ldots,\mu_N)^2$ then is $\sum_{y \in Y} \sum_{i =1}^N \lambda_i\,c(x_i,m_i^y,y,s^y)$, which me minimize over non-negative decompositions $\vec{m}=\sum_{y \in Y} \vec{m}^y$ and masses $s^y$, $y \in Y$.
	Clearly, the $s^y$ must be minimizers for $\cMM(\vec{x},\vec{m}^y,y)$ in \eqref{eq:CMMAux} so that the objective becomes $\sum_{y \in Y} \cMM(\vec{x},\vec{m}^y,y)$.
	By Proposition \ref{prop:HKMMDirac} \eqref{item:HKMMDiracCMMHull} this implies that the optimal decomposition satisfies the left part of \eqref{eq:MNormalConeDecompositionApplication}, which implies the right part by Lemma \ref{lem:CMMHullCaratheodory}, which then, by convex analysis, implies the right part of \eqref{eq:MNormalConeDecomposition}.

	\textbf{\eqref{item:HKMMDiracDualUnique}:}
	If the subdifferentials are linearly independent the decomposition $\vec{m}=\sum_{y \in Y} \vec{m}^y$ is unique. Since, by \eqref{item:HKMMDiracDualNecessary} any barycenter can be written as such a decomposition, this must be the unique barycenter.
	
	\textbf{\eqref{item:HKMMDiracDualNonZero}:} For every $i \in \{1,\ldots,N\}$, $m_i^y>0$ for at least one $y \in Y$. By Proposition \ref{prop:CMMMinimizerUnique} $|y-x_i|<\pi/2$ and therefore by explicit computation (cf.~\eqref{eq:CMMAux}), $s_y>0$.
\end{proof}

\subsection{Examples and illustrations}
\label{sec:HKMMDiracExamples}
In this subsection we assume $\vec{x} \in \Omega^N$, $\vec{m} \in \R_{++}^N$ and $\vec{\psi} \in \QMM(\vec{x})$ is the unique maximizer of $\vec{\xi} \mapsto \sum_{i=1}^N \xi_i\,m_i$ over $\QMM(\vec{x})$, and hence the gradient of $\cMMHull(\vec{x},\cdot)$ at $\vec{m}$, cf.~Lemma \ref{lem:PsiUniqueness}.
Similar to the proof of Lemma \ref{lem:DropInactiveConstraints}, for $\vec{\xi} \in \R^N$, $\xi_i \in (-\infty,\lambda_i-\lambda_i^2]$, throughout this subsection we will use the function
\begin{align*}
	f(\vec{\xi},\cdot) & : \Omega \to \R, &
	y \mapsto \sum_{i=1}^N \frac{\lambda_i\,\Cos(|x_i-y|)^2}{1-\xi_i/\lambda_i}
\end{align*}
to study the set $\QMM(\vec{x})$.
By the assumption $\vec{\psi} \in \QMM(\vec{x})$, which implies $\psi_i \leq \lambda_i-\lambda_i^2$, the function $\vec{\xi} \mapsto f(\vec{\xi},y)$ is smooth at $\vec{\psi}$ for all $y \in \Omega$. This implies that for $f(\vec{\psi},y)=1$ the normal cone of $\QMM(\vec{x},y)$ at $\vec{\psi}$, which is equal to $\partial \iota_{\QMM(\vec{x},y)}(\vec{\psi})$, is given by
\begin{align}
	\label{eq:QMMyNormalCone}
	\partial \iota_{\QMM(\vec{x},y)}(\vec{\psi})
	= \R_+ \cdot \nabla_{\vec{\psi}} f(\vec{\psi},y)
	= \R_+ \cdot \left( \frac{\Cos(|x_1-y|)^2}{(1-\psi_1/\lambda_1)^2},\ldots,\frac{\Cos(|x_N-y|)^2}{(1-\psi_N/\lambda_N)^2}\right).
\end{align}

We now consider the case $N=2$.
\begin{example}[$N=2$]
\label{ex:CMM:N2}
By the `convex hull argument' (see Lemma \ref{lem:CMMMinimizerExist} \eqref{item:CMMConvexHull} and Proposition \ref{prop:QMM}) we can restrict our considerations to line segment between $x_1$ and $x_2$.

\textbf{Case $|x_1-x_2|<\pi/2$.}
This case has already been discussed in Example \ref{ex:N2:1D}.
On the line segment one has $\Cos(|x_i-y|)=\cos(|x_i-y|)$, $f(\vec{\xi},\cdot)$ is a sum of trigonometric functions and so as in Proposition \ref{prop:QMMQuadratic}, $f(\vec{\xi},\cdot)$ has a unique maximum on the line segment and the proposition implies that $\QMM(\vec{x})=\lambda_1\,\lambda_2\,Q$ with $Q$ as in \eqref{eq:HKQ}, see \eqref{eq:ExampleN2QMM}.
So $f(\vec{\psi},y)=1$ for a unique $y$ and by Proposition \ref{prop:HKMMDiracDual} $(\vec{x},\vec{m}) \in \CSet$ and the unique $\HK$ barycenter is given by $s \cdot \delta_y$ with
\begin{equation}
\label{eq:ExampleT}
(y,s)=\HKT(\vec{x},\vec{m})=\left(\arg\left(\sum_{i=1}^N \lambda_i \sqrt{m_i} e^{i x_i} \right),\left|\sum_{i=1}^N \lambda_i \sqrt{m_i} e^{i x_i}\right| \right)
\end{equation}
where the last expression is due to \eqref{eq:CMMSimpleMinimizers}.
One finds that $(y,s)$ is the point on the constant speed geodesic between the two marginals at time $t=1-\lambda_1=\lambda_2$, cf.~\cite[Theorem 4.1]{ChizatOTFR2015}.

\textbf{Case $|x_1-x_2|>\pi/2$.}
The function $f(\vec{\xi},\cdot)$ has two (local) maxima at $x_1$ and $x_2$, each implying the constraints $f(\vec{\xi},x_i)\leq1$, and so $\QMM(\vec{x})=\prod_{i=1}^N (-\infty,\lambda_i-\lambda_i^2]$. Since $\vec{\psi}$ must be maximal for $\sup_{\vec{\xi} \in \QMM(\vec{x})} \sum_{i=1}^N m_i \cdot \xi_i$ we find that $\psi_i=\lambda_i-\lambda_i^2$ and $f(\vec{\psi},x_i)=1$ (and strictly smaller elsewhere).
Since $\Cos(|x_1-x_2|)=0$ we find with \eqref{eq:QMMyNormalCone},
\begin{align*}
\QMM(\vec{x},x_1) & = (-\infty,\lambda_1-\lambda_1^2] \times (-\infty,\lambda_2], &
\QMM(\vec{x},x_2) & = (-\infty,\lambda_1] \times (-\infty,\lambda_2-\lambda_2^2], \\
\partial \iota_{\QMM(\vec{x},x_1)}(\vec{\psi}) & =\R_+ \times \{0\}, &
\partial \iota_{\QMM(\vec{x},x_2)}(\vec{\psi}) & = \{0\} \times \R_+.
\end{align*}
So the decomposition implied by Proposition \ref{prop:HKMMDiracDual} \eqref{item:HKMMDiracDualDecomp} is
\begin{align*}
	\vec{m} = \sum_{i=1}^2 \vec{m}^i,
	\quad
	\vec{m}^1 = (m_1,0) \in \partial \iota_{\QMM(\vec{x},x_1)},
	\quad
	\vec{m}^2 = (0,m_2) \in \partial \iota_{\QMM(\vec{x},x_2)}
\end{align*}
and is unique.
It follows that $\cMMHull(x_1,m_1,x_2,m_2) = \cMM(\vec{x},\vec{m}^1,x_1)+\cMM(\vec{x},\vec{m}^2,x_2)$
and that the unique barycenter is given by $\sum_{i=1}^2 s_i\,\delta_{x_i}$ where $s_i = \lambda_i^2\,m_i$ is obtained by explicitly solving \eqref{eq:CMMAux} for $\cMM(\vec{x},\vec{m}^i,x_i)$ (Proposition \ref{prop:HKMMDiracDual} \eqref{item:HKMMDiracDualBarycenter}). Again, this is consistent with the $\HK$ geodesic \cite[Theorem 4.1]{ChizatOTFR2015}. We deduce by Proposition \ref{prop:HKMMDirac} \eqref{item:HKMMDiracSingle} that $(\vec{x},\vec{m}) \notin \CSet$.

\textbf{Case $|x_1-x_2|=\pi/2$.} $\psi_i=\lambda_i-\lambda_i^2$ is still optimal and $f(\vec{\psi},\cdot)=1$ on the whole line segment. Thus, $Y$ is uncountable and Proposition \ref{prop:HKMMDiracDual} does not apply immediately.
\eqref{eq:QMMyNormalCone} yields that $\partial \iota_{\QMM(\vec{x},y)}=\R_+ \cdot (\Cos(|x_1-y|)^2/\lambda_1^2,\Cos(|x_2-y|)^2/\lambda_2^2)$. Then, applying the complementary slackness argument used in the proof of Proposition \ref{prop:HKMMDiracDual} \eqref{item:HKMMDiracDualSupport} we find that we can construct $\HK$ barycenters with essentially arbitrary support on the line segment. In addition, there is some $y$ such that $\vec{m} \in \partial \iota_{\QMM(\vec{x},y)}$ and so there also exists a barycenter which is a single Dirac measure, i.e.~again $(\vec{x},\vec{m}) \in \CSet$.

In conclusion, we find that for $\vec{m} \in \R_{++}^2$, $(\vec{x},\vec{m}) \in \CSet$ is equivalent to $|x_1-x_2|\leq \pi/2$ and we have reconstructed the behaviour of $\HK$ geodesics between Dirac measures via the barycenter for $N=2$.
\end{example}

\begin{figure}[hbtp]
	\centering
	{\def\imgw{0.35\textwidth}
	\begin{tikzpicture}[img/.style={anchor=south west,inner sep=0pt},y=\imgw,x=\imgw,
		lbl/.style={inner sep=0pt},arr/.style={line width=0.5pt,black}]
	\node[img] at (0,1.3) [label=below:(a)]{\includegraphics[width=\imgw]{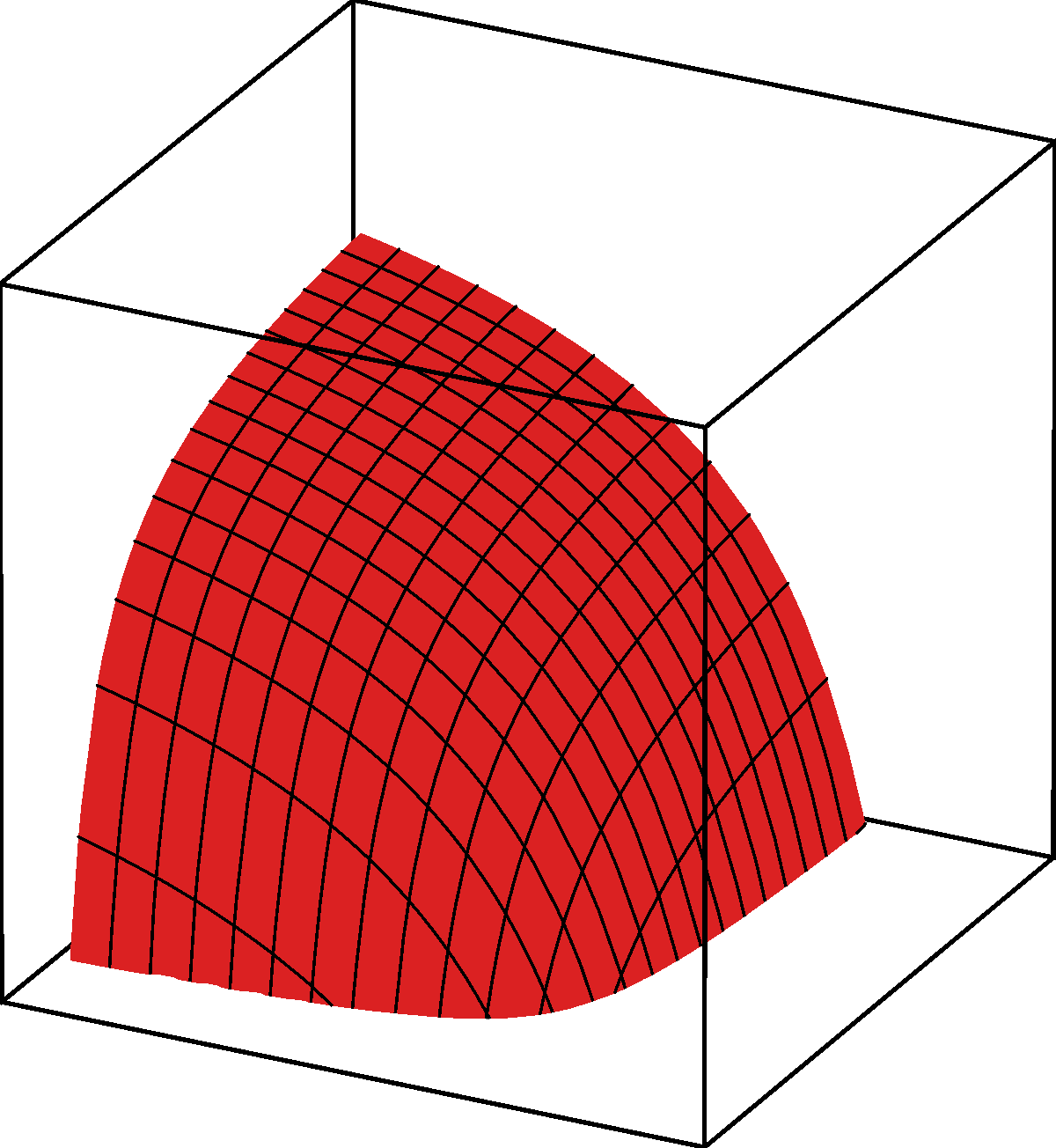}};
	\node[img] at (1.3,1.3) [label=below:(b)]{\includegraphics[width=\imgw]{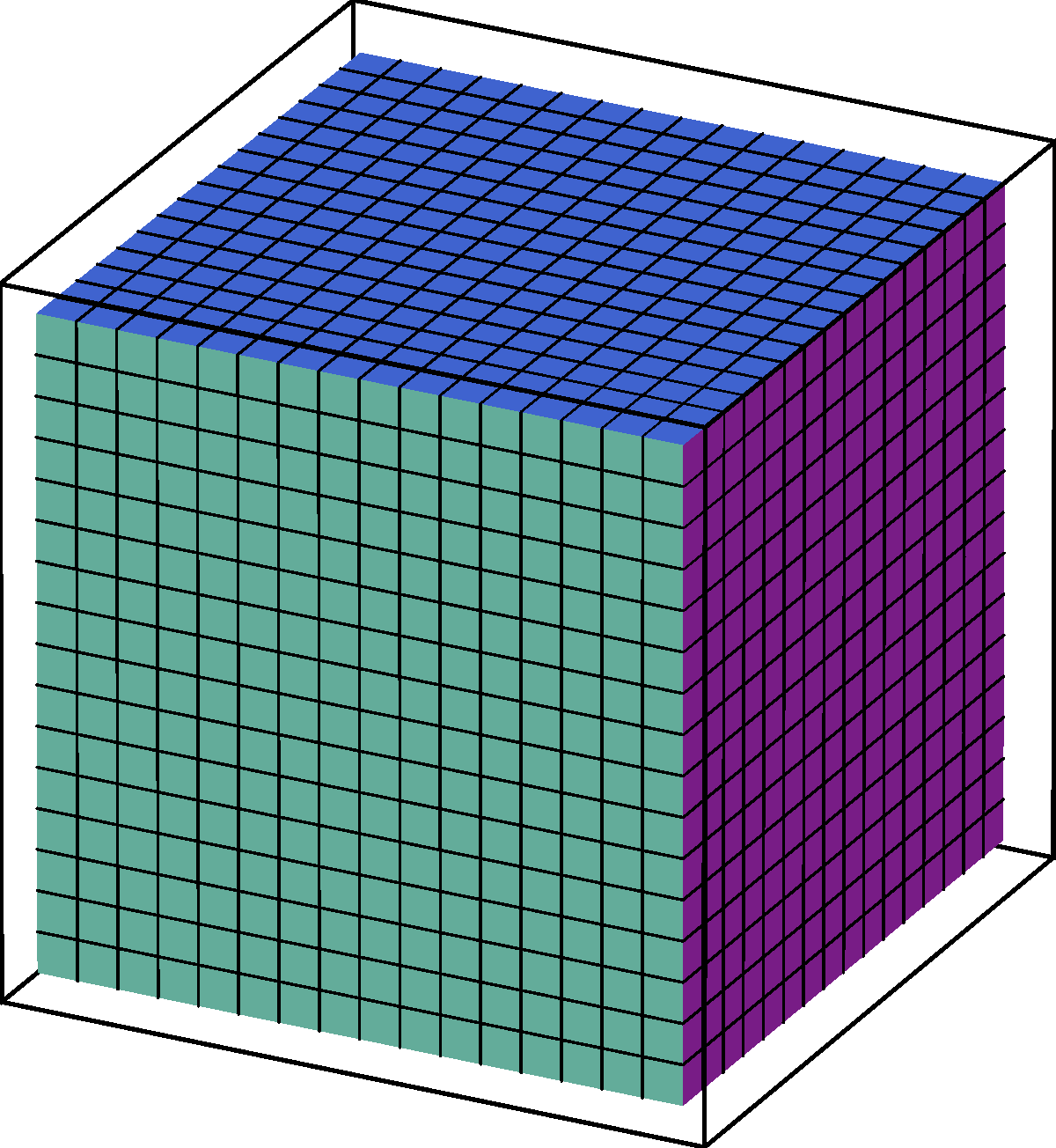}};
	\node[img] at (0,0) [label=below:(c)]{\includegraphics[width=\imgw]{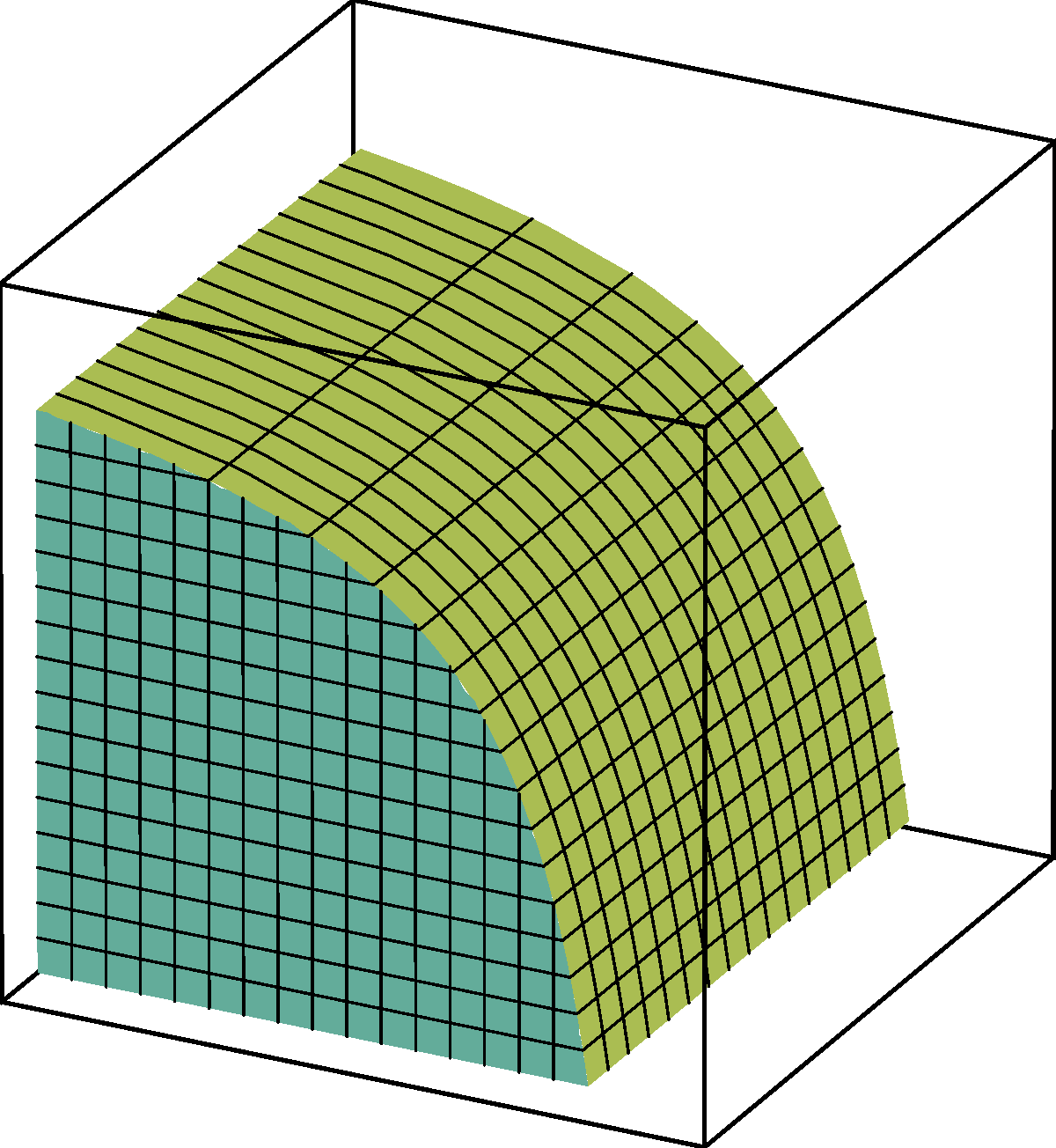}};
	\node[img] at (1.3,0) [label=below:(d)]{\includegraphics[width=\imgw]{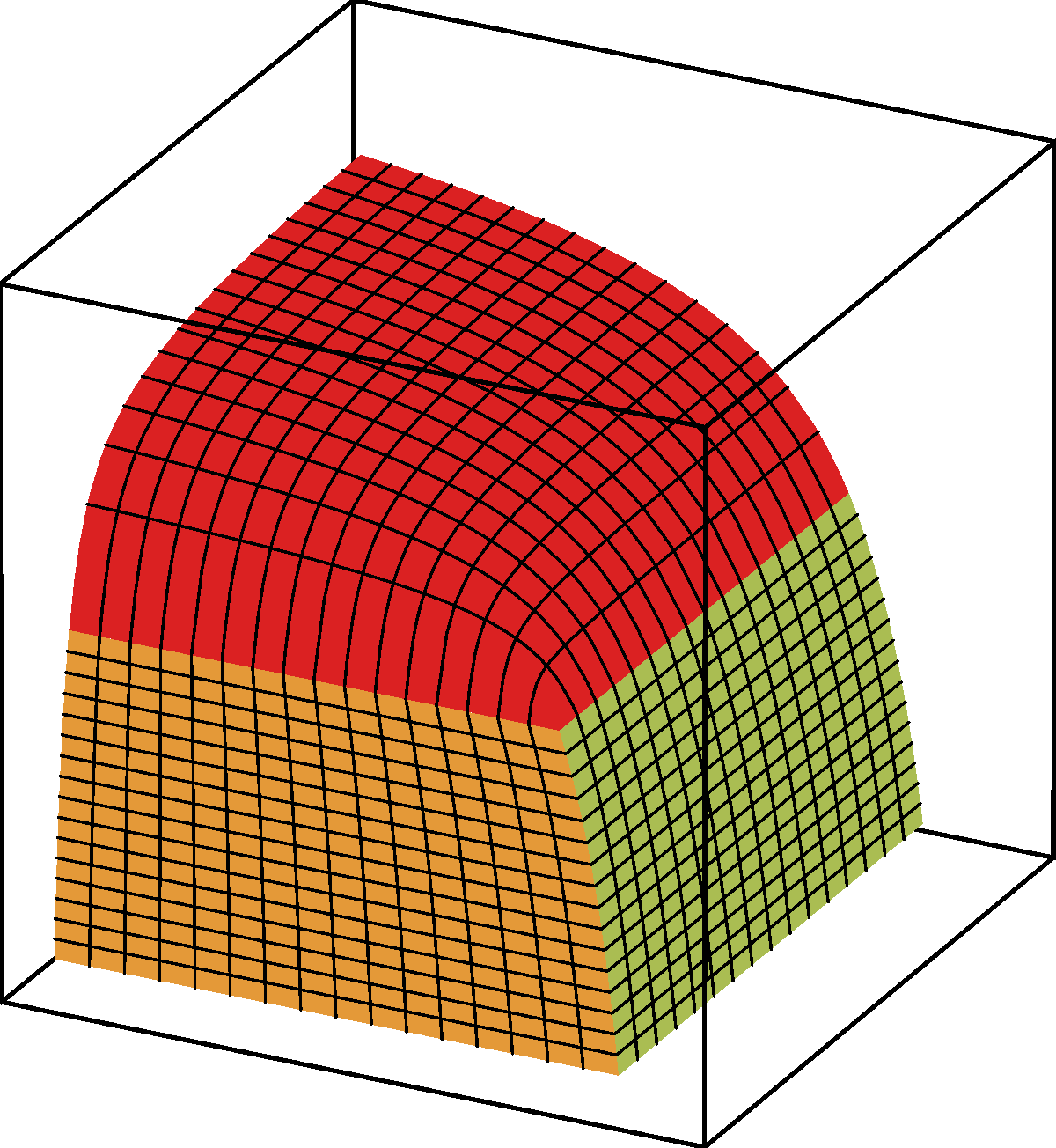}};
	\begin{scope}[shift={(0,2.3)},y=-\imgw]
		\draw[arr,->] (0.25,.940) -- ++(12:0.1);
 		\node[lbl] at (0.30,.990) []{$\psi_1$};
		\draw[arr,->] (0.9,.83) -- ++(140:0.1);
		\node[lbl] at (0.90,0.90) []{$\psi_3$};
		\draw[arr,->] (1.015,.48) -- ++(-90:0.1);
		\node[lbl] at (1.06,0.430) []{$\psi_2$};
	\end{scope}
	\end{tikzpicture}
	}
	\caption{Visualization of a part of the surface of $\QMM(\vec{x})$ for various $\vec{x}$ for $N=3$, $d=1$. The black box (approximately) represents the region $\bigtimes_{i=1}^N [0,\lambda_i-\lambda_i^2]$, orientation of the axes is given in (a).
	The cases (a) to (d) are described in Example \ref{ex:CMM:N3}.
	\newline	
	In the interior of each colored region the function $f(\vec{\psi},\cdot)$ has a unique maximum $y \in \Omega$. The location of $y$ is encoded by the color. Red: $|y-x_i|<\pi/2$ for $i=1,2,3$; purple: $|y-x_1|<\pi/2$, $|y-x_i|>\pi/2$ for $i=2,3$; analogous for blue ($x_2$) and teal ($x_3$); green: $|y-x_i|<\pi/2$ for $i=1,2$, $|y-x_3|>\pi/2$; orange: $|y-x_i|<\pi/2$ for $i=2,3$, $|y-x_1|>\pi/2$.
	\newline
	For $\vec{m} \in \R_{++}^N$ the vector $\vec{\psi} = \nabla_{\vec{m}} \cMMHull(\vec{x},\vec{m})$ is the unique point on the surface where $\vec{m}$ is normal to the surface.}
	\label{fig:QMM3D}
\end{figure}

\begin{example}[$N=3$]
\label{ex:CMM:N3}
For simplicity we restrict this discussion to one spatial dimension ($d=1$). In principle, higher dimensions can be treated analogously but the distinction between the separate cases becomes more complicated. We ignore cases with $|x_i-x_j|=\pi/2$, they can either be dealt with similarly to the $N=2$ example or the third point `breaks the symmetry'.
W.l.o.g.~we may assume $x_1 < x_2 < x_3$.
The sets $\QMM(\vec{x})$ for all cases are illustrated in Figure \ref{fig:QMM3D}.

\textbf{Case $|x_1-x_3|<\pi/2$, `all close', Figure \ref{fig:QMM3D} (a).}
The form of $\QMM(\vec{x})$ is given by Proposition \ref{prop:QMMQuadratic}, similar to the case $N=2$, $f(\vec{\xi},\cdot)$ has a unique maximum on the line segment and so $f(\vec{\psi},y)=1$ only for a single point $y$ in $(x_1,x_3)$. By Proposition \ref{prop:HKMMDiracDual} consists of a single Dirac $s \cdot \delta_y$ and is unique. By Proposition \ref{prop:HKMMDirac}, $(\vec{x},\vec{m})\in\CSet$. Applying Proposition \ref{prop:CMMSimple}, just like in \eqref{eq:ExampleT} we get
\begin{equation}
\label{eq:ExampleTThree}
(y,s)=\HKT(\vec{x},\vec{m})=\left(\arg\left(\sum_{i=1}^N \lambda_i \sqrt{m_i} e^{i x_i} \right),\left|\sum_{i=1}^N \lambda_i \sqrt{m_i} e^{i x_i}\right| \right).
\end{equation}

\textbf{Case $|x_1-x_2|>\pi/2$ $\wedge$ $|x_2-x_3|>\pi/2$, `all far', Figure \ref{fig:QMM3D} (b).}
Again, similarly to $N=2$, $f(\vec{\xi},\cdot)$ has three local maxima at $y \in \{x_1,x_2,x_3\}$ and the set $\QMM(\vec{x})$ becomes a product set, $\QMM(\vec{x})=\prod_{i=1}^N (-\infty,\lambda_i-\lambda_i^2]$.
Therefore one finds $f(\vec{\psi},y)=1$ for $y \in \{x_1,x_2,x_3\}$, i.e.~$\psi_i=\lambda_i-\lambda_i^2$ and $\vec{\psi}$ lies at the `vertex' of the three faces (see Figure).
The unique barycenter is given by $\sum_{i=1}^N s_i \delta_{x_i}$ with $s_i=\lambda_i^2\,m_i$ and $(\vec{x},\vec{m}) \notin \CSet$.

\textbf{Case $|x_1-x_2|<\pi/2$ $\wedge$ $|x_2-x_3|>\pi/2$, `two close, one far', Figure \ref{fig:QMM3D} (c).}
A careful but elementary piecewise analysis yields that $f(\vec{\xi},\cdot)$ has two maxima $y_1, y_2$, with $y_1 \in (x_1,x_2)$ and $y_2=x_3$. The two maxima can be studied separately, the former one similarly to Proposition \ref{prop:QMMQuadratic}, the latter one simply implying the constraint $f(\vec{\xi},x_3)\leq 1$. One finds $\QMM(\vec{x})=\QMM^{1,2} \times (-\infty,\lambda_3-\lambda_3^2]$ where
$$\QMM^{1,2}=\left\{ (\xi_1,\xi_2) \in \R^2 \middle|
	\xi_i \leq \lambda_i-\lambda_i^2,\, \chi_1+\chi_2-\chi_1\chi_2\sin(x_1-x_2)^2 \leq 1 \tn{ for } \chi_i=\tfrac{\lambda_i}{1-\xi_i/\lambda_i}\right\}.$$
So one has $f(\vec{\psi},y)=1$ for $y \in Y=\{y_1,y_2\}$ with $y_1 \in (x_1,x_2)$ and $y_2=x_3$, see \eqref{eq:ActiveConstraintSet}. $\vec{\psi}$ lies at the `edge' between the teal and green region of the surface (see Figure).
Any decomposition implied by Proposition \ref{prop:HKMMDiracDual} \eqref{item:HKMMDiracDualDecomp} must be of the form $\vec{m}=\vec{m}^1 + \vec{m}^2$, $\vec{m}^i \in \partial \iota_{\QMM(\vec{x},y_i)}(\vec{\psi})$ for $i=1,2$.
Since $\Cos(|x_1-y_2|)=\Cos(|x_2-y_2|)=\Cos(|x_3-y_1|)=0$, by \eqref{eq:QMMyNormalCone}, we must have $\vec{m}^1=(m_1,m_2,0)$, $\vec{m}^2=(0,0,m_3)$.
The unique barycenter is therefore given by $s_1 \delta_{y_1} + s_2 \delta_{y_2}$ with $(y_i,s_i)=\HKT(\vec{x},\vec{m}^i)$. Analogous to \eqref{eq:CMMSimpleMinimizers} one finds
\begin{align*}
	y_1 & = \arg(\lambda_1 \sqrt{m_1} e^{i x_1}+\lambda_2 \sqrt{m_2} e^{i x_2}), &
	s_1 & = |\lambda_1 \sqrt{m_1} e^{i x_1}+\lambda_2 \sqrt{m_2} e^{i x_2}|^2, \\
	y_2 & = \arg(\lambda_1 \sqrt{m_3} e^{i x_3})=x_3, &
	s_2 & = |\lambda_3 \sqrt{m_3} e^{i x_3}|^2=\lambda_3^2\,m_3,
\end{align*}
with the convention $y_1 \in [x_1,x_2]$ for the $\arg$ function in the first line.
Since the unique barycenter consists of two masses, $(\vec{x},\vec{m}) \notin \CSet$.
The case $|x_1-x_2|>\pi/2$ $\wedge$ $|x_2-x_3|<\pi/2$ is analogous.

\textbf{Case $|x_1-x_2|<\pi/2$ $\wedge$ $|x_2-x_3|<\pi/2$ $\wedge$ $|x_1-x_3|>\pi/2$, `intermediate', Figure \ref{fig:QMM3D} (d).}
This case is more complicated than the previous scenarios and the number of Dirac masses in the barycenter does not solely depend on $\vec{x}$, but also on $\vec{m}$.
We will reuse and refine arguments from Proposition \ref{prop:QMMQuadratic}.
Let $a \assign x_3-\pi/2$, $b \assign x_1+\pi/2$, $\chi_i \assign \frac{\lambda_i}{1-\xi_i/\lambda_i}$, $w_i \assign \chi_i\,e^{2i x_i}$, and $\arg: \C \setminus \{0\} \to [2x_2-\pi,2x_2+\pi)$ map a nonzero complex number to its unique complex phase in that interval.

By the convex hull argument (cf.~Proposition \ref{prop:QMM}), $f(\vec{\xi},\cdot)$ has its global (and thus a local) maximum on $[x_1,x_3]$.
On each of the intervals $[x_1,a]$, $[a,b]$, $[b,x_3]$, $f(\vec{\xi},\cdot)$ is of the form $y \mapsto \tfrac{A}{2}+ \tfrac{B}{2} \cos(2(y-\delta))$ for suitable $A,B$ and $\delta$, see \eqref{eq:fSumComplex}.
For instance, on $[a,b]$ one has $A=\sum_{i=1}^3 \chi_i$, $B=\big|\sum_{i=1}^3 w_i\big|$ and $\delta=\tfrac12 \arg(\sum_{i=1}^3 w_i)$. On $[x_1,a]$ and $[b,x_3]$ similar formulas hold where one only adds the contributions from $\{x_1,x_2\}$ and $\{x_2,x_3\}$ respectively.
In addition, $f(\vec{\xi},\cdot)$ is differentiable on $\R$.
So $f(\vec{\xi},\cdot)$ has a local maximum at some point if and only if the relevant function $y \mapsto \tfrac{A}{2}+ \tfrac{B}{2}\cos(2(y-\delta))$ has a maximum at that point.
For $w_1+w_2+w_3 \neq 0$ we find
\begin{equation*}
f(\vec{\xi},\cdot) \tn{ has local maximum on } \begin{cases}
	[x_1,a] & \tn{iff } \tfrac12 \arg(w_1+w_2) \in [x_1,a], \\
	[a,b] & \tn{iff } \tfrac12 \arg(w_1+w_2+w_3) \in [a,b], \\
	[b,x_3] & \tn{iff } \tfrac12 \arg(w_2+w_3) \in [b,x_3].
\end{cases}
\end{equation*}
Arguing geometrically via the addition of complex numbers (and using $\tfrac12 \arg(-w_3)=a$), we find
\begin{align*}
& & \big[ \tfrac12 \arg(w_1+w_2+w_3) \in [a,b] \big]
& \Rightarrow
\big[ \tfrac12 \arg(w_1+w_2) \notin [x_1,a) \big] \\
\tn{ and } & &
\big[ \tfrac12 \arg(w_1+w_2) \in [x_1,a) \big]
& \Rightarrow
\big[ \tfrac12 \arg(w_1+w_2+w_3) \notin [a,b] \big]
\big]
\end{align*}
and analogously with the roles of $w_1$ and $w_3$ swapped.
So for $w_1+w_2+w_3 \neq 0$, $f(\vec{\psi},\cdot)$ either has a unique maximum on $[a,b]$ or (at least) one local maximum on $[x_1,a) \cup (b,x_3]$.
When $w_1+w_2+w_3=0$, $f$ is maximal on the whole interval $[a,b]$ with value $\tfrac12 \sum_{i=1}^3 \chi_i$.
Arguing as in Proposition \ref{prop:QMMQuadratic} for each of the intervals $[x_1,a]$, $[a,b]$ and $[b,x_3]$, we conclude that $\vec{\chi} \in (0,1]^3$ lies in $\QMMTilde(\vec{x})$ if and only if
\begin{equation}
\label{eq:QMMTildeTransition}
\begin{cases}
\sum_{i=1}^3 \chi_i - \sum_{1 \leq j<k\leq 3} \chi_j \chi_k \sin(x_j-x_k)^2 \leq 1 & \tn{if } \sum_{i=1}^3 w_i \neq 0, \, \tfrac12 \arg(\sum_{i=1}^3 w_i) \in [a,b] \\
\begin{pmatrix}
\chi_1+\chi_2-\chi_1\chi_2\sin(x_1-x_2)^2 \leq 1 \tn{ and } \\
\chi_2+\chi_3-\chi_2\chi_3\sin(x_2-x_3)^2 \leq 1
\end{pmatrix}
& \tn{if } \sum_{i=1}^3 w_i \neq 0, \, \tfrac12 \arg(\sum_{i=1}^3 w_i) \notin [a,b] \\
\frac12\sum_{i=1}^3 \chi_i \leq 1 & \tn{if } \sum_{i=1}^3 w_i=0.
\end{cases}
\end{equation}
$\QMM(\vec{x})$ is then obtained by switching back from $\vec{\chi}$ to $\vec{\xi}$.

$f(\vec{\psi},\cdot)=1$ therefore either at a unique point in $(a,b)$, two separate points in $[x_1,a)$ and $(b,x_3]$ or on the interval $[a,b]$, corresponding to a single Dirac, two separate Diracs or a potentially diffuse barycenter.
In Figure \ref{fig:QMM3D} (d) this corresponds to the red region, the `edge' between the orange and green region, and the point where the three regions meet, respectively.

For given $\vec{m} \in \R_{++}^3$ the corresponding $\vec{\psi}$ and $\cMMHull(\vec{x},\vec{m})$ can be obtained as follows:
Assume first the barycenter were a single Dirac and compute $\cMM(\vec{x},\vec{m})$ with \eqref{eq:CMMSimple}.
Since $|x_1-x_3|<\pi$, the minimizing $y$, given by \eqref{eq:CMMSimpleMinimizers}, can be chosen in $[x_1,x_3]$ and satisfies $|x_i-y|\leq \pi$, which implies $|p_i-q|^2\geq c(x_i,m_i,y,s)$ and thus \eqref{eq:CMMSimple} gives an upper bound on $\cMMHull(\vec{x},\vec{m})$.
Now set $\vec{\psi}$ to be the $\vec{m}$-gradient of \eqref{eq:CMMSimple}, i.e.~$\psi_i=\lambda_i-\sum_{j=1}^N \lambda_i\,\lambda_j \sqrt{m_j/m_i} \cos(x_i-x_j)$ (note the regular $\cos$ instead of $\Cos$), and test if $\vec{\psi} \in \QMM(\vec{x})$ via \eqref{eq:QMMTildeTransition}.
If this is true, then we have found $\cMMHull(\vec{x},\vec{m})$ and the optimal $\vec{\psi}$ since
\begin{equation*}
	\cMMHull(\vec{x},\vec{m}) \leq \tn{\eqref{eq:CMMSimple}} = \sum_{i=1}^3 \psi_i\,m_i \leq \cMMHull(\vec{x},\vec{m})
\end{equation*}
where the equality follows from the positive 1-homogeneity of \eqref{eq:CMMSimple}. The unique barycenter is then a single Dirac $s \cdot \delta_y$ with $(y,s)$ given by \eqref{eq:ExampleTThree}.

If $\vec{\psi} \notin \QMM(\vec{x})$, assume the barycenter consists of two distinct Diracs in $[x_1,a)$ and $(b,x_3]$, such that the mass at $x_2$ is split in two. The corresponding candidate for $\cMMHull(\vec{x},\vec{m})$ is obtained by solving
\begin{align}
	& \min_{r \in [0,1]} \cMM(x_1,m_1,x_2,r\,m_2,x_3,0) + \cMM(x_1,0,x_2,(1-r)\,m_2,x_3,m_3) \nonumber \\
	= {} & \min_{r \in [0,1]} \left(\lambda_1 m_1 + \lambda_2 r\,m_2-\lambda_1 \lambda_2 \sqrt{m_1 r m_2} \cos(|x_1-x_2|) \right)
		\nonumber \\
		& \qquad +\left(\lambda_3 m_3 + \lambda_2 (1-r)\,m_2-\lambda_3 \lambda_2 \sqrt{m_3 (1-r) m_2} \cos(|x_3-x_2|) \right)
		\nonumber \\
	= {} & \sum_{i=1}^N \lambda_i\,m_i - \left(
		\lambda_1^2\lambda_2^2 m_1m_2 \cos(x_1-x_2)^2
		+\lambda_3^2\lambda_2^2 m_3m_2 \cos(x_3-x_2)^2
		\right)^{1/2}.
		\label{eq:CMMTransitionSplit}
\end{align}
Here we used \eqref{eq:CMMSimple} to obtain $\cMM(x_1,m_1,x_2,r\,m_2,x_3,0)$, which is admissible since $|x_1-x_2|<\pi/2$ and the mass argument for the third point is zero, and analogously for $\cMM(x_1,0,x_2,(1-r)\,m_2,x_3,m_3)$.
Again, we set $\vec{\psi}$ to the $\vec{m}$-gradient of \eqref{eq:CMMTransitionSplit} and test if $\vec{\psi} \in \QMM(\vec{x})$.
If this is true, the barycenter is given by $s_1\delta_{y_1} + s_2 \delta_{y_2}$ with
\begin{align*}
	y_1 & = \arg(\lambda_1 \sqrt{m_1} e^{ix_1}+ \lambda_2 \sqrt{r\,m_2} e^{ix_2}), &
	s_1 & = |\lambda_1 \sqrt{m_1} e^{ix_1}+ \lambda_2 \sqrt{r\,m_2} e^{ix_2}|^2, \\
	y_2 & = \arg(\lambda_3 \sqrt{m_3} e^{ix_3}+ \lambda_2 \sqrt{(1-r)\,m_2} e^{ix_2}), &
	s_2 & = |\lambda_3 \sqrt{m_3} e^{ix_3}+ \lambda_2 \sqrt{(1-r)\,m_2} e^{ix_2}|^2,
\end{align*}
where $r \in [0,1]$ is the minimizer in \eqref{eq:CMMTransitionSplit} and the convention for $\arg$ is $y_1 \in [x_1,x_2]$, $y_2 \in [x_2,x_3]$.

If once more $\vec{\psi} \notin \QMM(\vec{x})$, we must be in the diffuse regime where $\cMMHull(\vec{x},\vec{m})=\sum_{i=1}^3 \psi_i\,m_i$ and $\vec{\psi}$ is the unique vector where the corresponding $\vec{\chi}$ satisfies $\sum_{i=1}^3 w_i=0$ and $\tfrac12\sum_{i=1}^3 \chi_i=1$.

An illustration of $\cMM(\vec{x},\cdot)$ and $\cMMHull(\vec{x},\cdot)$ for this case is given in Figure \ref{fig:CMM3D}.
A visualization of the $\HK$ barycenter between three Dirac masses with corresponding $f(\vec{\psi},\cdot)$ for varying pairwise distances is given in Figure \ref{fig:ThreeDiracs}. We observe that in case (d) the diffuse barycenter not only occurs for specific `non-generic' configurations, but over extended intervals of point distances and masses, unlike in the case $N=2$ where a diffuse barycenter only occurs for the point distance $\pi/2$.
\end{example}

Moving from $N=2$ to $N=3$ we observed that the behaviour of the $\HK$ barycenter between Dirac measures became more complex.
Nearby points are `clustered' into a single Dirac mass, far separated points are approximated by distinct masses.
The transition between the two regimes is not purely based on the distance between the points but also depends on their masses. Two points at a distance between $\pi/2$ and $\pi$ may still be represented by the same Dirac mass in the barycenter if there is a suitable third point between them with sufficiently high mass.
This `clustering behaviour' could be particularly useful for computing meaningful barycenters between empirical measures where none of the samples accurately represents the whole underlying distribution but only their combination does.
It should be noted that clustering effect of the $\HK$ barycenter compared to the Wasserstein barycenter depends on the length scale of the data points which must be chosen in advance.
This suggests in data analysis applications to a priori consider \emph{all} re-scalings $(t \cdot x_1,\ldots,t \cdot x_N)$ of the positions of the Diracs; see the next section.

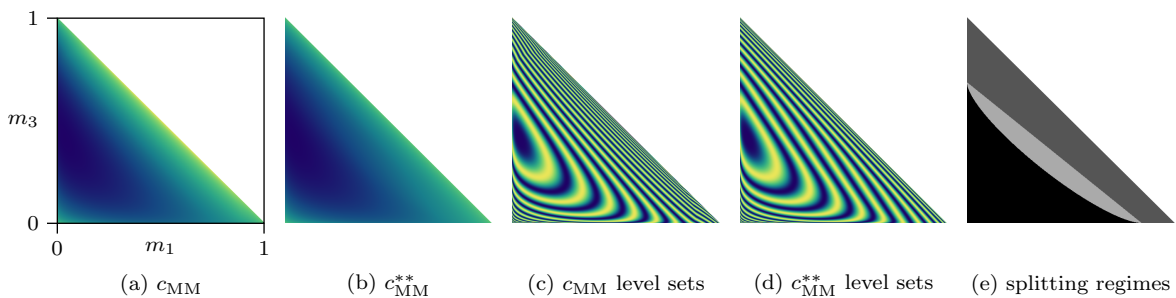
\begin{figure}[hbt]
	\centering
	{\def\imgw{0.17\textwidth}
	\begin{tikzpicture}[img/.style={anchor=south west,inner sep=0pt},y=\imgw,x=\imgw,border/.style={line width=0.5pt,black}]
	\node[img] (i1) at (0,0) []{\includegraphics[width=\imgw]{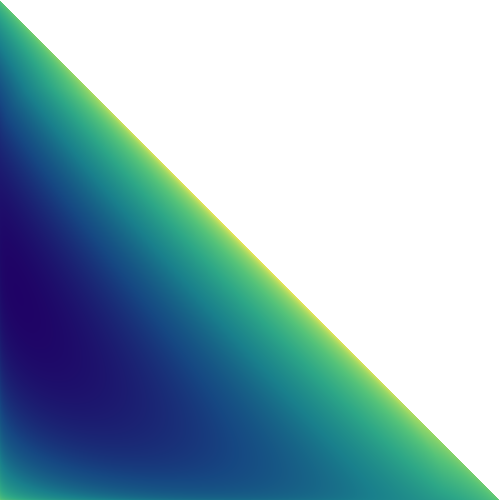}};
	\node[img] (i2) at (1.1,0) []{\includegraphics[width=\imgw]{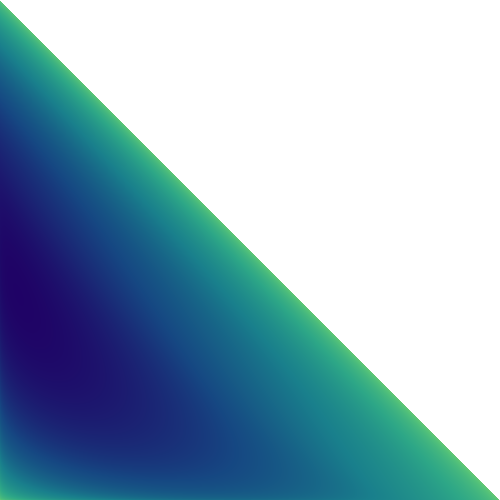}};
	\node[img] (i3) at (2.2,0) []{\includegraphics[width=\imgw]{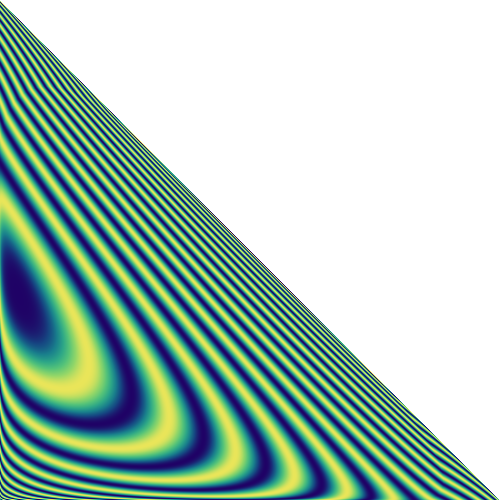}};
	\node[img] (i4) at (3.3,0) []{\includegraphics[width=\imgw]{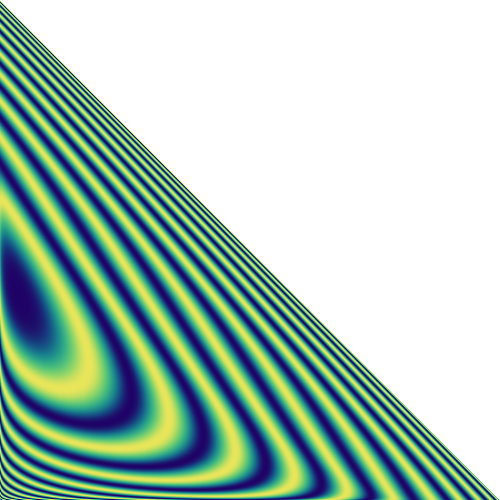}};
	\node[img] (i5) at (4.4,0) []{\includegraphics[width=\imgw]{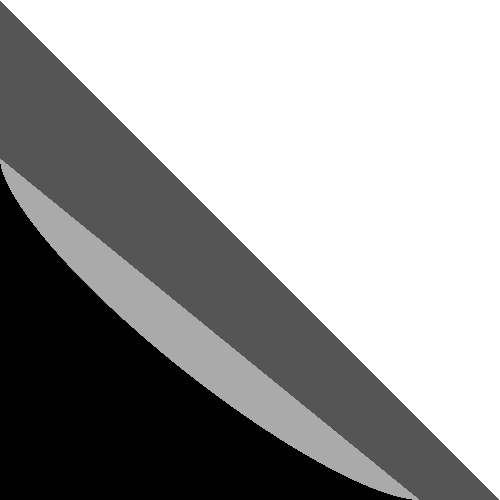}};
	{\scriptsize
	\node at ($(i1.south)+(0,-0.3)$) []{(a) $\cMM$};
	\node at ($(i2.south)+(0,-0.3)$) []{(b) $\cMMHull$};
	\node at ($(i3.south)+(0,-0.3)$) []{(c) $\cMM$ level sets};
	\node at ($(i4.south)+(0,-0.3)$) []{(d) $\cMMHull$ level sets};
	\node at ($(i5.south)+(0,-0.3)$) []{(e) splitting regimes};
	}
	\draw[border] (0,0) rectangle (\imgw,\imgw);
	{\scriptsize
	\draw[border] (0,0) -- ++(0,-0.05) node[below]{0};
	\draw[border] (1,0) -- ++(0,-0.05) node[below]{1};
	\draw[border] (0,0) -- ++(-0.05,0) node[left]{0};
	\draw[border] (0,1) -- ++(-0.05,0) node[left]{1};
	\node at (0.5,-0.05) [below]{$m_1$};
	\node at (-0.05,0.5) [left]{$m_3$};
	}
	\end{tikzpicture}
	}
	\caption{Visualization of $\cMM(\vec{x},\cdot)$ and $\cMMHull(\vec{x},\cdot)$ for $N=3$, $d=1$, $\vec{x}$ as in Figure \ref{fig:QMM3D} (d) for $m_1, m_3 \in [0,1]^2$, $m_1+m_3 \leq 1$ as in (a), $m_2 \assign 1-m_1-m_3$. By positive 1-homogeneity of $\cMM^{(\ast\ast)}$ this fully characterizes the functions. (a,b): absolute color scale (blue: small, yellow: large), (c,d): visualization of level sets via periodic color scale.
	\newline
	(e): mass splitting behaviour. Black: $m_2$ large, $m_1,m_3$ small. $\cMM(\vec{x},\cdot)=\cMMHull(\vec{x},\cdot)$, no splitting, this corresponds to the red region in Figure \ref{fig:QMM3D} (d). Dark gray: splitting into two Dirac masses, orange-green boundary in Figure \ref{fig:QMM3D} (d). Light gray: potentially diffuse barycenter, orange-green-red vertex in Figure \ref{fig:QMM3D} (d).}
	\label{fig:CMM3D}
\end{figure}

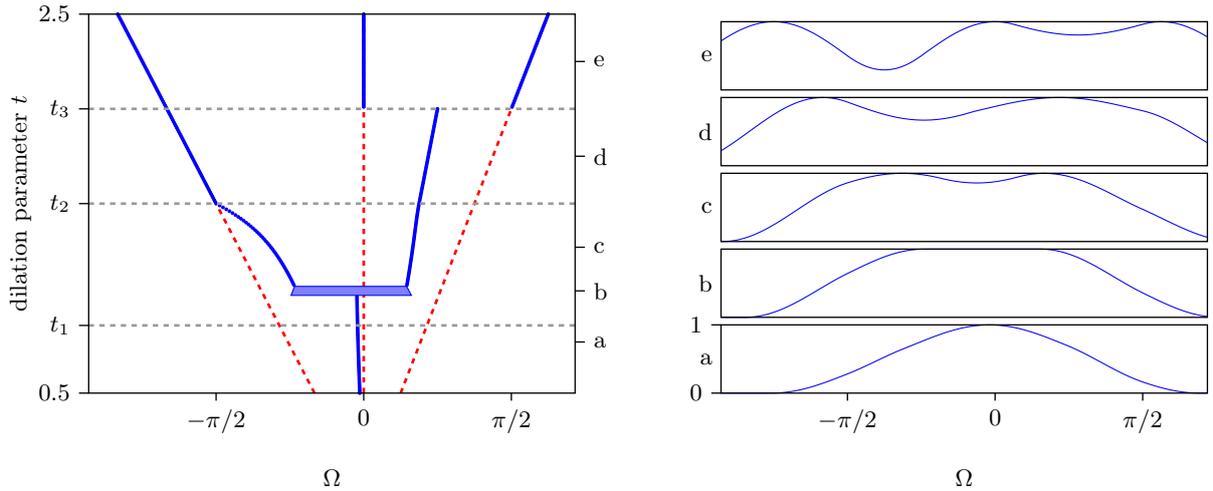
\begin{figure}[hbt]
	\centering
	{\footnotesize
	\def\imgwquotient{0.4}
	\def\imgw{\imgwquotient\textwidth}
	\pgfmathsetmacro\imghquotient{201/256*\imgwquotient}
	\def\imgh{\imghquotient\textwidth}
	\begin{tikzpicture}[
		img/.style={anchor=south west,inner sep=0pt},
		y=\imgh,x=\imgw,
		border/.style={line width=0.5pt,black},
		mass/.style={line width=1pt,dash pattern=on 2pt off 2pt,red},
		time/.style={line width=1pt,dash pattern=on 2pt off 2pt,white!60!black}
		]
	\begin{scope}[shift={(466/827+0.5/256,0)},x={(0.3035,0)}]
	\draw[mass] (0,0) -- (0,1);
	\draw[mass] (-2/3*0.5,0) -- (-2/3*2.5,1);
	\draw[mass] (1/2*0.5,0) -- (1/2*2.5,1);

	\draw plot[mark=*,mark size=0.5,only marks,mark options={blue}] file {fig/exact_clean.txt};
	
	{\def\tA{0.257841}
	\def\tB{0.281646}
	\draw[blue,fill=blue!50!white] ({1/2*(0.5+2*\tA)-1},\tA) -- ({-2/3*(0.5+2*\tA)+1},\tA)
		-- ({-2/3*(0.5+2*\tB)+1},\tB) -- ({1/2*(0.5+2*\tB)-1},\tB) -- cycle;
	}

	\foreach \t/\l in {0/0,1/{$\pi/2$},-1/{$-\pi/2$}} {
		\draw[border] (\t,0) -- ++(0,-0.02) node[below]{\l};
	}
	\end{scope}

	\foreach \t in {{5/28},0.5,0.75}{
		\draw[time] (0,\t) -- ++(1,0);
	}
	\foreach \t/\l in {0/{0.5},1/{2.5},{5/28}/{$t_1$},{0.5}/{$t_2$},{0.75}/{$t_3$}}{%
		\draw[border] (0,\t) -- ++(-0.02,0) node[left]{\l};
	}
	\node at (-0.10,0.5) [left,rotate=90,anchor=south]{dilation parameter $t$};

	\foreach \t/\l in {0.134872/a,0.269743/b,0.384872/c,0.625/d,0.875/e}{%
		\draw[border] (1,\t) -- ++(0.02,0) node [right]{\l};
	}

	\draw[border] (0,0) rectangle (1,1);
	\node at (0.5,-0.18) [anchor=north]{$\Omega$};

	\foreach \n/\l in {1/a,2/b,3/c,4/d,5/e} {
	\begin{scope}[shift={(1.3,0.2*(\n-1))}]

	\begin{scope}[shift={(466/827,0)},x={(0.3035,0)},y={(0,0.18)}]
	\draw[blue] plot[] file {fig/exact_clean_f_\n.txt};
%
%
	\end{scope}
%
%
%
	\draw[border] (0,0) rectangle (1,0.18);
	\node at (0,0.5*0.18) [anchor=east]{\l};
	\end{scope}
	}
	\begin{scope}[shift={(1.3,0)}]
		\begin{scope}[shift={(466/827,0)},x={(0.3035,0)},y={(0,1)}]
			\foreach \t/\l in {0/0,1/{$\pi/2$},-1/{$-\pi/2$}} {
				\draw[border] (\t,0) -- ++(0,-0.02) node[below]{\l};
			}
		\end{scope}
		\node at (0.5,-0.18) [anchor=north]{$\Omega$};
		\foreach \t in {0,1}{%
			\draw[border] (0,\t*0.18) -- ++(-0.02,0) node[left]{\t};
		}
	\end{scope}
	\end{tikzpicture}
	}
	\caption{Left: (Maximal) support of $\HK$ barycenter in $d=1$ between $N=3$ Dirac measures $\mu_i(t)=\delta_{x_i(t)}$ for various distances between the three points.
	\newline
	Horizontal axis: $\Omega \subset \R$, vertical axis: time for parametrizing moving points. Red lines: positions of points $x_i(t)$, $i$ increasing from left to right. Gray horizontal lines: transition between the cases discussed in Example \ref{ex:CMM:N3}.
	$t < t_1$: $|x_1-x_3|<\pi/2$. $t <t_2$: $|x_1-x_2|<\pi/2$, $t < t_3$: $|x_2-x_3|<\pi/2$.
	Blue markers: support of $\HK$ barycenter between the three points at their current locations.
	With increasing distance between the points the barycenter goes through the stages described in Example \ref{ex:CMM:N3}.
	The blue shaded area corresponds to the diffuse transition in case (d). It extends approximately over $t \in [1.02,1.06]$.
	\newline
	Right: Functions $f(\vec{\psi},\cdot)$ for the optimal $\vec{\psi}$ at times corresponding to the markers (a-e) on the left plot. $f(\vec{\psi},\cdot)$ has a single maximum in (a), is maximal on an interval in (b), has two maxima in (c,d) and three in (e).
	}
	\label{fig:ThreeDiracs}
\end{figure}

\section{HK barycenter tree}
\label{sec:HKBarycenterTree}

\begin{figure}[hbt]
	\centering
	{\footnotesize
	\def\imgwquotient{0.35}
	\def\imgw{\imgwquotient\textwidth}
	\pgfmathsetmacro\imghquotient{300/256*\imgwquotient}
	\def\imgh{\imghquotient\textwidth}
	\begin{tikzpicture}[
		img/.style={anchor=south west,inner sep=0pt},
		y=\imgh,x=\imgw,
		border/.style={line width=0.5pt,black},
		mass/.style={line width=1pt,dash pattern=on 2pt off 2pt,red},
		time/.style={line width=1pt,dash pattern=on 2pt off 2pt,white!60!black}
		]
	\begin{scope}[shift={(466/827+0.5/256,0)},x={(0.3035,0)}]
	\draw[mass] (0,0) -- (0,1);
	\draw[mass] (-2/3*2.5,0) -- (-2/3*2.5,1);
	\draw[mass] (1/2*2.5,0) -- (1/2*2.5,1);

	\draw plot[mark=*,mark size=0.5,only marks,mark options={blue}] file {fig/exact_clean_tree.txt};
	
	{\def\tA{0.257841}
	\def\tB{0.281646}
	\draw[blue,fill=blue!50!white] ({(1/2*2.5-1*2.5/(0.5+2*\tA))},\tA) -- ({(-2/3*2.5+1*2.5/(0.5+2*\tA))},\tA)
		-- ({(-2/3*2.5+1*2.5/(0.5+2*\tB))},\tB) -- ({(1/2*2.5-1*2.5/(0.5+2*\tB))},\tB) -- cycle;
	}

	\end{scope}

	\foreach \t in {{5/28},0.5,0.75}{
		\draw[time] (0,\t) -- ++(1,0);
	}
	\foreach \t/\l in {0/{0.5},1/{2.5},{5/28}/{$t_1$},{0.5}/{$t_2$},{0.75}/{$t_3$}}{%
		\draw[border] (0,\t) -- ++(-0.02,0) node[left]{\l};
	}
	\node at (-0.10,0.5) [left,rotate=90,anchor=south]{dilation parameter $t$};

	\draw[border] (0,0) rectangle (1,1);
	\node at (0.5,-0.05) [anchor=north]{$\Omega$};

	\begin{scope}[shift={(1.3,0)}]

	\node[img] (i1) at (0,0) []{\includegraphics[width=\imgw,height=\imgh]{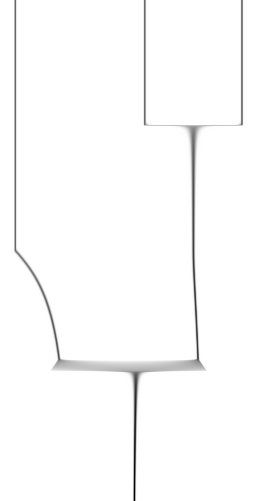}};
	\begin{scope}[shift={(466/827+0.5/256,0)},x={(0.3035,0)}]
	\draw[mass] (0,0) -- (0,1);
	\draw[mass] (-2/3*2.5,0) -- (-2/3*2.5,1);
	\draw[mass] (1/2*2.5,0) -- (1/2*2.5,1);


	\end{scope}

	\foreach \t in {{5/28},0.5,0.75} {
		\draw[time] (0,\t) -- ++(1,0);
	}

	\draw[border] (0,0) rectangle (1,1);
	\node at (0.5,-0.05) [anchor=north]{$\Omega$};

	\end{scope}

	\end{tikzpicture}
	}
	\caption{Visualization of the $\HK$ barycenter tree for the example from Figure \ref{fig:ThreeDiracs}. %
	Left: Exact support, as described in Example \ref{ex:CMM:N3}. This is the same plot as in Figure \ref{fig:ThreeDiracs} except that now each time-slice is re-scaled by $D_t^{-1}$.
	Right: Numerical approximation of the $\HK$ barycenter tree, computed with entropic regularization on a grid of 256 points and Sinkhorn algorithm introduced in \cite{ChizatEntropicNumeric2018} using the code from \cite{SchmitzerScaling2019} (black: high mass, white: no mass).
	When the barycenter is not unique (e.g.~during the transitions) the entropic smoothing selects a diffuse solution. This also happens close to the transitions where, even though the minimal barycenter is unique, the cost functional is very shallow around it.
	}
	\label{fig:ThreeDiracsTree}
\end{figure}

As demonstrated in the previous section, the local clustering behaviour of the HK barycenter depends on the length scale of the support of the marginal measures. With the (somewhat arbitrary) choice of relative importance of transport and source term in \eqref{eq:HKBB}, the critical distance between transportation and teleportation is set at $\pi/2$. In data analysis applications it makes sense to not fix this scale a priori, but to introduce the 1-parameter family of HK barycenters corresponding to all choices of this scale.

To this end, recall our general set-up that we are interested in the HK barycenter of $N$ nonnegative measures $(\mu_1,\ldots,\mu_N)$ defined on some compact, convex subset $\Omega$ of $\R^d$.

\begin{definition}
For $t > 0$, denote the dilation by a factor $t$ by
\begin{align*}
            D_t & : \Omega \to \Omega_t = t\cdot \Omega, &
            x & \mapsto t\cdot x.
\end{align*}
We say that $\nu_t$ is a HK barycenter at length scale $t > 0$ if $D_{t\sharp} \nu_t$ is a HK barycenter of $(D_{t\sharp} \mu_1, \ldots, D_{t \sharp} \mu_N)$.
\end{definition}

Thus, informally, the HK barycenter at scale $t$ is obtained by dilating the marginals via $D_t$, taking the scale-1 HK barycenter, and applying the inverse dilation.
Visualizations of examples for this 1-parameter family of barycenters are given in Figure \ref{fig:HKBarycenterTree}.
Motivated by Figure \ref{fig:HKBarycenterTree}, we call the set
\begin{equation*}
        \left\{(t,\nu_t) \, \middle|\, t \in \R_{++},\, \nu_t \tn{ a HK  barycenter of } (\mu_1,\ldots,\mu_N) \tn{ at scale } t \right\}
\end{equation*}
the \emph{$\HK$ barycenter tree}.

\begin{figure}
	\centering
	{\def\imgw{0.35\textwidth}
	\def\imgwb{0.4\textwidth}
	\footnotesize
	\begin{tikzpicture}[
		img/.style={anchor=south west,inner sep=0pt},
		y=\imgw,x=\imgw,
		border/.style={line width=0.5pt,black},
		mass/.style={line width=1pt,dash pattern=on 2pt off 2pt,red},
		time/.style={line width=1pt,dash pattern=on 2pt off 2pt,white!80!black},
		pt/.style={red,line width=0.5pt}
		]

	\node[img] (i1) at (0,0) []{\includegraphics[width=\imgw,height=\imgw]{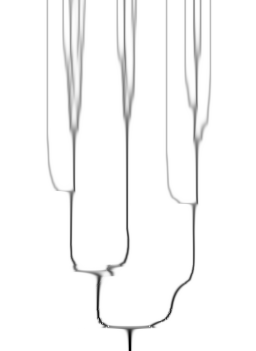}};
	\foreach \t/\l in {1/{$e^0$},{2.5/3.5}/{$e^{-1}$},{1.5/3.5}/{$e^{-2}$},{0.5/3.5}/{$e^{-3}$}} {
		\draw[border] (0,\t) -- ++(-0.02,0) node[left]{\l};
	}
	\draw[border] (0,0) rectangle (1,1);
	\node at (-0.15,0.5) [left,rotate=90,anchor=south]{dilation parameter $t$};

	\foreach \x in {72.801,80.366,69.568,62.637,46.848,77.016,68.224,84.569,133.459,124.175,126.227,%
	122.099,116.728,123.512,130.557,136.514,184.672,209.896,192.499,193.524,198.939,194.607,166.083,201.224} {
		\draw[pt] (\x/256+0.5/256-0.015,1.015) -- ++ (0.03,-0.03);
		\draw[pt] (\x/256+0.5/256+0.015,1.015) -- ++ (-0.03,-0.03);
	}

	\foreach \t/\l in {-40/{-20\pi},0/0,40/{20\pi}} {
		\draw[border] (0.5+0.01*\t+1/256,0) -- ++(0,-0.02) node[below]{$\l$};
	}

	\node at (0.5,-0.1) [anchor=north]{$\Omega$};

	\begin{scope}[shift={(1.2,0)}]
	\node[img] (i1) at (0,0) [anchor=south west]{\includegraphics[width=\imgwb]{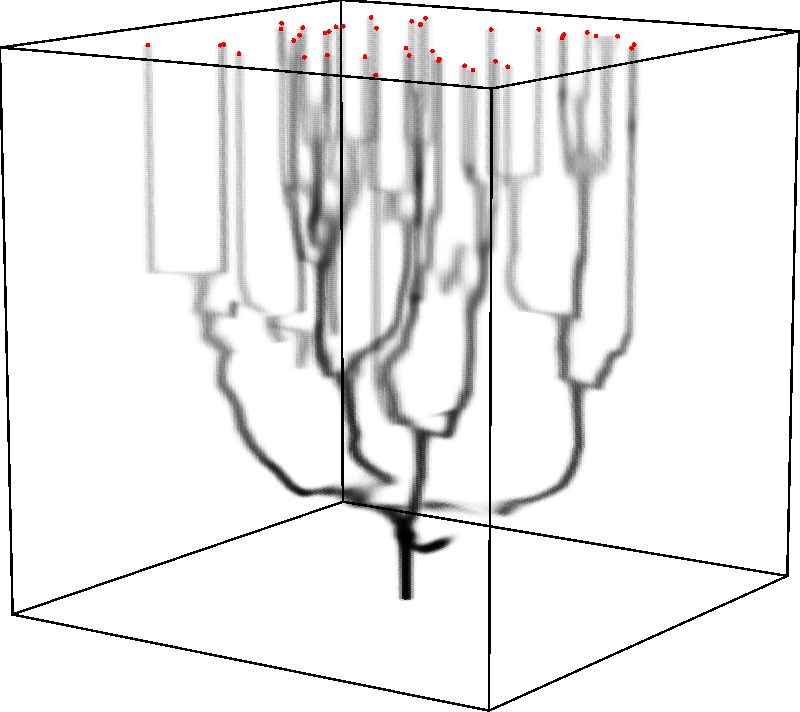}};
	\end{scope}

\end{tikzpicture}
}
\caption{Visualizations of the HK barycenter tree in one and two dimensions.\newline
	Left: $d=1$, $N=24$, $\mu_i=\delta_{x_i}$ with $(x_i)_{i}$ sampled from a mixture of three Gaussians (points marked in red on upper image boundary). For each value of the dilation parameter $t$ (vertical axis, logarithmic scaling), the corresponding horizontal slice shows the entropic approximation of the HK barycenter at scale $t$, computed as in Figure \ref{fig:ThreeDiracsTree}. The three branches near the bottom correspond to the underlying Gaussians.\newline
	Right: Similar visualization for $d=2$, $N=40$ and a mixture of four Gaussians. Again, the four branches near the bottom of the tree are prominently visible. See also Figures \ref{fig:PointClouds} and \ref{fig:Betti}.}
\label{fig:HKBarycenterTree}
\end{figure}

To illustrate the usefulness of this concept we apply it to the basic problem in data science to analyze point clouds sampled from a mixture of Gaussians and inferring the number and location of the underlying Gaussians.

\begin{figure}[hbt]
	\centering
	{\def\imgw{0.28\textwidth}
	\begin{tikzpicture}[
		img/.style={anchor=south west,inner sep=0pt},
		y=\imgw,x=\imgw,
		border/.style={line width=0.5pt,black},
		mass/.style={line width=1pt,dash pattern=on 2pt off 2pt,red},
		time/.style={line width=1pt,dash pattern=on 2pt off 2pt,white!80!black}
		]

	\begin{scope}[shift={(0,0)}]
	\node[img] (i1) at (0,0) [label=below:{(a) single points, $W$}]{\includegraphics[width=\imgw]{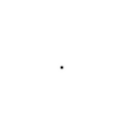}};
	\begin{scope}[shift={(0.5/128*\imgw,\imgw+0.5/128*\imgw)},x={(0,-\imgw/64.)},y={(\imgw/64.,0)}]
		\draw plot[mark=x,only marks,mark options={red}] file {fig/pointclouds/txt_single_sigma14.txt};
	\end{scope}
	\draw[border] (0,0) rectangle (1,1);
	\end{scope}

	\begin{scope}[shift={(1.3,0)}]
	\node[img] (i1) at (0,0) [label=below:{(b) single points, $\HK$}]{\includegraphics[width=\imgw]{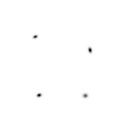}};
	\begin{scope}[shift={(0.5/128*\imgw,\imgw+0.5/128*\imgw)},x={(0,-\imgw/64.)},y={(\imgw/64.,0)}]
		\draw plot[mark=x,only marks,mark options={red}] file {fig/pointclouds/txt_single_sigma14.txt};
	\end{scope}
	\draw[border] (0,0) rectangle (1,1);
	\end{scope}

	\begin{scope}[shift={(0,-1.2)}]
	\node[img] (i1) at (0,0) [label=below:{(c) multiple points, $W$}]{\includegraphics[width=\imgw]{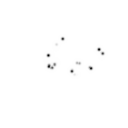}};
	\begin{scope}[shift={(0.5/128*\imgw,\imgw+0.5/128*\imgw)},x={(0,-\imgw/64.)},y={(\imgw/64.,0)}]
		\draw plot[mark=x,only marks,mark options={red}] file {fig/pointclouds/txt_multi3_sigma14.txt};
	\end{scope}
	\draw[border] (0,0) rectangle (1,1);
	\end{scope}

	\begin{scope}[shift={(1.3,-1.2)}]
	\node[img] (i1) at (0,0) [label=below:{(d) multiple points, $\HK$}]{\includegraphics[width=\imgw]{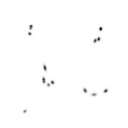}};
	\begin{scope}[shift={(0.5/128*\imgw,\imgw+0.5/128*\imgw)},x={(0,-\imgw/64.)},y={(\imgw/64.,0)}]
		\draw plot[mark=x,only marks,mark options={red}] file {fig/pointclouds/txt_multi3_sigma14.txt};
	\end{scope}
	\draw[border] (0,0) rectangle (1,1);
	\end{scope}
	
	\end{tikzpicture}
	}
	\caption{Comparison of Wasserstein and Hellinger--Kantorovich barycenters for point clouds in $d=2$ according to Example \ref{ex:PointClouds}.
	At the chosen length scale each image region represents the area $[0,5.12]^2$. Points marked with red crosses were sampled from a mixture of Gaussians. Image color coding shows the approximate $W$ or $\HK$ barycenter between marginal measures generated from the points (black: high mass, white: no mass), computed on a $128 \times 128$ grid with entropic regularization methods, see \cite{ChizatEntropicNumeric2018}, using code from \cite{SchmitzerScaling2019}.
	(b) is a horizontal slice of Figure \ref{fig:HKBarycenterTree} (right).
	}
	\label{fig:PointClouds}
\end{figure}

\begin{example}[Point clouds sampled from mixtures of Gaussians]
	\label{ex:PointClouds}
	Let $(x_i)_{i=1}^N$ be a set of points in $\R^d$, sampled from a mixture of Gaussians.
	
	If we set $\mu_i=\delta_{x_i}$ the Wasserstein barycenter of $(\mu_1,\ldots,\mu_N)$ (for simplicity with uniform weights $\lambda_i=1/N$) is given by $\nu=\delta_{\ol{x}}$ with $\ol{x}=T(x_1,\ldots,x_N)$. The information about the mixture of Gaussians is lost.

	Let now $\nu_t$ be a $\HK$ barycenter of $(\mu_1,\ldots,\mu_N)$ at scale $t$.
	Formally we expect $\nu_t$ to approach the Wasserstein barycenter as $t\to 0$.
	For $t$ sufficiently large, such that $t \cdot |x_i-x_j|>\pi/2$ for all $i,j$, $\nu_t=\tfrac{1}{N^2} \sum_{i=1}^N \delta_{x_i}$ is a sum of $N$ separate Dirac measures at the positions $x_i$. This equals the Hellinger barycenter.
	See \cite[Section 7.7]{LieroMielkeSavare-HellingerKantorovich-2015a} and \cite[Theorem 5.10]{ChizatDynamicStatic2018} for related convergence results.
	For intermediate $t$, $\nu_t$ will typically consist of multiple Dirac masses that can be interpreted as a clustering of sufficiently close points $x_i$.
	If the distance between the means of the underlying Gaussians is larger than their widths, for suitable $t$, each of the Gaussians will be represented by a separate Dirac mass.

	Numerical examples of the path $t \mapsto \nu_t$ for points sampled from mixtures of Gaussians in $d=1$ and $2$ are visualized in Figure \ref{fig:HKBarycenterTree}. In both cases, the separation into the underlying Gaussians is prominently visible.
	A comparison between the Wasserstein barycenter and the $\HK$ barycenter at a suitably chosen scale is shown in Figure \ref{fig:PointClouds} (a) and (b).
	
	The situation is similar when $\mu_i=\sum_{j=1}^{n_i} \delta_{x_{i,j}}$ with points $x_{i,j}$ sampled from some underlying distribution, but now each $\mu_i$ contains $n_{i}>0$ points (and the $n_i$ might not all be equal). The Wasserstein barycenter (after normalization of all $\mu_i$) becomes a combination of Dirac measures near the global center of mass of the sampling distribution (Figure \ref{fig:PointClouds} (c)) whereas the $\HK$ barycenter contains Dirac masses near the means of the underlying Gaussians, provided the length scale is larger than the width of the individual distributions but smaller than the distance between their centers (Figure \ref{fig:PointClouds} (d)). Once more, this provides the more accurate representation.
\end{example}

\begin{figure}[hbt]
	\centering
	\includegraphics[]{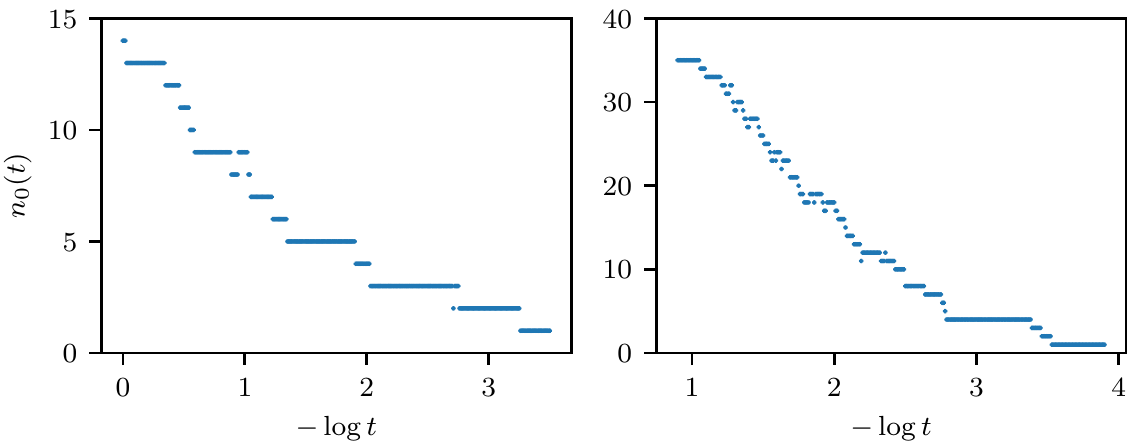}
	\caption{Number of Dirac masses $n_0(t)$ in HK barycenter $\nu_t$ between Dirac measures at different scales $t$ for examples shown in Figure \ref{fig:HKBarycenterTree}. Value is determined by thresholding the entropic approximation of the barycenter and counting the connected components. Near transitions this leads to small numerical fluctuations.\newline
	As expected, for sufficiently small $t$, $n_0(t)$ becomes 1, for large $t$ it approaches $N$ (left: $N=24$, right: $N=40$).
	The division into the underlying Gaussians leads to a plateau with $n_0(t)=3$ (left) and $n_0(t)=4$ (right).}
	\label{fig:Betti}
\end{figure}

We remark that analyzing data sets simultaneously at all scales is common in topological data analysis, for example in hierarchical clustering methods such as single linkage-based algorithms (see \cite{ClusterTreeEstimation2014} for an overview, additional references and in particular for the question of consistency as the number of samples increases), or the notion of barcode of the Rips or \v{C}zech complexes at all scales \cite{Ghrist-Barcodes-2008}. For instance, if the $\mu_i$ are Diracs at $x_i$, the number $n_0(t)$ of connected components of the support of the HK barycenter as a function of the scale $t$ is tantalizingly similar to the zeroth Betti number $b_0(\veps)$ of the \v{C}zech complex of the $x_i$ as a function of scale $\veps$, i.e., the number of connected components of the union of the closed radius-$\veps/2$ balls around the $x_i$. However, the number $n_0(t)$ also takes into account the masses of the Diracs, and naturally extends to arbitrary marginal measures $\mu_i$. A numerical computation of $n_0(t)$ for the HK barycenter trees in Example \ref{ex:PointClouds} (visualized in Figure \ref{fig:HKBarycenterTree}) is shown in Figure \ref{fig:Betti}. 

\section{Non-existence of a multi-marginal soft-marginal formulation}
\label{sec:NonExistence}
In Section \ref{sec:HK} various formulations for $\HK$ were given.
Formulation \eqref{eq:HKLifted} is based on a lifting to a transport problem on the higher-dimensional cone $\cone$, formulation \eqref{eq:HKSemiCoupling} involves multiple `semi-couplings' and a non-linear transport cost.
It is remarkable that there also exists formulation \eqref{eq:HKSoftMarginal} in terms of a `simple' transport problem on $\Omega$ for a particular cost function $\cKL$ and relaxed marginal constraints.

In this section we study whether an analogous simplification for the $\HK$-barycenter problem exists and show that the answer is negative for $N\geq 3$ under natural structural assumptions.
More we precisely, we want to see if
\begin{align}
	\label{eq:HKMMSoftMarginalForm}
	\HKMM(\mu_1,\ldots,\mu_N)^2 = \inf\left\{
		\int_{\Omega^N} \cKLMM(\vec{x})\,\diff \gamma(\vec{x})
		+ \sum_{i=1}^N H_i(\pi_{i\sharp} \gamma)
		\middle|
		\gamma \in \measp(\Omega^N)
		\right\}	
\end{align}
for suitable choices of $\cKLMM : \Omega^N \to \RCupInf$ and $H_i : \measp(\Omega) \to \RCupInf$.
For very particular choices of $\cKLMM$ and $H_i$ this is certainly possible (for instance, set $\cKLMM(\vec{x}) \assign \HKMM(\mu_1,\ldots,\mu_N)^2$ and $H_i(\sigma) = 0$ if $\|\sigma\|=1$ and $+\infty$ otherwise, but it is not mathematically interesting) and we are interested whether it is possible under additional natural structural assumptions:
\begin{enumerate}[({A}1)]
	\item $\cKLMM$ is lower-semicontinuous in $\vec{x}$ and depends only on the weights $(\lambda_1,\ldots,\lambda_N)$ and the pairwise distances $|x_i-x_j|$, $i,j \in \{1,\ldots,N\}$ but not on $(\mu_1,\ldots,\mu_N)$.
	\label{asp:C}
	\item $H_i$ only depends on $\mu_i$ (not on the other reference measures) and the dependency is local, i.e.~via an integral representation. More precisely, $H_i$ can be written as
	\begin{align}
		\label{eq:HKMMSoftMarginalF}
		H_i(\sigma) = \begin{cases}
			\lambda_i \cdot \left[ \int_\Omega h\left(\RadNik{\sigma}{\mu_i}\right)\,\diff \mu_i + h^\infty \cdot \|\sigma^\perp\| \right] & \tn{if } \sigma \geq 0, \\
			+ \infty & \tn{else.}
		\end{cases}
	\end{align}
	for a proper convex lower-semicontinuous function $h : \R \to \RCupInf$ where $h^\infty \assign \lim_{s \to \infty} h(s)/s \in \RCupInf$ is the recession constant of $h$ and $\sigma=\RadNik{\sigma}{\mu_i} \cdot \mu_i + \sigma_i^\perp$ is the Lebesgue decomposition of $\sigma$ with respect to $\mu_i$. Without loss of generality we may assume that $h(s)=+\infty$ for $s<0$.
	\label{asp:F}
\end{enumerate}

\begin{remark}[On the assumptions]
Indeed, the joint dependency of $H_i$ on $\sigma$ and $\mu_i$ must be positively 1-homogeneous (so that under rescaling all reference measures $\mu_i$, the minimizing $\gamma$ is rescaled as well) implying that by the assumption of local dependency on $\mu_i$, $H_i$ can be written as
\begin{align*}
	H_i(\sigma) = \int_\Omega h_i\left(x,\RadNik{\sigma}{\mu_i}\right)\,\diff \mu_i + \int_{\Omega} h_i^\infty(x)\,\diff \sigma_i^\perp(x)
\end{align*}
with a function $h_i(\cdot,\cdot)$ that may, a priori, also explicitly depend on $x \in \Omega$ and $i$ (and thus, so does the recession constant).
$h_i(\cdot,\cdot)$ must be jointly lower-semicontinuous in both arguments and convex in the second argument, $\cKLMM$ must be lower-semicontinuous to yield a well-defined minimization problem.
Since $\HKMM(\mu_1,\ldots,\mu_N)^2$ is invariant under applying isometric transformations to the measures $\mu_i$ (as long as their support remains within $\Omega$), $h_i(\cdot,\cdot)$ cannot depend explicitly on the location $x$ in $\Omega$ and $\cKLMM$ only depends on pairwise distances and the weights $(\lambda_1,\ldots,\lambda_N)$.
If one were to reweigh all $\lambda_i$ by a constant factor $\eta>0$ (ignoring the normalization condition $\sum_{i=1}^N \lambda_i=1$) then $\HKMM(\mu_1,\ldots,\mu_N)^2$ is also rescaled by $\eta$ (cf.~\eqref{eq:HKCTM}). Thus, the dependency of $H_i$ on $\lambda_i$ must be positively 1-homogeneous.
Applying a permutation to the order of $(\mu_1,\ldots,\mu_N)$ and $(\lambda_1,\ldots,\lambda_N)$ must leave $\HKMM(\mu_1,\ldots,\mu_N)^2$ invariant. Thus, $h_i(\cdot,\cdot)$ cannot explicitly depend on the index $i$.
This leaves us with \eqref{eq:HKMMSoftMarginalF}.
\end{remark}

We will show that $\HKMM(\mu_1,\ldots,\mu_N)^2$ cannot be written in the form \eqref{eq:HKMMSoftMarginalForm} under assumptions (A\ref{asp:C}) and (A\ref{asp:F}) by considering a dual problem to \eqref{eq:HKMMSoftMarginalForm} and show that it cannot agree with \eqref{eq:HKMMDual} for any choices of $\cKLMM$ and $h$. Since a rigorous derivation of the dual problem would only be possible if additional properties of $\cKLMM$ and $h$ were known (e.g.~the details of the asymptotic behaviour of $h$ at $\infty$, the behaviour of $h^\ast$ at the boundaries of its domain, the sign of $\inf \cKLMM+h^\infty$), we do a formal derivation of the dual problem instead and show that its form is incompatible with \eqref{eq:HKMMDual}.
\begin{remark}
\label{rem:HKMMSoftMarginalDual}
Formally, a dual for the optimization problem in \eqref{eq:HKMMSoftMarginalForm} is given by
\begin{align}
	\label{eq:HKMMSoftMarginalDual}
	\sup \left\{ \sum_{i=1}^N \int_\Omega \xi_i\,\diff \mu_i
	\middle| \xi_1,\ldots,\xi_N \in \cont(\Omega),\,
	\sum_{i=1}^N \lambda_i\,g(\xi_i(x_i)/\lambda_i) \leq \cKLMM(\vec{x})
	\,\forall\, \vec{x} \in \Omega^N
	\right\}	
\end{align}
where the convex function $g$ is given by
\begin{align*}
	g(s) & \assign \inf \{ r \in \R | -h^\ast(-r) \geq s \}.
\end{align*}
\end{remark}
\begin{proof}[Sketch of proof]
Since $h(s)=\infty$ for $s<0$, $h^\ast$ is an increasing function. Therefore, so is $s \mapsto -h^\ast(-s)$ and consequently also $g$ is increasing.
Since $s \mapsto -h^\ast(-s)$ is concave, $g$ is convex. Therefore, $g$ is continuous on (the interior of) its domain. Then, by a change of variables $\psi_i \assign \lambda_i\,g(\xi_i/\lambda_i)$, we find that \eqref{eq:HKMMSoftMarginalDual} is equivalent to
\begin{align*}
	\sup \left\{ \sum_{i=1}^N \int_\Omega -\lambda_i\,h^\ast\big(-\tfrac{\psi_i}{\lambda_i}\big)\,\diff \mu_i
	\middle| \psi_1,\ldots,\psi_N \in \cont(\Omega),\,
	\sum_{i=1}^N \psi_i \leq \cKLMM(\vec{x})
	\,\forall\, \vec{x} \in \Omega^N
	\right\}	
\end{align*}
which in turn can be written as
\begin{align*}
	-\inf \left\{ G(\vec{\psi}) + F(A\vec{\psi}) \middle| \vec{\psi} \in \cont(\Omega)^N \right\}
\end{align*}
with
\begin{align*}
	G & : \cont(\Omega)^N \to \RCupInf, & & \vec{\psi} \mapsto 
	\begin{cases}
		\sum_{i=1}^N
		\int_\Omega \lambda_i\,h^\ast\big(-\tfrac{\psi_i}{\lambda_i}\big)\,\diff \mu_i
		& \tn{if } \psi_i/\lambda_i \geq -h^\infty \tn{ for } i=1,\ldots,N, \\
		+ \infty & \tn{else.}
		\end{cases} \\
	F & : \cont(\Omega^N) \to \RCupInf, & & \phi \mapsto \begin{cases}
		0 & \tn{if } \phi(\vec{x}) \leq \cKLMM(\vec{x}) \tn{ for all } \vec{x} \in \Omega^N, \\
		+\infty & \tn{else,}
		\end{cases} \\
	A & : \cont(\Omega)^N \to \cont(\Omega^N), & & \vec{\psi} \mapsto \phi
	\tn{ with } \phi(\vec{x}) = \sum_{i=1}^N \psi_i(x_i).
\end{align*}
The conjugate of $F$ and the adjoint of $A$ are analogous to Remark \ref{rem:WDuality}. The conjugate of $G$ is formally given by
$G^\ast(\vec{\sigma}) = \sum_{i=1}^N H_i(-\sigma_i)$.
Therefore, the formal dual problem
\begin{align*}
	\inf\left\{ G^\ast(-A^\ast \gamma) + F^\ast(\gamma) \middle| \gamma \in \meas(\Omega^N) \right\}
\end{align*}
coincides with \eqref{eq:HKMMSoftMarginalForm} with $H_i$ given by \eqref{eq:HKMMSoftMarginalF}.
\end{proof}

Now we observe that by Proposition \ref{prop:HKMMDirac} $\cMMHull$ is the unique function that is convex and positively 1-homogeneous in its mass arguments that can be plugged into \eqref{eq:HKMMLifted} or \eqref{eq:HKMMSemiCoupling} to yield $\HKMM$ ($\cMM$ can also be used in \eqref{eq:HKMMLifted} but it is not everywhere convex in its mass arguments). Further, $\QMM$ is the unique family of closed convex sets that satisfies the relation $\cMMHull(\vec{x},\cdot)=\iota_{\QMM(\vec{x})}^\ast$ (because $\cMM^{\ast\ast\ast}(\vec{x},\cdot)=\iota_{\QMM(\vec{x})}$).
Therefore, in the dual formulation \eqref{eq:HKMMDual} the constraint set $\QMM(\vec{x})$ is the unique choice to yield $\HKMM$ in combination with the objective $\sum_{i=1}^N \int_\Omega \psi_i\,\diff \mu_i$.
So finally, for \eqref{eq:HKMMSoftMarginalForm} with \eqref{eq:HKMMSoftMarginalDual} as its (formal) dual to be an equivalent formulation of $\HKMM$ one must have that
\begin{align}
	\label{eq:QMMSoftMarginalEquiv}
	\left[\vec{\psi} \in \QMM(\vec{x})\right]
	\qquad \Leftrightarrow \qquad
	\left[\sum_{i=1}^N \lambda_i\,g(\psi_i/\lambda_i) \leq \cKLMM(\vec{x})\right]
	\qquad \forall\, \vec{x} \in \Omega^N,\,\vec{\psi} \in \R^N
\end{align}
for some choice of $g$ and $\cKLMM$. The following Proposition shows that this is not possible for $N \geq 3$ and thus no formulation of $\HKMM$ in the form of \eqref{eq:HKMMSoftMarginalForm} under assumptions (A\ref{asp:C}) and (A\ref{asp:F}) is possible.

\begin{proposition}
	For $N \geq 3$ there exists no function $g : \R \to \RCupInf$ and no function $\cKLMM : \Omega^N \to \RCupInf$ such that \eqref{eq:QMMSoftMarginalEquiv} is true.
\end{proposition}
\begin{proof}
We prove the result by providing a counterexample.
Assume that $g$ and $\cKLMM$ exist such that \eqref{eq:QMMSoftMarginalEquiv} is true.
It suffices to produce a contradiction for some particular choice of the $x_i$;
take $x_1\in\Omega$ and $x_2=\cdots=x_N\in\Omega$,
distinct but closer than $\pi/2$,
and define $\alpha:=\sin(|x_1-x_2|)^2\in(0,1)$.
For simplicity we choose $\lambda_i=1/N$ for $i \in \{1,\ldots,N\}$ but the example can be extended to general weights.

By Lemma \ref{lem:CMMMinimizerExist} \eqref{item:CMMConvexHull} we can restrict the minimization over $y$ in \eqref{eq:CMM} to the convex hull of $x_1$ and $x_2=\ldots=x_N$, i.e.~the line segment between the points.
Thus the setting of Proposition \ref{prop:QMMQuadratic} and the representation \eqref{eq:QMMTildeEquiv} for $\QMM$ applies.

The following choice of $\vec{\psi}$ lies in $\QMM(\vec{x})$:
\begin{align*}
  \psi_1= \hat\psi:=\left(\frac{1}{N}-\frac{1}{N^2}\right)\,\alpha,
  \quad
  \psi_2=\cdots=\psi_N=0.
\end{align*}
The corresponding $\vec{\chi}$ for $\QMMTilde(\vec{x})$, cf.~\eqref{eq:QMMTildeEquiv}, is given by
\begin{align*}
  \chi_1= \hat\chi:=\frac{1}{N-(N-1)\,\alpha},
  \quad
  \chi_2=\cdots=\chi_N=\frac{1}{N}.
\end{align*}
Indeed, observe that
\begin{align*}
  \sum_{i=1}^N \chi_i & = \hat{\chi}+\frac{N-1}{N}, &
  \sum_{1 \leq j<k\leq N} \chi_j\,\chi_k \sin(|x_j-x_k|)^2
  = \hat{\chi} \frac{N-1}{N} \alpha
\end{align*}
and therefore,
\begin{align*}
  \sum_{i=1}^N \chi_i - \sum_{1 \leq j<k\leq N} \chi_j\,\chi_k \sin(|x_j-x_k|)^2 & = \hat{\chi} \left(1-\frac{N-1}{N} \alpha\right)+\frac{N-1}{N}=1.	
\end{align*}

By the assumed equivalence \eqref{eq:QMMSoftMarginalEquiv},
\begin{align*}
	\frac{1}{N}\,g(N\,\hat{\psi}) + \left(1-\frac{1}{N}\right)\,g(0) \leq \cKLMM(\vec{x}).
\end{align*}
Using that equivalence in the other direction, we conclude
--- recall that the $\psi_i$ enter symmetrically, since $\lambda_i=1/N$ ---
that also the choice
\begin{align*}
  \psi_1=0,\quad \psi_2=\hat\psi,\quad
  \psi_3=\cdots=\psi_N=0
\end{align*}
should lie in $\QMM(\vec{x})$ and therefore,
\begin{align*}
  \chi_1=\frac{1}{N},\quad \chi_2=\hat\chi,\quad
  \chi_3=\cdots=\chi_N=\frac{1}{N}
\end{align*}
should lie in $\QMMTilde(\vec{x})$.
We show that this cannot be true.
Namely, observe that
\begin{align*}
  \sum_{i=1}^N \chi_i & = \hat{\chi}+\frac{N-1}{N}, &
  \sum_{1 \leq j<k\leq N} \chi_j\,\chi_k \sin(|x_j-x_k|)^2
  = \left(\frac{\hat{\chi}}{N}+\frac{N-2}{N^2}\right) \alpha
\end{align*}
and so
\begin{align*}
  \sum_{i=1}^N \chi_i - \sum_{1 \leq j<k\leq N} \chi_j\,\chi_k \sin(|x_j-x_k|)^2 & =
  \hat{\chi} \!\!\!\!\!\!\!\!\!\!\underbrace{\left(1-\frac{\alpha}{N}\right)}_{=\frac{1}{N}[N-(N-1)\alpha+(N-2)\alpha]} \!\!\!\!\!\!\!\!\!\! +\frac{N-1}{N} - \frac{(N-2)\,\alpha}{N^2} \\
  & = \frac{1}{N} +\hat{\chi}\,\frac{(N-2)\,\alpha}{N}+\left(1-\frac{1}{N}\right)-\frac{(N-2)\,\alpha}{N^2} \\
  & = 1 + (N-2) \left( \frac{1}{N[N-(N-1)\alpha]}-\frac{1}{N^2}\right)\alpha
\end{align*}
which is strictly greater than 1 since $N\geq 3$ and $\alpha \in (0,1)$.
\end{proof}
\begin{remark}
The above argument hinges on the fact that the set $\QMM(\vec{x})$ is not invariant under permutations of the entries of $\vec{\psi}$ (even when taking into account the weights $\lambda_1,\ldots,\lambda_N$) and thus cannot be adapted to the case $N=2$.
Indeed, for $N=2$ we find with \eqref{eq:ExampleN2CMM} that
\begin{align*}
	\left[(\psi_1,\psi_2) \in \QMM(x_1,x_2)\right]
	\qquad \Leftrightarrow \qquad
	\left[\frac{(\psi_1,\psi_2)}{t\,(1-t)} \in Q(x_1,x_2)\right]
\end{align*}
and from \eqref{eq:HKQ} we see that $Q(x_1,x_2)$ can be represented as
\begin{align*}
	Q(x_1,x_2) = \left\{
		(\psi_1,\psi_2) \in (-\infty,1]^2 \middle|
		g(\psi_1) + g(\psi_2) \leq \cKL(x_1,x_2) \right\}
\end{align*}
for $g(s)=-\log(1-s)$ for $s \in (-\infty,1)$ and $g(1)=+\infty$. Then, retracing the duality arguments for Remark \ref{rem:HKMMSoftMarginalDual} we arrive at the formulation \eqref{eq:HKSoftMarginal} which is indeed of the form \eqref{eq:HKMMSoftMarginalForm}.
\end{remark}

\section{Conclusion}
In this article we have studied several ultimately equivalent formulations for the barycenter with respect to the $\HK$ metric on non-negative measures over $\R^d$. Particular attention was paid to the barycenter between Dirac measures which illustrates the fundamental difference to the Wasserstein barycenter. In this case the Wasserstein barycenter is always a single Dirac measure whereas the $\HK$ barycenter `clusters' only sufficiently close masses into a single Dirac, while far separated masses are represented by separate Dirac measures. This clustering behaviour as a function of length scale as encoded in the HK barycenter tree yields interesting information for data analysis applications.

Fruitful questions for future research suggested by the results reported here are the following. Give a more explicit characterization of the set $\CSet$ where the multi-marginal cost function agrees with its convex hull, which governs the splitting behaviour of the HK barycenter. Understand whether optimal plans in the multi-marginal formulation exhibit a `Monge-type' sparsity as established by Agueh and Carlier for the Wasserstein barycenter \cite{WassersteinBarycenter}. Analyze the structure of the HK barycenter tree; for instance: can the barycenters at scale $t$ be chosen in such a way that the number of connected components of the support is increasing in $t$ (with the possible exception of transitions)? Does the HK barycenter tree together with the associated optimal plans induce a useful notion of hierarchical simplicial complexes, and higher Betti numbers?
It would also be interesting to study the barycenter with respect to related transport metrics such as the spherical Hellinger--Kantorovich distance \cite{LaMi2017}.
Does the notion of $\HK$-barycenter extend to an infinite number (or even a continuous distribution) of reference measures, as studied in \cite{BrendanInfMarginal2013,LeGouicLoubes2017,BigotKleinBarycenter2018} for the Wasserstein barycenter? It seems plausible that arguments for existence, consistency and uniqueness employed there can be adapted to the coupled-two-marginal formulation \eqref{eq:HKCTM}. Is it also possible to generalize the dual problem? Can the $\HK$-barycenter between a continuum of Dirac measures still be a finite combination of Dirac measures? The latter would be related to the consistency of density-based clustering, see \cite{ClusterTreeEstimation2014}.
\paragraph{Acknowledgements} Bernhard Schmitzer was supported by the Emmy Noether Programme of the DFG.

\bibliography{references}{}

\begin{thebibliography}{10}

\bibitem{WassersteinBarycenter}
M.~Agueh and G.~Carlier.
\newblock Barycenters in the {W}asserstein space.
\newblock {\em SIAM J. Math. Anal.}, 43(2):904--924, 2011.

\bibitem{FixedPointWassersteinBarycenters2016}
P.~C. \'Alvarez-Esteban, E.~del Barrio, J.~Cuesta-Albertos, and C.~Matr\'an.
\newblock A fixed-point approach to barycenters in {Wasserstein} space.
\newblock {\em Journal of Mathematical Analysis and Applications},
  441(2):744--762, 2016.

\bibitem{BigotKleinBarycenter2018}
{Bigot, J\'er\'emie} and {Klein, Thierry}.
\newblock Characterization of barycenters in the {Wasserstein} space by
  averaging optimal transport maps.
\newblock {\em ESAIM: PS}, 22:35--57, 2018.

\bibitem{BoVa1988}
G.~Bouchitt\'e and M.~Valadier.
\newblock Integral representation of convex functionals on a space of measures.
\newblock {\em J. Funct. Anal.}, 80(2):398--420, 1988.

\bibitem{BDePGG-DFT-2012}
G.~Buttazzo, L.~De~Pascale, and P.~Gori-Giorgi.
\newblock Optimal-transport formulation of electronic density-functional
  theory.
\newblock {\em Phys. Rev. A}, 85:062502, 2012.

\bibitem{CarEke2010}
G.~Carlier and I.~Ekeland.
\newblock Matching for teams.
\newblock {\em Econ Theory}, 42(2):397--418, 2010.

\bibitem{ClusterTreeEstimation2014}
K.~{Chaudhuri}, S.~{Dasgupta}, S.~{Kpotufe}, and U.~{von Luxburg}.
\newblock Consistent procedures for cluster tree estimation and pruning.
\newblock {\em IEEE Transactions on Information Theory}, 60(12):7900--7912,
  2014.

\bibitem{ChizatOTFR2015}
L.~Chizat, G.~Peyr\'e, B.~Schmitzer, and F.-X. Vialard.
\newblock An interpolating distance between optimal transport and {Fisher--Rao}
  metrics.
\newblock {\em Found. Comp. Math.}, 18(1):1--44, 2018.

\bibitem{ChizatEntropicNumeric2018}
L.~Chizat, G.~Peyr\'e, B.~Schmitzer, and F.-X. Vialard.
\newblock Scaling algorithms for unbalanced optimal transport problems.
\newblock {\em Math. Comp.}, 87:2563--2609, 2018.

\bibitem{ChizatDynamicStatic2018}
L.~Chizat, G.~Peyr\'e, B.~Schmitzer, and F.-X. Vialard.
\newblock Unbalanced optimal transport: Dynamic and {Kantorovich} formulations.
\newblock {\em J. Funct. Anal.}, 274(11):3090--3123, 2018.

\bibitem{HKBarycenters2019}
N.-P. Chung and M.-N. Phung.
\newblock Barycenters in the {Hellinger--Kantorovich} space.
\newblock arXiv:1909.05513, 2019.

\bibitem{CFK18}
C.~Cotar, G.~Friesecke, and C.~Kl\"uppelberg.
\newblock Smoothing of transport plans with fixed marginals and rigorous
  semiclassical limit of the {Hohenberg--Kohn} functional.
\newblock {\em Arch. Rat. Mech. Analysis}, 228(3):891--922, 2018.

\bibitem{CFK-DFT-2013}
C.~Cotar, G.~Friesecke, and C.~Klüppelberg.
\newblock Density functional theory and optimal transportation with {Coulomb}
  cost.
\newblock {\em Comm. Pure Appl. Math.}, 66(4):548--599, 2013.

\bibitem{GalichonOTEconomics}
A.~Galichon.
\newblock {\em Optimal Transport Methods in Economics}.
\newblock Princeton University Press, 2017.

\bibitem{GangboSwiechMultimarginal98}
W.~Gangbo and A.~{\'S}wiech.
\newblock Optimal maps for the multidimensional monge-kantorovich problem.
\newblock {\em Comm. Pure Appl. Math.}, 51(1):23--45, 1998.

\bibitem{Ghrist-Barcodes-2008}
R.~Ghrist.
\newblock Barcodes: The persistent topology of data.
\newblock {\em Bull. Amer. Math. Soc.}, 45(1):61--75, 2008.

\bibitem{KantorovichOT1958}
L.~V. Kantorovich.
\newblock On the translocation of masses.
\newblock {\em Managment Science}, 5(1):1--4, 1958.

\bibitem{Kell64}
H.~Kellerer.
\newblock Verteilungsfunktionen mit gegebenen marginalverteilungen.
\newblock {\em Wahrscheinlichkeitstheorie}, 3:247--270, 1964.

\bibitem{KMV-OTFisherRao-2015}
S.~Kondratyev, L.~Monsaingeon, and D.~Vorotnikov.
\newblock A new optimal transport distance on the space of finite {R}adon
  measures.
\newblock {\em Adv. Differential Equations}, 21(11-12):1117--1164, 2016.

\bibitem{LaMi2017}
V.~Laschos and A.~Mielke.
\newblock Geometric properties of cones with applications on the
  {Hellinger--Kantorovich} space, and a new distance on the space of
  probability measures.
\newblock arXiv:1712.01888, 2017.

\bibitem{LeGouicLoubes2017}
T.~Le~Gouic and J.~Loubes.
\newblock Existence and consistency of {Wasserstein} barycenters.
\newblock {\em Probab. Theory Relat. Fields}, 168:901--917, 2017.

\bibitem{LieroMielkeSavare-HellingerKantorovich-2015b}
M.~Liero, A.~Mielke, and G.~Savar\'e.
\newblock Optimal transport in competition with reaction: the
  {Hellinger--Kantorovich} distance and geodesic curves.
\newblock {\em SIAM J. Math. Anal.}, 48(4):2869--2911, 2016.

\bibitem{LieroMielkeSavare-HellingerKantorovich-2015a}
M.~Liero, A.~Mielke, and G.~Savar\'e.
\newblock Optimal entropy-transport problems and a new {Hellinger--Kantorovich}
  distance between positive measures.
\newblock {\em Inventiones mathematicae}, 211(3):969--1117, 2018.

\bibitem{MongeOT1781}
G.~Monge.
\newblock M\'emoire sur la th\`eorie des d\'eblais et des remblais.
\newblock {\em Histoire de l’Académie royale des sciences avec les
  m\'emoires de math\'ematique et de physique tir\'es des registres de cette
  Acad\'emie}, pages 666--705, 1781.

\bibitem{BrendanInfMarginal2013}
B.~Pass.
\newblock Optimal transportation with infinitely many marginals.
\newblock {\em J. Funct. Anal.}, 264(4):947--963, 2013.

\bibitem{PeyreCuturiCompOT}
G.~Peyré and M.~Cuturi.
\newblock Computational optimal transport.
\newblock {\em Foundations and Trends® in Machine Learning},
  11(5--6):355--607, 2019.

\bibitem{Pi68}
W.~P. Pierskalla.
\newblock The multidimensional assignment problem.
\newblock {\em Oper. Res.}, 16:422--431, 1968.

\bibitem{RachevRueschendorfVolBoth}
S.~T. Rachev and L.~R\"uschendorf.
\newblock {\em Mass Transportation Problems. Volume I: Theory, Volume II:
  Applications}.
\newblock Probability and its Applications. Springer, 1998.

\bibitem{Rockafellar-IntegralConvexFunctionals71}
R.~T. Rockafellar.
\newblock Integrals which are convex functionals. {II}.
\newblock {\em Pacific J. Math.}, 39(2):439--469, 1971.

\bibitem{Rockafellar1972Convex}
R.~Rockafellar.
\newblock {\em Convex analysis}, volume~28.
\newblock Princeton University Press, 2nd edition, 1972.

\bibitem{Santambrogio-OTAM}
F.~Santambrogio.
\newblock {\em Optimal Transport for Applied Mathematicians}, volume~87 of {\em
  Progress in Nonlinear Differential Equations and Their Applications}.
\newblock Birkh\"auser Boston, 2015.

\bibitem{SchmitzerScaling2019}
B.~Schmitzer.
\newblock Stabilized sparse scaling algorithms for entropy regularized
  transport problems.
\newblock {\em SIAM J. Sci. Comput.}, 41(3):A1443--A1481, 2019.

\bibitem{Villani-TOT2003}
C.~Villani.
\newblock {\em Topics in Optimal Transportation}, volume~58 of {\em Graduate
  Studies in Mathematics}.
\newblock American Mathematical Society, Providence, 2003.

\bibitem{Villani-OptimalTransport-09}
C.~Villani.
\newblock {\em Optimal Transport: Old and New}, volume 338 of {\em Grundlehren
  der mathematischen Wissenschaften}.
\newblock Springer, 2009.

\end{thebibliography}
\bibliographystyle{plain}

\end{document}